\documentclass[11pt]{amsart}

\usepackage[margin=1in]{geometry}
\usepackage[T1]{fontenc}
\usepackage{amsmath,amssymb,amsthm}
\usepackage{mathrsfs}
\usepackage{enumitem}
\usepackage[all,cmtip]{xy}
\usepackage{xcolor}
\usepackage{listings}
\usepackage{graphicx}
\usepackage{float}
\usepackage{bm}
\usepackage{tikz}
\usepackage{tikz-cd}
\usetikzlibrary{arrows.meta}
\usepackage{orcidlink}
\usepackage[round,authoryear]{natbib}
\usepackage{xurl}
\usepackage{hyperref}
\hypersetup{colorlinks=true,linkcolor=blue,citecolor=blue,urlcolor=blue}
\usepackage[capitalize,noabbrev]{cleveref}

\usepackage{etoolbox}
\makeatletter
\newcommand{\affiliations}[1]{\gdef\@affiliations{#1}}
\let\@affiliations\@empty
\patchcmd{\@maketitle}{\@setabstract}{%
  \ifx\@affiliations\@empty\else
    \begin{center}\footnotesize\@affiliations\end{center}\vspace{-0.5\baselineskip}%
  \fi
  \@setabstract}{}{\PackageWarning{main}{Could not place affiliations}}
\makeatother

\newcommand{\doi}[1]{\href{https://doi.org/#1}{doi:\nolinkurl{#1}}}

\newcommand{\cat}{\mathbf}
\newcommand{\id}{\mathrm{id}}
\newcommand{\VR}[1]{\textnormal{VR}(#1)}
\newcommand{\dVR}[1]{\overrightarrow{\textnormal{VR}}(#1)}
\newcommand{\colimit}{\textnormal{colim}}
\newcommand{\limit}{\textnormal{lim}}
\newcommand{\topint}[1]{I^{#1}}
\newcommand{\simp}[1]{|\Delta^{#1}|}

\definecolor{qqqqff}{rgb}{0,0,1}
\definecolor{qqccqq}{rgb}{0,0.8,0}
\definecolor{ffqqqq}{rgb}{1,0,0}

\newtheoremstyle{thmstyleone}
  {3pt}
  {3pt}
  {\itshape}
  {}
  {\bfseries}
  {.}
  { }
  {}

\newtheoremstyle{thmstyletwo}
  {3pt}
  {3pt}
  {\normalfont}
  {}
  {\bfseries}
  {.}
  { }
  {}

\newtheoremstyle{thmstylethree}
  {3pt}
  {3pt}
  {\normalfont}
  {}
  {\itshape}
  {.}
  { }
  {}

\theoremstyle{thmstyleone}%
\newtheorem{theorem}{Theorem}
\newtheorem{proposition}[theorem]{Proposition}%

\theoremstyle{thmstyletwo}%
\newtheorem{example}{Example}%
\newtheorem{remark}{Remark}%
\newtheorem{lemma}{Lemma}
\newtheorem{corollary}{Corollary}

\theoremstyle{thmstylethree}%
\newtheorem{definition}{Definition}%

\newcommand{\new}[1]{{#1}}

\begin{document}

\title[A McCord-type theorem for pseudotopological spaces and directed graphs]{A McCord-type theorem for pseudotopological spaces and directed graphs}

\author[N. Mili\'cevi\'c]{Nikola Mili\'cevi\'c$^{1}$\,\orcidlink{0000-0002-2896-7374}}
\author[N. A. Scoville]{Nicholas A. Scoville$^{2}$\,\orcidlink{0000-0002-4736-3605}}

\affiliations{%
$^{1}$Department of Mathematical Sciences, The University of Texas at El Paso, El Paso, TX 79968, USA\\
\texttt{nmilicevic@utep.edu}\\[4pt]
$^{2}$Department of Mathematics, Computer Science and Statistics, Ursinus College, Collegeville, PA 19426, USA\\
\texttt{nscoville@ursinus.edu}}

\thanks{The authors contributed equally to this work.}


\begin{abstract}
We construct weak homotopy equivalences between the geometric realizations of directed Vietoris-Rips complexes and their underlying directed graphs, seen as pseudotopological spaces. Pseudotopological spaces are a generalization of (\v{C}ech) closure spaces which in turn generalize topological spaces, but they also include graphs and directed graphs as full subcategories, making them a natural bridge that connects classical algebraic topology with the more applied side of topology. This weak homotopy equivalence implies that singular homology groups of finite directed graphs can be efficiently calculated from finite combinatorial structures, despite their associated chain groups being infinite dimensional. Along the way, we establish analogues of classical results such as the existence of a long exact sequence for homotopy groups of pairs of pseudotopological spaces and that a weak homotopy equivalence induces isomorphisms for homology groups.  This work is similar to the work of McCord for finite topological spaces but in the context of pseudotopological spaces. Our results also give a novel approach for studying (higher) homotopy groups of discrete mathematical structures such as (directed) graphs or digital images.
\end{abstract}

\maketitle

\section{Introduction}\label{sec1}

Singular homology groups of topological spaces are difficult to compute due to the chain groups having infinite dimension, even for spaces of finite cardinality. Computing higher homotopy groups of topological spaces presents an even greater challenge. An elegant classical solution is to replace the space with a CW complex (or a simplicial complex) that is (weak) homotopy equivalent to the original space. However, for finite topological spaces, their non-Hausdorff nature generally disallows having continuous maps to a Hausdorff space like a CW complex that are not constant. Thus, finite spaces are generally not homotopy equivalent to a CW complex. Therefore, to effectively compute these algebraic invariants of finite spaces the best hope is to find a weak homotopy equivalence to a CW complex, as weak homotopy equivalences will also induce isomorphisms on singular homology. This is precisely what McCord did with the following celebrated result:

\begin{theorem}{\cite[Theorem 1]{McCord1966}}
\begin{enumerate}
    \item For each finite topological space $X$ there exists a finite simplicial complex $K$ and a weak homotopy equivalence $f\colon|K|\to X$.
    \item For each finite simplicial complex $K$ there exists a finite topological space $X$ and a weak homotopy equivalence $f\colon|K|\to X$.
\end{enumerate}
\end{theorem}

An immediate consequence of this result was the surprising realization that finite topological spaces have more interesting singular homology and higher homotopy groups than one might imagine; homology and homotopy groups of any finite simplicial complex can be realized as singular homology groups and homotopy groups, respectively, of a finite topological space. 

Finite topological spaces are equivalent to a preorder on a set \new{\citep{Alexandroff:1937}}, i.e., a set with a transitive and reflexive relation $(X,E)$. However, in applied topology data typically does not come with the transitivity assumption; we only have a set with a reflexive relation $(X,E)$, i.e., a simple directed graph, or digraph for short, with a loop at each vertex \new{\citep{LuptonScoville2022Digital,Lupton2021,ChihScull2021Graphs,HobbsMartinoScull2023Spider}}.

The category of topological spaces, $\cat{Top}$, and the category of digraphs, $\cat{DiGph}$, are both full subcategories of \v{C}ech closure spaces, $\cat{Cl}$ \new{\citep{Dikranjan1995}}. Closure spaces are a generalization of topological spaces where one does not require the closure operator to be idempotent \new{\citep{cech1966topologicalspaces}}. This suggests we can develop a common homotopy and homology theory for seemingly very distinct objects: the classical topological spaces and the objects studied in applied topology \new{\citep{Bubenik2024,rieser2021cechclosurespaces,Bubenik2021}}. However, just like in the case of finite topological spaces, the singular homology groups of finite digraphs are difficult to compute on their own due to the infinite dimensional chain groups. By leveraging analogues of classical long exact sequences like the Mayer-Vietoris sequence, some examples of singular homology groups of undirected graphs can be computed \new{\citep{Milicevic2025}}. Using a covering space theory of closure spaces, \new{\citet{rieser2021cechclosurespaces}} computed examples of fundamental groups of finite closure spaces. However, so far there is no analogue of McCord's result for computing singular homology and homotopy groups of arbitrary finite digraphs more effectively.

In this paper, we show that McCord's result (part (1)) can be extended to finite digraphs, which are equivalent to finite closure spaces, i.e., we can drop the transitivity requirement for $(X,E)$. Throughout, when referring to a digraph $(X,E)$ in a topological setting, we mean the closure space along with closure operator. By abuse of notation, we will use $(X,E)$ for $X$ viewed as a digraph and $X$ viewed as a closure space. 

Part (2) of McCord's theorem is immediately true in $\cat{Cl}$, as every topological space is a closure space which is why we are only interested in part (1). In particular, in part (2), given an abstract simplicial complex, McCord showed that its geometric realization is weak homotopy equivalent to its underlying set of faces with the partial order given by inclusion. 

Given a finite digraph $(X,E)$, we consider its directed Vietoris-Rips complex $\dVR{X,E}$ (also known as the directed clique complex or directed flag complex)  which is the collection of all subsets $\sigma\subset X$ on which there is an ordering $\sigma=\{v_0,\dots v_n\}$ such that $v_iEv_j$ whenever $i<j$. Note that $E$ potentially having bidirected edges implies that there might be more than one such ordering of vertices in $\sigma$. If $E$ is symmetric on all of $X$, this definition recovers the usual Vietoris-Rips complex $\VR{X,E}$ (also called the clique complex or the flag complex) on $(X,E)$. \citet{McCord1966} showed that a finite  $T_0$ topological space (a partial order $(X,E)$) is weakly homotopy equivalent to the geometric realization of its order complex, which consists of all chains in the preorder $(X,E)$ (part (1) of the theorem). As the directed Vietoris-Rips complex is an obvious generalization of the order complex (when transitivity of $E$ is not certain), the directed Vietoris-Rips complex is the natural candidate for us to consider for our main result. In other words, our main result is the following:

\begin{theorem}[{\cref{theorem:vr_weak_homotopy_equivalence}}]
    For each finite digraph $(X,E)$ there exists a finite abstract simplicial complex $\dVR{X,E}$ and a weak homotopy equivalence $f_X\colon |\dVR{X,E}|\to (X,E)$ which is a morphism in $\cat{Cl}$. 
\end{theorem}

This result allows us to more effectively compute higher pseudotopological homotopy groups of finite digraphs (or equivalently finite closure spaces) by using combinatorial structures such as simplicial complexes. 

In order to prove our results, we extended classical results from algebraic topology from $\cat{Top}$ to the larger setting of $\cat{Cl}$. This was not entirely trivial, since $\cat{Top}$ is a reflective but not coreflective subcategory of $\cat{Cl}$ and thus colimits of diagrams of topological spaces are not ``well behaved" in $\cat{Cl}$. Furthermore, for the sake of generality and a very technical reason (\cref{remark:no_mapping_cylinder_in_cl}), we develop all this theory for an even larger category of spaces, the category of pseudotopological spaces, $\cat{PsTop}$ \new{\citep{Preuss2002}}. However, finite pseudotopological spaces coincide with finite closure spaces, for which our main result is intended. As an example we prove the following important result in $\cat{PsTop}$.

\begin{theorem}[\cref{theorem:n_equivalences}]
    Let $f\colon(X;X_1,X_2)\to (Y;Y_1,Y_2)$ be a map of excisive triads in $\cat{PsTop}$ such that $f\colon (X_i,X_1\cap X_2)\to (Y_i,Y_1\cap Y_2)$ is an $n$-equivalence for $i=1,2$. Then $f\colon (X,X_i)\to (Y,Y_i)$ is an $n$-equivalence for $i=1,2$.
\end{theorem} 
 
Furthermore, we prove the following theorem \new{which, when} combined with Theorem 2 above, allows us to compute singular (co)homology groups of finite digraphs more effectively.

\begin{theorem}[\cref{theorem:weak_homotopy_iso_on_homology}]
    A weak homotopy equivalence $f\colon X\to Y$ of pseudotopological spaces induces isomorphisms
    $f_* \colon H_n(X;G)\to H_n(Y;G)$ and $f^*\colon H^n(Y;G)\to H^n(X;G)$ for all $n$ and all coefficient groups $G$.
\end{theorem}

In summary, this manuscript further develops singular homology and higher homotopy groups in $\cat{PsTop}$. Our results apply to the important full subcategory digraphs, $\cat{DiGph}$, which is useful for applications. These homology and homotopy groups are a natural extension of homology and homotopy groups of topological spaces. Our main result allows us to compute them by replacing the digraph with its directed Vietoris-Rips complex. 

This manuscript is structured as follows. In \cref{section:discussion}\new{,} we discuss the implications of this work with several examples in applied topology. In \cref{section:background} we recall the necessary background material on pseudotopological and closure spaces. In \cref{section:homotopy_and_homology} we define higher homotopy groups and singular homology groups for pseudotopological spaces. We show these groups satisfy analogues of classical results from algebraic topology. For the sake of readability, some of the proofs in \cref{section:homotopy_and_homology} are relegated to the Appendix. In \cref{section:main_results} we prove our main result.  

\subsection*{Related works}
The present work sits at the intersection of several active research programs, and it is worth orienting the reader with respect to each of them.

\paragraph{Algebraic topology for closure spaces.} The extension of classical algebraic topology to closure spaces was pioneered, to our knowledge, by \citet{Demaria1984}. Since then, the field has developed considerably: singular homology theories and homotopy theories for closure spaces have been constructed and studied in \citet{Bubenik2021,Bubenik2024,Milicevic2025},  model structures on pseudotopological spaces have been established by \citet{rieser2022cofibration} and \citet{Ebel2023}, and the theory of CW complexes in the category $\cat{Cl}$ has been developed by \citet{Bubenik2025}. Fundamental groups and covering space theory for closure spaces were treated by \citet{rieser2021cechclosurespaces}. We note that some of the results of the present paper concerning the weak homotopy type of pseudotopological spaces have been arrived at independently, and by quite different methods, in \citet{TrevinoMarroquin2024graphs}.

\paragraph{Homotopy and homology theories for graphs.} A significant portion of the literature has been devoted to constructing combinatorial analogues of classical homotopy and homology theories directly for graphs, without passing through an ambient topological category. Among the earliest and most influential of these is $A$-theory, a discrete homotopy theory for simple graphs developed in \citet{Barcelo2001,Barcelo2005,Babson2006}, complemented by a cubical singular homology theory in \citet{Barcelo2014}. A separate homotopy theory for simple graphs, the $\times$-homotopy theory, was developed by \citet{Dochtermann2009}, and it was extended to directed graphs in \citet{Dochtermann2023}. For directed graphs, a homotopy theory was constructed by \citet{Grigoryan2014}. A particularly illuminating result, which provides important context for our own approach, is that of \citet{Bubenik2024}, which showed that nearly all of these graph-theoretic homotopy and homology theories can be recovered as special cases of a single uniform framework, obtained by varying the choice of product and interval object in the category of closure spaces. This unification lends further motivation to the \new{pseudotopological} space perspective adopted in the present paper.

\paragraph{Digital topology.}
A special class of graphs, in which the underlying set is a finite subset of the integer lattice $\mathbb{Z}^n$, forms the subject of digital topology and digital homotopy theory \new{\citep{Boxer1999}}. These theories were originally motivated by questions in image processing and computer vision, and they provide another instance of the kind of combinatorial homotopy theory that the closure space framework is well suited to unify. A development of homotopy-theoretic foundations in the digital setting, including function spaces, cofibrations, and the homotopy extension property, was carried out \new{by} \citet{Lupton2021}. A digital fundamental group, denoted $\pi_1^d$, was defined \new{by} \citet{LuptonOpreaScoville2021Fundamental}. The digital fundamental group, in turn, was shown \new{by} \citet{LuptonScoville2022Digital} to be isomorphic to the edge group of the clique complex of the underlying graph, which in turn is isomorphic to the ordinary fundamental group of its geometric realization.  A digital second homotopy group, denoted $\pi_2^d$, was constructed \new{by} \citet{Lupton2024}; notably, the authors show that $\pi_2^d(S^2)\cong \mathbb{Z}$ where $S^2$ is a 6-point digital model of a 2-sphere.  The question of which notion of homotopy is appropriate in the digital setting was investigated in \new{\citet{Staecker2021}} which also studied the interaction between the homotopy relations and several digital homology theories.  

\paragraph{Directed Vietoris-Rips simplicial complex.}
The Vietoris-Rips complex, also known as the clique or flag complex, was introduced by \citet{vietoris1927hoheren} as a means of associating a simplicial complex to a graph or metric space. Due to its simple definition from pairwise distances and strong theoretical guarantees, the Vietoris–Rips complex has become a central construction in topological data analysis \new{\citep{Carlsson2009}}, with efficient computation made possible by tools such as Ripser \new{\citep{Bauer2021}}.
The extension of this construction to directed graphs is more recent: to our knowledge, the first definition of a ``directed clique complex" appears in \new{\citet{masulli2016topology}}, though that definition excludes digraphs with bidirected edges. The definition of ``directed Vietoris-Rips complex" we adopt in the present work removes this restriction, and is thus applicable to the most general class of digraphs.  The ``directed flag complex" \new{\citep{Ltgehetmann2020}} is defined for digraphs with potential bidirected edges. However, its definition is different than our directed Vietoris-Rips complex. Recall that $\sigma$ is a simplex in the directed Vietoris-Rips complex if there is an ordering $\sigma=\{v_0,\dots,v_n\}$ such that $v_iEv_j$ whenever $i<j$. By contrast, the directed flag complex considers all such different orderings to be distinct simplices themselves and thus its homotopy type \new{may} be different as well. Therefore to avoid confusing the reader familiar with this related work, and to orient our own work towards future applications in topological data analysis we have chosen the name directed Vietoris-Rips complex for our central construction.

\section{Contributions to Applied Topology}
\label{section:discussion}

Before our preliminaries and main results, we discuss the implications of this manuscript with several examples from applied topology.

\begin{example}
\citet{Milicevic2025} developed a Mayer–Vietoris long exact sequences and other tools to compute singular homology for closure spaces.  In particular, the homology of the closure space of roots of unity $(\mathbb{Z}_n,c_m)$, where $c_m$ is the m-nearest neighbors closure on the integers modulo $n$, is studied.  One question that was left open was to compute the homology  of $(\mathbb{Z}_6,c_2)$.  Using  \cref{theorem:vr_weak_homotopy_equivalence,theorem:weak_homotopy_iso_on_homology}, we  see that $$H_i(\mathbb{Z}_6, c_2)\cong H_i(|\mathrm{VR}(\mathbb{Z}_6, c_2)|)\cong \begin{cases}
			\mathbb{Z}, & \text{if $i=$ 0, 2}\\
            0, & \text{otherwise}
		 \end{cases},$$   
since one computes that $|\VR{\mathbb{Z}_6,c_2}|$ is the regular octahedron. More generally we can now compute the singular homology of the closure space $(\mathbb{Z}_n,c_m)$ for any values of $n$ and $m$ by considering the associated Vietoris-Rips complex whose homotopy types were computed to some extent \new{by} \citet{Adamaszek2013-co}.
\end{example}

\begin{example} Digital homotopy theory studies topological properties of digital images \new{\citep{Lupton2021}}.  A digital image $X$ is a finite subset of $\mathbb{Z}^n$ along with an adjacency relation where two point $\vec{x},\vec{y}$ are adjacent if $|x_i-y_i|\leq 1$ for every $1\leq i \leq n$. A digital image may thus be thought of as a graph whose vertices are points in $\mathbb{Z}^n$ with edge relations constrained by the lattice. Both a digital $\pi_1^{d}$ \new{\citep{LuptonOpreaScoville2021Fundamental}} and digital $\pi_2^{d}$ \new{\citep{Lupton2024}} have been studied, and it has been shown that $\pi_1^{d}(C_n)\cong \pi^{d}_2 (S^2)\cong \mathbb{Z}$ where $C_n$ is any cycle of length $\geq 4$ and $S^2$ is the digital image given by $\{\pm \mathbf{e_1}, \pm \mathbf{e_2}, \pm \mathbf{e_3}\}$ where $\mathbf{e_i}$ is the $i^{th}$ standard basis vector in $\mathbb{Z}^3$. More generally, \citet{LuptonScoville2022Digital} showed that the digital fundamental group of a digital image is isomorphic to the edge group of its clique complex.  This latter group is well-known to be isomorphic to the fundamental group of the geometric realization of the clique complex. Hence the digital fundamental group of a digital image $X$ (defined purely in terms of digital homotopy classes of digital loops) turns out to be computed by simply computing the classical fundamental group of the geometric realization of $X$. 

If we instead view a digital image as a closure space, we may associate an alternative version of homotopy groups as well as homology groups to digital images, namely, the classical homotopy and homology groups of the closure space.  \cref{theorem:vr_weak_homotopy_equivalence} may be used to compute these alternative homotopy and homology groups by constructing the Vietoris-Rips complex of the graph. The result mentioned in the above paragraph then implies that the alternate fundamental group of $X$ coincides with its digital fundamental group.  Whether or not the higher digital homotopy groups coincide with the higher homotopy groups induced from $X$ as a closure space is unknown. In particular, the digital image $S^2$ described above as a closure space is homeomorphic to $(\mathbb{Z}_6,c_2)$ whose Vietoris-Rips complex is the regular octahedron. Thus, in this case we also have an isomorphism of digital homotopy groups and the classical homotopy groups in degree 2 as well. We conjecture that we have isomorphisms of all higher homotopy groups, for all digital images.  
\end{example}

\begin{example}
    The Vietoris-Rips and the \v{C}ech complex are two commmonly used constructions in topological data analysis \new{\citep{Ghrist2018,Rabadn2019}}. More specifically if $(X,d)$ is a metric space and $P\subset X$ is a finite point sample, the (ambient) \v{C}ech complex at scale $r\ge 0$ of $P$ is the abstract simplicial complex defined by 
    \[\check{C}_{r}(P)=\{\sigma\subset P\,|\, \bigcap_{v\in \sigma}B_r(v)\neq \varnothing\},\]
    where $B_r(v)$ is the closed ball of radius $r$ in $(X,d)$ centered at $v$.
    The Vietoris-Rips complex at scale $r$ of $P$ is obtained by constructing a graph on $P$, $(P,E_r)$ by declaring that $xE_ry$ if and only if $d(x,y)\le r$, and then taking the clique complex of this graph. Computing the \v{C}ech complex is hard due to the complicated combinatorics and thus the Vietoris-Rips complex is a preferred choice, as one only needs to consider pairwise distances in the construction. 
    The famous nerve theorem states that for a paracompact Hausdorff topological space $X$, the \v{C}ech nerve of a ``good" cover is homotopy equivalent to $X$. If for example $(X,d)$ is the Euclidean space $\mathbb{R}^m$, then the \v{C}ech complex described above is precisely the nerve of the cover by closed balls of radius $r$, centered at points of $P$. This cover is ``good", meaning that all finite intersections of balls are contractible, and thus the nerve theorem says that $\bigcup_{p\in P}B_r(p)$ is homotopy equivalent to $|\check{C}_r(p)|$. Therefore, even though the \v{C}ech complex is hard to compute, it is considered desirable as it preserves the homotopy type of $\bigcup_{p\in P}B_r(p)$. On the other hand, the Vietoris-Rips complex in general will not have the same homotopy type as $\bigcup_{p\in P}B_r(p)$ \new{\citep{Adamaszek2017}}. Our work instead shows that the Vietoris-Rips complex always preserves the weak homotopy type of the underlying graph $(P,E_r)$.
\end{example} 

\begin{example}
    A-homotopy theory (or A-theory) is a discrete version of homotopy developed for graphs \new{by} \citet{Barcelo2001}. In \new{\citet{Grigoryan2014}} A-theory was extended to digraphs. This homotopy theory is compatible with path homology of digraphs also developed in \new{\citet{Grigoryan2014}}, in the sense that A-theory \new{homotopy equivalent} digraphs have the same path homology. Path homology has since become widely used when studying digraphs due to its many nice properties such as analogues of Eilenberg-Steenrod axioms and K\"unneth theorems \new{\citep{Grigoryan2017,Grigoryan2018}}. The natural question to ask is how our homology and homotopy of digraphs developed in this paper compares to path homology and A-theory. The simplest answer is that they are fundamentally different theories. For example the directed Vietoris Rips complex of the ``square" digraph in \cref{fig:path_hom_examples} is a square which has nontrivial homology. However its path homology is known to be trivial. Furthermore, the digraph on the right in \cref{fig:path_hom_examples} is an example constructed by Gabor Lippner and Paul Horn in 2012 of a digraph that has nontrivial path homology groups in infinitely many dimensions. This does not happen in our homology theory as its directed Vietoris-Rips complex is a finite simplicial complex.

\begin{figure}[H]
\centering
\begin{tikzpicture}[
    every node/.style={circle, fill=blue, inner sep=3pt},
    arrow/.style={-{Stealth[scale=1.2]}, blue, thick},
    garrow/.style={-{Stealth[scale=1.2]}, green!60!black, thick}
]

\node (t1) at (0, 2)  {};
\node (t2) at (2, 2)  {};

\node (b1) at (0, 0)  {};
\node (b2) at (2, 0)  {};

\node[draw=none, fill=none, above=4pt] at (1,-1.5)  {$a$};

\draw[arrow] (t1) -- (t2);
\draw[arrow] (b1) -- (t1);
\draw[arrow] (b2) -- (t2);
\draw[arrow] (b1) -- (b2);
\end{tikzpicture}
\quad\quad\quad\quad
\begin{tikzpicture}[
    every node/.style={circle, fill=blue, inner sep=3pt},
    arrow/.style={-{Stealth[scale=1.2]}, blue, thick},
    garrow/.style={-{Stealth[scale=1.2]}, green!60!black, thick}
]

\node (t1) at (0, 2)  {};
\node (t2) at (2, 2)  {};
\node (t3) at (4, 2)  {};
\node (t4) at (6, 2)  {};
\node (t5) at (8, 2)  {};

\node (b1) at (0, 0)  {};
\node (b2) at (2, 0)  {};
\node (b3) at (4, 0)  {};
\node (b4) at (6, 0)  {};
\node (b5) at (8, 0)  {};



\draw[arrow] (t1) -- (t2);
\draw[arrow] (t2) -- (t3);
\draw[arrow] (t3) -- (t4);
\draw[arrow] (t4) -- (t5);

\draw[arrow] (b1) -- (b2);
\draw[arrow] (b2) -- (b3);
\draw[arrow] (b3) -- (b4);
\draw[arrow] (b4) -- (b5);

\draw[arrow] (t1) -- (b2);
\draw[arrow] (t2) -- (b3);
\draw[arrow] (t3) -- (b4);
\draw[arrow] (t4) -- (b5);

\draw[arrow] (b1) -- (t2);
\draw[arrow] (b2) -- (t3);
\draw[arrow] (b3) -- (t4);
\draw[arrow] (b4) -- (t5);

\draw[garrow] (t4) to (b2);
\draw[garrow] (b4) to (t2);
\draw[garrow] (b4) to[out=-160, in=-20, looseness=1.5] (b2);
\draw[garrow] (t4) to[out=160, in=20, looseness=1.5] (t2);
\node[draw=none, fill=none, above=4pt] at (4,-1.5)  {$b$};
\end{tikzpicture}
\caption{a) A digraph whose path homology is trivial, b) a digraph whose path homology has infinite homological dimension.}
\label{fig:path_hom_examples}
\end{figure}
Analogues of Eilenberg-Steenrod and K\"unneth theorems for the singular homology of closure spaces considered in this work were proved in \new{\citet{Bubenik2021}}. Therefore our work provides an alternative to path homology for studying digraphs. Our homology theory is shown to have nice properties, and it appears more computable at a first glance, as there will be no examples of finite digraphs having nontrivial homology groups in infinitely many degrees.

\end{example}

\section{Background}
\label{section:background}

Here we give some of the necessary definitions and background results for topological, closure, and pseudotopological spaces that will be used throughout this paper. The proofs of some lemmas are relegated to the Appendix. We refer to \new{\citet[Chapter 1]{Beattie2002}} and \new{\citet{Preuss2002}} for background on pseudotopological spaces. We refer to \new{\citet{cech1966topologicalspaces}} for background on closure spaces. 

\subsection{Pseudotopological spaces}
We start by defining convergence spaces. For an additional background on convergence spaces see \new{\citet{Dolecki2016,Dolecki2024}}.

\begin{definition} Let $X$ be a non-empty set.  A \textbf{filter} on $X$ is a collection $F$ of non-empty subsets of $X$ satisfying 
\begin{enumerate}[left=0pt]
    \item $X\in F$ and $\varnothing\not \in F$.
    \item If $A,B\in F$, then $A\cap B\in F$.
    \item If $A\in F$ with $A\subset B \subset X$, then $B\in F$. 
\end{enumerate}
For any set $X$, let $\mathbb{F}(X)$ denote the set of filters on $X$.

\end{definition}

\begin{definition}
 A \textbf{convergence space} is a pair $(X,\Lambda)$ where $\Lambda \subset \mathbb{F}(X)\times X$ is a relation between $X$ and the set of filters on $X$ satisfying
 \begin{enumerate}[left=0pt]
     \item If $(\lambda, x)\in \Lambda$ and $\lambda\subset \lambda'$, then $(\lambda',x)\in \Lambda$ (upward closure).
     \item $(\dot{x},x)\in \Lambda$, where $\dot{x}$ denotes the \textbf{filter generated by $x$}, i.e. $\dot{x}=\{A\subseteq X: x\in A\}$.
 \end{enumerate}
 Any such relation $\Lambda$ is called a \textbf{convergence structure} or \textbf{convergence} on $X$. If $(\lambda,x)\in \Lambda$, we write $\lambda \to x$ and  say that $\lambda$ \textbf{converges} to $x$ or that $x$ is a \textbf{limit point} of $\lambda$. Note that there is no requirement that when $\lambda\to x$,  every element of $\lambda$ must contain $x$. If in addition we have that 
 \begin{enumerate}[left=0pt]
     \item[3)] Whenever $(\lambda,x),(\lambda',x)\in \Lambda$, we have $(\lambda\cap \lambda',x)\in \Lambda$,
 \end{enumerate}
 we say that $(X,\Lambda)$ is a \textbf{limit space}. 
\end{definition}

\begin{remark}
    Some authors (for example \citet{Beattie2002}) use the more general term ``convergence space” to refer to what we have defined as a
limit space above.
\end{remark}

\begin{definition}
An \textbf{ultrafilter} $\gamma$ on a set $X$ is a filter such that for all $A\subset X$, either $A\in \gamma$ or $X\setminus A\in \gamma$.  A limit space $(X,\Lambda)$ is called a \textbf{pseudotopological space}, (also called Choquet spaces \new{\citep{Beattie2002}}), if $\lambda\to x$ whenever every ultrafilter $\gamma$ finer than $\lambda$ also converges to $x$.
\end{definition}

Ultrafilters are maximal among all filters as the following result shows.

\begin{lemma} An  ultrafilter is not contained in any other filter.
\end{lemma}

\begin{proof}
    Let $\gamma$ be an ultrafilter on $X$ and suppose that $\lambda$ is a filter such that $\gamma\subset \lambda$. Let $A\in \lambda$. Since $\gamma$ is an ultrafilter, either $A\in \gamma$ or $X\setminus A\in \gamma$. Suppose that $X\setminus A\in \gamma$. Then $X\setminus A\in \lambda$. Since $\lambda$ is a filter we have $(X\setminus A)\cap A=\varnothing\in \lambda$, which is a contradiction. Therefore $A\in \gamma$.
\end{proof}

\begin{definition}
Let $(X,\Lambda_X)$ and $(Y,\Lambda_Y)$ be convergence spaces, and let $f\colon X\to Y$ be a map.  Then $f$ is \textbf{continuous} if whenever $\lambda \to x\in \Lambda_X$, we have $f(\lambda)\to f(x)\in \Lambda_Y$.  Here $f(\lambda)$ is the filter in $Y$ generated by the collection $\{f(U)\subset Y : U\in \lambda\}$ of subsets of $Y$. Continuous maps between limit or pseudotopological spaces are maps that are continuous with respect to the underlying convergence structures.
\end{definition}

We denote the category of pseudotopological spaces and continuous maps between pseudotopological spaces by $\cat{PsTop}$.

Given two convergence structures $\Lambda$ and $\Lambda'$ on $X$, we say that $\Lambda$ is \textbf{finer} than $\Lambda'$ and $\Lambda'$ is \textbf{coarser} than $\Lambda$ if for all $x\in X$, $\{\lambda\,|\, (\lambda,x)\in \Lambda\}\subset \{\lambda'\,|\, (\lambda',x)\in \Lambda'\}$. In other words for all $x$, $\Lambda$ has fewer convergent filters to $x$ than $\Lambda'$. This terminology may seem counterintuitive at first, \new{but the convention reflects the fact that $\Lambda$ is finer than $\Lambda'$ if and only if the identity function $\id_X\colon(X,\Lambda)\to (X,\Lambda')$ is continuous. }

\begin{definition}
    \label{def:subspace_convergence}
    Let $(X,\Lambda)$ be a convergence space and let $A\subset X$. The \textbf{subspace convergence} on $A$ is the convergence structure $\Lambda_A$ on $A$ defined by $\lambda\to x$ in $\Lambda_A$ if and only if the filter generated by $\lambda$ on $X$ converges to $x$ in $\Lambda$. We say that $(A,\Lambda_A)$ is a \textbf{subspace} of $(X,\Lambda)$. The inclusion map $(A,\Lambda_A)\hookrightarrow (X,\Lambda)$ is clearly continuous.
\end{definition}

\begin{example}
    The finest convergence structure on $X$ is called the \textbf{discrete convergence} where for every $x\in X$ we have $\dot{x}\to x$ as the only filter that converges to $x$. If $(X,\Lambda)$ is such that $\Lambda$ is the discrete convergence, then any function $f\colon(X,\Lambda)\to (Y,\Lambda_Y)$ is continuous for any convergence space $(Y,\Lambda_Y)$. That is because for all $x\in X$ we have that $\dot{x}$ is the only filter converging to $x$ and thus $f(\dot{x})=\dot{f(x)}\to x$ in $\Lambda_Y$.
\end{example}

\begin{example}
    The coarsest convergence structure on $X$ is called the \textbf{indiscrete convergence} where for every $x\in X$, we declare that every $\lambda\in 2^{\mathbb{F}(X)}$ converges to $x$, i.e., all the filters on $X$ converge to $x$. If $(X,\Lambda)$ is such that $\Lambda$ is the indiscrete convergence, then any function $f\colon (Y,\Lambda_Y)\to (X,\Lambda)$ is continuous for any convergence space $(Y,\Lambda_Y)$. That is because if $\lambda\to y$ in $\Lambda_Y$ we have $f(\lambda)\to f(y)$ by definition as $\Lambda$ is indiscrete.
\end{example}

\begin{definition}
    \label{def:neighborhood_filter}
    Let $(X,\Lambda)$ be a convergence space. For all $x\in X$, the filter 
    \[\mathcal{U}(x)=\bigcap\{\lambda\,|\, \lambda\to x\}\]
    is called the \textbf{neighborhood filter} of $x$ and its elements are called the \textbf{neighborhoods} of $x$. A set $A\subset X$ is called \textbf{open} if it is a neighborhood of each of its points. For each $A\subset X$ the set 
    \[a(A)=\{x\in X\,|\, \exists \lambda\to x, A\in \lambda\}\]
    is called the \textbf{adherence} of $A$. Call $A\subset X$ \textbf{closed} if $a(A)=A$.
\end{definition}

The following is straightforward.

\begin{lemma}
\label{lemma:cont_maps_in_pstop_are_cont_in_cl}
    Let $f\colon (X,\Lambda_X)\to (Y,\Lambda_Y)$ be a continuous map of pseudotopological spaces. Then $f(a_X(A))\subset a_Y(f(A))$ for all $A\subset X$, where $a_X$ and $a_Y$ are the adherences in $(X,\Lambda_X)$ and $(Y,\Lambda_Y)$, respectively.
\end{lemma}

\begin{proof}
    Let $A\subset X$ and let $y\in f(a_X(A))$. Then there exists an $x\in a_X(A)$ such that $f(x)=y$. Since $x\in a_X(A)$, there exists $\lambda\to x\in \Lambda_X$ such that $A\in \lambda$. Since $f$ is continuous we have that $f(\lambda)\to f(x)\in \Lambda_Y$. Since $f(A)\in f(\lambda)$ we have that $f(x)=y\in a_Y(f(A))$ by definition of $a_Y$. Therefore $f(a_X(A))\subset a_Y(f(A))$.
\end{proof}

\begin{example}
\label{example:discrete_pstop}
    Let $(X,\Lambda)$ be such that $\Lambda$ is the discrete convergence. Then $(X,\Lambda)$ is a pseudotopological space as $\dot{x}$ is an ultrafilter for all $x\in X$ and $\dot{x}$ is the only filter converging to $x$ by definition. Thus, we say $(X,\Lambda)$ is a \textbf{discrete pseudotopological space}. Furthermore, for any pseudotopological space $(X',\Lambda')$ any function $f\colon(X,\Lambda)\to (X',\Lambda')$ is continuous in $\cat{PsTop}$.
\end{example}

\begin{example}
\label{example:indiscrete_pstop}
    Let $(X,\Lambda)$ be such that $\Lambda$ is the indiscrete convergence. Then $(X,\Lambda)$ is a pseudotopological space for all $x\in X$ every filter on $X$ (and thus ultrafilter) converges to $x$ by definition. Thus, we say $(X,\Lambda)$ is an \textbf{indiscrete pseudotopological space}. Furthermore, for any pseudotopological space $(X',\Lambda')$ any function $f\colon(X',\Lambda')\to (X,\Lambda)$ is continuous in $\cat{PsTop}$.
\end{example}

\subsection{Closure spaces and topological spaces}\label{sec: closure spaces}

A (\v{C}ech) \textbf{closure space} is a pair $(X,c)$ where $X$ is a set and $c\colon 2^X\to 2^X$ is a function such that 
\begin{enumerate}[left=0pt]
    \item $c(\varnothing)=\varnothing$.
    \item $A\subset c(A)$ for all $A\subset X$.
    \item $c(A\cup B)=c(A)\cup c(B)$.
\end{enumerate}

If $c(A)=c(c(A))$ for all $A\subset X$ we say $(X,c)$ is a \textbf{topological space}.  This alternative definition of topological spaces is equivalent to the standard one, in terms of declaring what subsets are called open \new{\citep{cech1966topologicalspaces}}. A \textbf{continuous map} of closure spaces $f\colon (X,c)\to (Y,d)$ is a function $f\colon X\to Y$ such that $f(c(A))\subset d(f(A))$. We denote the category of closure spaces and continuous maps of closure spaces by $\cat{Cl}$. We denote the full subcategory of closure spaces, i.e, the topological spaces, by $\cat{Top}$. 

\begin{definition}
    \label{def:subspace_closure}
    Let $(X,c)$ be a closure space and let $A\subset X$.
    The \textbf{subspace closure operation} on $A$ is $c_A(B):=A\cap c(B)$ for all $B\subset A$. We say that $(A,c_A)$ is a \textbf{subspace} of $(X,c)$. The inclusion map $(A,c_A)\hookrightarrow (X,c)$ is clearly continuous.
\end{definition}

\begin{theorem}{\cite[Theorem 17.A.13]{cech1966topologicalspaces}}
\label{theorem:range_restriction_cont}
    A map of closure spaces $f\colon X\to Y$ is continuous iff the range restriction $f\colon X\to f(X)$ is continuous.
\end{theorem}

\begin{definition}
    \label{def:interior_cover}
    Let $(X,c)$ be a closure space. The \textbf{interior operator} of $(X,c)$ is a function $i\colon 2^X\to 2^X$ defined by $i(A)=X\setminus c(X\setminus A)$. A set $U$ is a \textbf{neighbhorhood} of $A$ if $A\subset i(U)$. We say a collection $\mathcal{U}$ of subsets of $X$ is an \textbf{interior cover} of $(X,c)$ if $X=\bigcup_{U\in \mathcal{U}}i(U)$. 
\end{definition}

\begin{theorem}{\cite[Theorem 16.A.4 and Corollary 16.A.5]{cech1966topologicalspaces}}
\label{theorem:continuity_in_terms_of_nbhds}
   A function $f\colon (X,c)\to (Y,d)$ between closure spaces is continuous at $x \in X$ if and only if for every
    neighborhood $V \subset Y$ of $f(x)$, the inverse image $f^{-1}(V)$ is a neighborhood of $x$.   Equivalently, $f$ is continuous at $x$ if and only if for each neighborhood $V \subset Y$ of $f(x)$, there exists a neighborhood $U \subset X$ of $x$ such that $f(U) \subset V$. 
 \end{theorem}

We consider each closure space $(X,c)$ as a pseudotopological space by declaring that $\mathcal{U}(x)\to x$ for all $x\in X$, where $\mathcal{U}(x)$ is the filter consisting of all neighborhoods of $x$ in $(X,c)$. An alternative definition of a closure space is that of a pseudotopological space for which the neighborhood filter (\cref{def:neighborhood_filter}) at each point always converges to the point. In this case one sets the closure operator to be the adherence operator of the pseudotopological space. We will present this alternative definition as a proposition here. For a proof, see \new{\citet[Chapter 2]{Preuss2002}}.

\begin{proposition}
\label{prop:alternative_closure_space_definition}
\label{prop:alternative_char_of_closure_spaces}
    A pseudotopological space $(X,\Lambda)$ is a closure space if $\mathcal{U}(x)\to x$ (\cref{def:neighborhood_filter}) for all $x\in X$. 
\end{proposition}

\begin{proposition}\cite[Remark 2.3.1.2]{Preuss2002}
$\cat{Cl}$ is a full subcategory of $\cat{PsTop}$.
\end{proposition}

\begin{example}
\label{example:discrete_cl}
   Let $(X,\Lambda)$ be a discrete pseudotopological space. Then for all $x\in X$, $\mathcal{U}(x)=\{\dot{x}\}\to x$ and thus the neighborhood filter converges to $x$. Thus $(X,\Lambda)$ is a closure space. One also computes that $a(A)=A$ for all $A\subset X$, where $a$ is the adherence for $\Lambda$ or the underlying closure operator. Therefore, the closure operator $a$ is idempotent and thus $(X,a)$ is a topological space. Hence, we say $(X,\Lambda)$ or equivalently $(X,a)$ is a \textbf{discrete closure (topological) space}. Additionally, for any closure (resp. topological) space $(Y,d)$ any function $f: (X,a)\to (Y,d)$ is continuous in $\cat{Cl}$ (resp. $\cat{Top}$). 
\end{example}

\begin{example}
\label{example:indiscrete_cl}
    Let $(X,\Lambda)$ be an indiscrete pseudotopological space. Then for all $x\in X$, $\mathcal{U}(x)\to x$ by definition of the indiscrete convergence structure and thus the neighborhood filter converges to $x$. Thus $(X,\Lambda)$ is a closure space. One also computes that $a(\varnothing)=\varnothing$, $a(A)=X$ for all $A\subset X$, where $a$ is the adherence for $\Lambda$ or the underlying closure operator. Therefore, the closure operator $a$ is idempotent and thus $(X,a)$ is a topological space. Thus we say $(X,\Lambda)$ or equivalently $(X,a)$ is an \textbf{indiscrete closure (topological) space}. Additionally, for any closure (resp. topological) space $(Y,d)$ any function $f:(Y,d)\to (X,a)$ is continuous in $\cat{Cl}$ (resp. in $\cat{Top}$).
\end{example}

The definition of interior covers of closure spaces motivates the following definition of interior covers of pseudotopological spaces.

\begin{definition}
    \label{def:interior_cover_pseudotopological}
    Let $(X,\Lambda)$ be a pseudotopological space. For $A\subset X$ a \textbf{neighborhood} of $A$ is a subset $U\subset X$ such that $U$ is a neighborhood of $A$ in the closure space $(X,a)$. That is $A\subset i(U)$ where $i(U)=X\setminus a(X\setminus U)$. An \textbf{interior cover} of $(X,\Lambda)$ is a cover of $X$ that happens to be an interior cover of $(X,a)$.
\end{definition}

We have the following version of the pasting lemma.

\begin{proposition}\cite[Proposition A.2.2]{Preuss2002}[Pasting Lemma]
\label{prop:pasting_lemma}
    Let $(X,c)$ be a closure space and let $(X',\Lambda)$ be a limit space. If $A_1$ and $A_2$ are closed subsets of $(X,c)$ such that $X=A_1\cup A_2$, then a map $f\colon (X,c)\to (X',\Lambda)$ is continuous iff the restrictions $f|_{A_i}$ are continuous for $i=1,2$.
\end{proposition}

\subsection{Modifications and inclusion adjunctions}

The following describes different modifications of one type of space to another.  These modifications are also functors on the underlying categories. 

\begin{itemize}[left=0pt]
    \item  Let $(X,\Lambda)\in \cat{PsTop}$. Define $(X,a)\in \cat{Cl}$ by
    \[a(A)=\{x\in X\,|\, \exists\lambda \to x,A\in \lambda\},\]
    that is, $a(A)$ is the adherence of $A$ in $(X,\Lambda)$. One can check that adherences satisfy the closure axioms \new{\citep[Lemma 1.3.3]{Beattie2002}}.
     The closure space $(X,a)$ is the finest closure space coarser than $(X,\Lambda)$, called the \textbf{closure modification} of $\Lambda$.
      \item  Let $(X,c)\in \cat{Cl}$. Define $(X,\tau(c))\in \cat{Top}$ by
      \[\tau(c)(A)=\bigcap \{K\subset X\,|\, c(K)=K,A\subset K\},\]
      that is, $\tau(c)(A)$ is the intersection of all closed sets in $(X,c)$ that contain $A$.
     The topological space $(X,\tau(c))$ is the finest topological space coarser than $(X,c)$, called the \textbf{topological modification} of $(X,c)$.
\end{itemize}

\begin{proposition}{\cite[Proposition 2.3.1.5]{Preuss2002}}
    \label{prop:adjunctions}
    The following are composite adjunctions, with right adjoints being inclusion functors of full subcategories, between $\cat{Top}$ and $\cat{PsTop}$.
    \begin{equation*}
        \begin{tikzcd}
        \cat{Top} \ar[r,shift right=1ex,hook,"\bot"] &
        \cat{Cl} \ar[l,shift right=1ex,"
        \tau"'] \ar[r,shift right=1ex,hook,"\bot"] &
        \cat{PsTop} \ar[l,shift right=1ex,"\textnormal{a}"'] 
        \end{tikzcd} 
    \end{equation*} 
\end{proposition}

\subsection{Limits and colimits of diagrams of spaces}
Let $\cat{C}\in \{\cat{Top},\cat{Cl},\cat{PsTop}\}$ represent any of the categories of spaces among topological, closure, or pseudotopological spaces. If $X$ is an object in $\cat{C}$ we say $X$ is a $\cat{C}$-space or that $X$ is a set with a $\cat{C}$-structure. Here we recall limits and colimits of diagrams of $\cat{C}$-spaces and recall constructions that show $\cat{C}$ is complete and co-complete. For a more detailed treatment of limits and colimits of $\cat{C}$-spaces, we refer the reader to \new{\citet[Chapter 1]{Preuss2002}}.

Let $X\colon \cat{I}\to \cat{C}$ be a diagram of $\cat{C}$-spaces, i.e., $X$ is a functor from some indexing category $\cat{I}$ to $\cat{C}$. More specifically, for all objects $i$ in $\cat{I}$ we have a $\cat{C}$-space $X_i$ and for every morphism $\varphi\colon i\to j$ in $\cat{I}$ we have a continuous map $X_{\varphi}\colon X_i\to X_j$.

\begin{definition}
    A \textbf{limit} of the $\cat{I}$-diagram $X$ in $\cat{C}$ is given by a $\cat{C}$-space $\limit_{\cat{I}}X$ together with continuous maps $p_i\colon\lim_{\cat{I}}X\to X_i$ such that
    \begin{enumerate}[left=0pt]
        \item For $\varphi\colon i\to j$ a morphism in $\cat{I}$ we have $p_j=X_{\varphi}\circ p_i$.
        \item For any $\cat{C}$-space $Y$ and any family of continuous maps $q_i\colon Y\to X_i$, indexed by $\cat{I}$, such that for all $\varphi\colon i\to j$ in $\cat{I}$ we have $q_j=X_{\varphi}\circ q_i$ there exists a unique continuous map $q\colon Y\to \lim_{\cat{I}}X$ such that $q_i=p_i\circ q$ for every object $i$ in $\cat{I}$.
    \end{enumerate}
\end{definition}

Dually, we define colimits of $\cat{C}$-spaces.

\begin{definition}
    A \textbf{colimit} of the $\cat{I}$-diagram $X$ in $\cat{C}$ is given by a $\cat{C}$-space $\colimit_{\cat{I}}X$ together with continuous maps $s_i\colon X_i\to \colimit_{\cat{I}}X$ such that
    \begin{enumerate}[left=0pt]
        \item For $\varphi\colon i\to j$ a morphism in $\cat{I}$ we have $s_i=s_j\circ X_{\varphi}$.
        \item For any $\cat{C}$-space $Y$ and any family of continuous maps $t_i\colon X_i\to Y$, indexed by $\cat{I}$, such that for all $\varphi\colon i\to j$ in $\cat{I}$ we have $t_i=t_j\circ X_{\varphi}$ there exists a unique continuous map $t\colon \colimit_{\cat{I}}X\to Y$ such that $t_i=t\circ s_i$ for every object $i$ in $\cat{I}$.
    \end{enumerate}
\end{definition}

\begin{remark}
    \label{remark:colimits}
    An immediate consequence of the inclusion functors being right adjoints in \cref{prop:adjunctions} is that limits of diagrams are preserved as we move up the categories, as right adjoints preserve limits. For example, a limit of a diagram topological spaces is still a topological space in $\cat{Cl}$ and $\cat{PsTop}$. However, colimits do not have to be preserved. In particular, a colimit of a diagram of topological spaces in $\cat{Cl}$ does not have to be a topological space, see for example \new{\citet[Introduction to Section 33.B]{cech1966topologicalspaces}} and \new{\citet{Bubenik2025}}.
\end{remark}

Recall that a \textbf{concrete category} is a pair $(\cat{K},U)$ where $\cat{K}$ is a category and $U \colon \cat{K}\to \cat{Set}$ is a faithful functor from $\cat{K}$ to the category of sets. For $\cat{C}$-spaces we have the forgetful functors $U\colon \cat{Top},\cat{Cl},\cat{PsTop}\to \cat{Set}$ that given a topological, closure, or pseudotopological space return the underlying set, forgetting the structure. These forgetful functors make any $\cat{C}\in \{\cat{Top},\cat{Cl},\cat{PsTop}\}$ into a concrete category. Furthermore, by \cref{example:discrete_pstop,example:discrete_cl}, $U$ is a right adjoint to the functor $\textnormal{di}\colon \cat{Set}\to \cat{C}$ that assigns to any set $Y$ the discrete topological structure on $Y$. In other words, we have canonical bijections for any $\cat{C}$-space $X$ and any set $Y$:
\[\textnormal{Hom}_{\cat{C}}(\textnormal{di}(Y),X)\cong \textnormal{Hom}_{\cat{Set}}(Y,U(X)).\]
Similarly, $U$ is left adjoint to the functor $\textnormal{ind}\colon \cat{Set}\to \cat{C}$ that assigns to any set $Y$ the indiscrete topological structure on $Y$, by \cref{example:indiscrete_pstop,example:indiscrete_cl}. That is we have canonical bijections for any $\cat{C}$-space $X$ and any set $Y$:
\[\textnormal{Hom}_{\cat{C}}(X,\textnormal{ind}(Y))\cong \textnormal{Hom}_{\cat{Set}}(U(X),Y).\]

Since $U$ is both a left and right adjoint, it preserves limits and colimits. This means that in order to compute limits or colimits in any of these categories we can  compute the limits or colimits in $\cat{Set}$. Afterwards we associate the coarsest or finest, respectively, $\cat{C}$-structures on the limits or colimits in $\cat{Set}$ that will make them limits or colimits in $\cat{C}$.

\begin{example}
    Let $I$ be a set, and let $\cat{I}$ the category whose objects are the elements of $I$ and whose morphisms are only the identity morphisms. 
    \begin{enumerate}[left=0pt]
        \item For a diagram of $\cat{C}$-spaces $X\colon \cat{I}\to \cat{C}$, $\limit_{\cat{I}}X_i$ is the product $\prod_{i\in I}X_i$ with projections maps $p_i\colon \prod_{i\in I}X_i\to X_i$ and the $\cat{C}$-space structure on the underlying set of $\prod_{i\in I}X_i$ is the coarsest $\cat{C}$-structure on $\prod_{x\in I}X_i$ making all the projection maps continuous. We call $\limit_{\cat{I}}X_i$ the \textbf{product} $\cat{C}$-structure. 
        
        If $\cat{C}=\cat{PsTop}$, a filter $\lambda$ on $\prod_{i\in I}X_i$ converges to $x$ in $\prod_{i\in I}X_i$ if and only if $p_i(\lambda)$ converges to $p_i(x)$ in the convergence structure on $X_i$.

        In the case that $\cat{C}\in \{\cat{Top},\cat{Cl}\}$ we can also describe the $\cat{C}$-structure on $\prod_{i\in I}X_i$, i.e., the closure operator on $\prod_{i\in I}X_i$, in the following way.
        For all $x\in \prod_{i\in I}X_i$, let $\mathcal{U}_x$ denote the collection of sets of the form
        \begin{equation*} \label{eq:product}
            \bigcap\{p_i^{-1}(V_i)\,|\,i\in A\},
        \end{equation*}
        where $A\subset I$ is finite and $V_i$ is a neighborhood of $p_i(x)$ in $X_i$. 
        For $B\subset  X$, let 
        \[c(B)=\{x\in X\,|\, \forall U\in \mathcal{U}_x, U\cap B\neq \varnothing\}.\]
        The closure $c$ is called the \textbf{product closure} for $\prod_{i\in I}X_i$ and $(\prod_{i\in I}X_i,c)$ is called the \textbf{product closure space}. For closure spaces $(X,c)$ and $(Y,d)$, their product is denoted by $(X\times Y,c\times d)$. If $X$ is a diagram of topological spaces, the product closure is topological and the resulting topological space is said to have the \textbf{product topology}.
        \item For a diagram of $\cat{C}$-spaces $X\colon \cat{I}\to \cat{C}$, $\colimit_{\cat{I}}X_i$ is the coproduct $\coprod_{i\in I}X_i$ with coprojection (inclusion) maps $s_i\colon X_i\to \coprod_{i\in I}X_i$ and the $\cat{C}$-structure on the underlying set of $\coprod_{i\in I}X_i$ is the finest $\cat{C}$-structure making all the coprojection maps continuous. We call $\colimit_{\cat{I}}X_i$ the \textbf{coproduct} $\cat{C}$-structure. 
        
        Let $\cat{C}=\cat{PsTop}$ and let $x\in \sqcup_{i\in I}X_i$. Suppose $k$ is such that $x\in X_k$. Then, a filter $\lambda$ on $\sqcup_{i\in I}X_i$ converges to $x\in \sqcup_{i\in I}X_i$ if and only if $\lambda=\dot{x}$ or there is a filter $\lambda_k$ converging to $x$ in the convergence structure on $X_k$, such that  $s_i(\lambda_k)\subset \lambda$. 
        
        In the case that $\cat{C}\in\{\cat{Top},\cat{Cl}\}$, the \textbf{coproduct}
        of $\{(X_i,c_i)\}_{i\in I}$ is the  disjoint union of sets $X=\sqcup_i X_i$
        with the closure operation $c$ for $X$ defined by $c(\sqcup_i A_i):=\sqcup_i c_i(A_i)$ for all subsets $\sqcup_i A_i$ of $\sqcup_{i\in I}X_i$. If all $X_i$ are topological spaces, then $c$ is a topological closure operator. 
    \end{enumerate}
\end{example}

\begin{example}
    Let $\cat{I}$ be the category with two objects, $1$ and $2$, and two parallel morphisms from $1$ to $2$. Then an $\cat{I}$-diagram of $\cat{C}$-spaces $X\colon \cat{I}\to \cat{C}$ consists of a pair of continuous maps
    $
    \begin{tikzcd}X_1\ar[r,shift left,"f"] \ar[r,shift right,"g"'] & X_2.
    \end{tikzcd}$
    \begin{enumerate}[left=0pt]
        \item  We call $\limit_{\cat{I}}X$ the \textbf{equalizer} of $f$ and $g$. In particular, the equalizer consists of the $\cat{C}$-space $E$ and map $\iota\colon E \to X_1$ defined in the following manner. Let $E=\{x \in X_1 \ | \ f(x)=g(x)\}$ with $\iota$ the inclusion map and the $\cat{C}$-structure on $E$ being the subspace $\cat{C}$-structure induced from $X_1$ (\cref{def:subspace_convergence,def:subspace_closure}).
        \item  We call $\colimit_{\cat{I}}X$ the \textbf{coequalizer} of $f$ and $g$. 
        In particular, the coequalizer consists of the $\cat{C}$-space $Q$ and a map $q\colon X_2 \to Q$ defined in the following manner. Let $Q$ be the quotient set $X_2/\!\! \sim$, where $\sim$ is the equivalence relation given by $f(x) \sim g(x)$, $\forall x \in X_1$. Let $q\colon X_2 \to Q$ be the quotient map. 
        Then the $\cat{C}$-structure on $Q$ is the finest $\cat{C}$-structure making $q$ continuous. 
        
        In the case that $\cat{C}=\cat{Cl}$, the coequalizer consists of the closure space $(Q,c_q)$, where for any $A \subset Q$, set $c_q(A) = q(c_2(q^{-1}(A)))$. In the case that $\cat{C}=\cat{Top}$ the quotient topology on $Q$ is the finest topology making $q$ continuous. This coincides with the topological modification of $c_q$.
    \end{enumerate}
\end{example}

\begin{theorem}{\cite[Theorem 1.2.1.10]{Preuss2002}} \label{complete}
  The product and equalizer defined above are the categorical product and equalizer in $\cat{C}$ and hence $\cat{C}$ is complete. The coproduct and coequalizer defined above are the categorical coproduct and coequalizer in $\cat{C}$ and hence $\cat{C}$ is co-complete. 
\end{theorem}

We now consider pushouts of $\cat{C}$-spaces.

\begin{definition} \label{def:pushout}
    Consider the maps of $\cat{C}$-spaces $f$ and $i$ in the following diagram.
    \begin{equation} \label{cd:pushout}
        \begin{tikzcd} 
            (A,b) \ar[r,"f"] \ar[d,"i"] & (Y,d) \ar[d,dashed,"j"] \\
            (X,c) \ar[r,dashed,"g"] & (Z,e)   
        \end{tikzcd}
    \end{equation}
    The \textbf{pushout} of $f$ and $i$ is the colimit of the above diagram, given by a $\cat{C}$-space $Z$ and continuous maps $g\colon X\to Z$, $j\colon Y\to Z$.
    The set $Z$ can be taken to be
    $Z = (Y \amalg X)/ \!\! \sim$, where for all $a \in A$, $f(a) \sim g(a)$,
    the induced functions $j$ and $g$.
    In the case that $\cat{C}=\cat{PsTop}$, we have written out the adherence operators $b,c,d$ and $e$ on $A, X,Y$ and $Z$ respectively in \eqref{cd:pushout} for convenience. However the adherence operators are not used to compute the pushout in \eqref{cd:pushout}, the convergence structures are. In the case that $\cat{C}=\cat{Cl}$, then $b,c,d$ and $e$ are the closure operators on $A,X,Y$ and $Z$, respectively. In either case,  for $B \subset Z$ we have
    \begin{equation*}
        e(B) = j (d( j^{-1}(B))) \cup g (c (g^{-1}( B))).
    \end{equation*}
    If $\cat{C}=\cat{Top}$, the topological closure operator on $Z$ is the topological modification of $e$.
\end{definition}

\subsection{Compactness of spaces}
\label{section:Compactness}
Here we discuss the notion of compactness of $\cat{C}$-spaces for $\cat{C}\in \{\cat{Top},\cat{Cl},\cat{PsTop}\}$. 

\begin{definition}
    Let $X$ be a pseudotopological space and let $x\in X$. A \textbf{local covering system} at $x$ is a collection $\mathcal{C}$ of subsets of $X$ such that $\mathcal{C}\cap \lambda\neq \varnothing$ for all $\lambda\to x$ in $X$. For a subset $A$ of $X$, a \textbf{covering system} of $A$ is a collection of subsets of $X$, $\mathcal{C}$, that is a local covering system at each point of $A$.
\end{definition}

\begin{lemma}
\label{lemma:interior_covers_are_covering_systems}
    Let $X$ be a pseudotopological space and $\mathcal{U}$ an interior cover of $X$. Then $\mathcal{U}$ is a covering system of $X$.
\end{lemma}

\begin{proof}
    Let $x\in X$ and let $\lambda\to x$ in $X$. Since $\mathcal{U}$ is an interior cover of $X$, there exists a $U\in \mathcal{U}$ such that $U$ is a neighborhood of $X$. In particular $U$ belongs to the neighborhood filter at $x$, and therefore $U$ belongs to every filter on $X$ that converges to $x$. Therefore $U\in \lambda$ and thus $\mathcal{U}\cap\lambda\neq\varnothing$. Therefore $\mathcal{U}$ is a local covering system at $x$, and therefore $\mathcal{U}$ is a covering system of $x$.
\end{proof}

The following are usually taken to be results for characterizing compactness of pseudotopological and closure spaces, but we will take them as definitions of compactness. 

\begin{definition}{\cite[Proposition 1.4.15]{Beattie2002}}
    A pseudotopological space $X$ is \textbf{compact} if for every covering system $\mathcal{C}$ of $X$, there is a finite subsystem $\mathcal{C}'\subset \mathcal{C}$ which covers $X$ ($\mathcal{C}'$ does not have to be a covering system of $X$, only a cover of $X$).
\end{definition}

In particular if $X$ is a compact pseudotopological space, then every interior cover of $X$ has a finite subcover by \cref{lemma:interior_covers_are_covering_systems}.

\begin{definition}{\cite[Theorem 41.A.9]{cech1966topologicalspaces}}
    A closure space $X$ is \textbf{compact} if every interior cover $\mathcal{U}$ of $X$ has a finite subcover (the finite subcover does not have to be an interior cover).
\end{definition}

A topological space $X$ is \textbf{compact} if every open cover has a finite subcover. As expected, the definitions of compactness agree across the categories $\cat{Top}$, $\cat{Cl}$ and $\cat{Pstop}$. For example, if $X$ is a topological space then $X$ is a compact pseudotopological space if and only if $X$ is a compact topological space.

\begin{proposition}{\cite[Proposition 1.4.7]{Beattie2002}}
\label{prop:image_of_compact_is_compact}
    Let $f\colon X\to Y$ be a map of pseudotopological spaces. If $X$ is compact then $f(X)$ as a subspace of $Y$ is compact.
\end{proposition}

\begin{lemma}
\label{lemma:finite_cover_whose_image_is_contained_in_covering_system}
Let $X$ be a compact topological space. Let $Y$ be a pseudotopological space with interior cover $\mathcal{V}$, and suppose that $f\colon X\to Y$ is continuous. Then there exists an (finite) open cover $\mathcal{U}$ of $X$ such that for all $U\in \mathcal{U}$, there is a $V\in \mathcal{V}$ such that $f(U)\subset V$.
\end{lemma}

\begin{proof}
    Let $d$ be the adherence operator on $Y$. Since $f\colon X\to Y$ is morphism in $\cat{PsTop}$ the function $f\colon X\to (Y,d)$ is a morphism in $\cat{Cl}$ by \cref{lemma:cont_maps_in_pstop_are_cont_in_cl}. Since $\mathcal{V}$ is an interior cover, for all $x\in X$, there is a neighborhood of $f(x)$ in $\mathcal{V}$. Label this neighborhood by $V_{f(x)}$. By \cref{theorem:continuity_in_terms_of_nbhds}, for the neighborhood $V_{f(x)}\subset Y$ of $f(x)$ there exists a neighborhood $W_x \subset X$ of $x$ in $X$ such that $f(W_x) \subset V_{f(x)}$. Since $W_x$ is a neighborhood of $x$ and $X$ is topological, there is an open set $U_x \subset W_x$ with $x \in U_x$ for every $x \in X$. The collection $\{U_x\}_{x\in X}$ is thus an open cover of $X$ which refines $\{W_x\}_{x\in X}$ by construction. Since $X$ is compact, $\{U_x\}_{x\in X}$ admits a finite subcover $\mathcal{U}=\{U_{x_i}\}_{i=1}^n$, and $\mathcal{U}$ satisfies the conclusion of the lemma, by construction. 
\end{proof}

\subsection{Exponential objects} Let $\cat{C}$ be a category with finite products. An object $X$ of $\cat{C}$ is called \textbf{exponentiable} if product with $X$ is left adjoint, i.e., if there are natural bijections 
\[\textnormal{Hom}_{\cat{C}}(X\times Z,Y)\cong \textnormal{Hom}_{\cat{C}}(Z,Y^X),\]
for all objects $Y$ and $Z$ of $\cat{C}$. The \textbf{power object} $Y^X$ is given by the right adjoint $-^X$ of $X\times -$. 

Suppose now that $\cat{C}\in \{\cat{Top},\cat{Cl},\cat{PsTop}\}$. Then $\cat{C}$ is a concrete category over $\cat{Set}$ and constant maps are morphisms in $\cat{C}$. Let $X$ be an exponential object in $\cat{C}$. In particular if $Z=\{*\}$ is the one-point space, i.e., is the terminal object in 
$\cat{C}$, we have $\textnormal{Hom}_{\cat{C}}(X\times \{*\},Y)\cong  \textnormal{Hom}_{\cat{C}}(\{*\},Y^X)$. Thus there is a natural bijection between $\textnormal{Hom}_{\cat{C}}(X,Y)$ and the elements of the power object $Y^X$. This means that there are canonical function spaces for exponential 
objects $X$. Furthermore, let $L_X,R_X\colon \cat{C}\to \cat{C}$ be the functors $Y\mapsto Y\times X$ and $Y\mapsto Y^X$, respectively. The \textbf{evaluation mapping} $\textnormal{ev}_X\colon L_X\circ R_X\to \id_{\cat{C}}$ is the co-unit of the adjunction, in particular $\textnormal{ev}_{X,Y}\colon Y^X\times X\to Y$ is a natural morphism in $\cat{C}$, for every $Y$ in $\cat{C}$. With the 
identification $Y^X\approx \textnormal{Hom}_{\cat{C}}(X,Y)$, the $\cat{C}$-structure on $\textnormal{Hom}_{\cat{C}}(X,Y)$ must be so that $\textnormal{ev}_{X,Y}\colon \textnormal{Hom}_{\cat{C}}(X,Y)\times X\to Y$ is continuous. Furthermore, one computes that $\textnormal{ev}_{X,Y}(f,x)=f(x)$.

\begin{example}
\label{example:continuous_convergence_structure}
    Let $(X,\Lambda_X)$ and $(Y,\Lambda_Y)$ be pseudotopological spaces. The \textbf{continuous convergence structure} $\Lambda_c$ on $\textnormal{Hom}_{\cat{PsTop}}((X,\Lambda_X),(Y,\Lambda_Y))$, which is in fact pseudotopological \new{\citep[Theorem 1.55]{Beattie2002}}, is given by 
    \[\mathcal{H}\to f\textnormal{ in } \Lambda_c\textnormal{ if and only if }\textnormal{ev}_{X,Y}(\mathcal{H}\times \lambda)\to f(x)\forall(\lambda,x)\in \Lambda_X.\]
    In other words, $\Lambda_c$ is the coarsest convergence on $\textnormal{Hom}_{\cat{PsTop}}((X,\Lambda_X),(Y,\Lambda_Y))$ such that the evaluation map $\textnormal{ev}_{X,Y}$ is continuous. Let $a_c$ denote the adherence operator associated to $\Lambda_c$.
\end{example}

\begin{example}
\label{example:function_spaces}
We have the following examples of exponential objects in different concrete categories over $\cat{Set}$ together with function spaces for exponential objects.
    \begin{enumerate}[left=0pt]
        \item Let $\cat{Top}_2$ be the category of $T_2$ (Hausdorff) topological spaces. A $T_2$ topological space $X$ is exponential in $\cat{Top}_2$ if and only if $X$ is locally compact. In particular, if $X$ is locally compact, then $\textnormal{Hom}_{\cat{Top}_2}(X,Y)$ has a canonical topology that makes the evaluation map continuous.  This is the well known \textbf{compact open topology}. See \new{\citet{Fox1945}} and \new{\citet{Michael1968}} for more details.
        \item Consider $\cat{Top}$, the category of topological spaces. A topological space $X$ is exponential in $\cat{Top}$ if and only if $X$ is quasi locally compact. Once again, $\textnormal{Hom}_{\cat{Top}}(X,Y)$ with the canonical compact open topology makes the evaluation map continuous. See \new{\citet{Day1970}} and \new{\citet{Hofmann1978}} for more details.
        \item Consider $\cat{PsTop}$, the category of pseudotopological spaces. The category $\cat{PsTop}$ is Cartesian closed and therefore any object $X$ of $\cat{PsTop}$ is exponential. The pseudotopological structure on $\textnormal{Hom}_{\cat{PsTop}}(X,Y)$ that makes the evaluation map continuous, for any pseudotopological spaces $X$ and $Y$, is the one described in \cref{example:continuous_convergence_structure}. See \new{\citet{Edgar1976}} and \new{\citet{Nel1977}} for more details. 
        \item Consider $\cat{Cl}$, the category of closure spaces. A closure space $X$ is exponential in $\cat{Cl}$ if and only if $X$ is finitely generated, i.e. $X$ is a colimit of a diagram of finite spaces. It can be shown that this is equivalent to the requirement that every point in $X$ has a smallest neighborhood. The closure space structure on $\textnormal{Hom}_{\cat{Cl}}(X,Y)$ making the evaluation map continuous, for given closure spaces $X$ and $Y$, is the following. First consider $X$ and $Y$ as objects in $\cat{PsTop}$ and consider the continuous convergence structure $\Lambda_c$ on $\textnormal{Hom}_{\cat{PsTop}}(X,Y)$ from \cref{example:continuous_convergence_structure}. Note that $\textnormal{Hom}_{\cat{PsTop}}(X,Y)=\textnormal{Hom}_{\cat{Cl}}(X,Y)$ as sets,  since as $\cat{Cl}$ is a full subcategory of $\cat{PsTop}$. Consider the adherence operator $a_c$ associated to $\Lambda_c$ and make $\textnormal{Hom}_{\cat{Cl}}(X,Y)$ into a closure space, where the closure operation is precisely $a_c$.
        See \new{\citet{LowenColebunders1993}} and \new{\citet{Richter1997}} for more details.
    \end{enumerate}
\end{example}

\begin{remark}
\label{remark:exponential_objects}
    In \cref{example:function_spaces}, we have a sequence of full subcategories $\cat{Top}_2\hookrightarrow \cat{Top}\hookrightarrow \cat{Cl}\hookrightarrow \cat{PsTop}$. Furthermore, as we go from $\cat{Top}_2$ through $\cat{Top}$ and finally $\cat{PsTop}$ the collection of exponentiable objects keeps expanding. Surprisingly, the collection of exponentiable objects shrinks when going from $\cat{Top}$ to $\cat{Cl}$. In particular, the unit interval $I$ with its standard topology is compact and hence a locally compact topological space and is thus an exponentiable object in $\cat{Top}$. However, no point in $I$ has a smallest neighborhood with respect to the topology as one can find arbitrarily small open euclidean balls around each point. Thus $I$ is not exponentiable in $\cat{Cl}$.
\end{remark}

\subsection{CW complexes}
Let $\cat{C}\in \{\cat{Top},\cat{Cl},\cat{PsTop}\}$. Here we show how to construct CW complexes in $\cat{C}$. 

Let $S^n$ be the topological $n$-sphere and $D^n$ the topological $n$-disk. A $\cat{C}$-space $X$ is obtained from a $\cat{C}$-space $A$ by \textbf{attaching cells} if there exists a pushout in $\cat{C}$,

\begin{center}
    \begin{tikzcd}
        \coprod_{j} S_j^{n_j-1}\arrow[d,hookrightarrow,"\coprod_j\iota_j"]\arrow[r,"\varphi"] & A\arrow[d,"\iota_A",dashed]\\
        \coprod_{j}D_j^{n_j}\arrow[r,"\Phi",dashed] & X
    \end{tikzcd}
\end{center}
where for each $j$, $D_j^{n_j}$ is the $n_j$-dimensional disk and $\iota_j\colon S_j^{n_j-1}\to D_j^{n_j}$ is the inclusion of its boundary. We call $\varphi$ the \textbf{cell-attaching map}. If for all $j$, $n_j=n$ we say that $X$ is obtained from $A$ by \textbf{attaching $n$-cells}.

For a $\cat{C}$-space $X$ and a closed subspace $A$, we say the pair $(X,A)$ is a \textbf{relative CW complex} in $\cat{C}$ if $X$ is the colimit of a diagram of $\cat{C}$-spaces
\[A=X^{-1}\hookrightarrow X^0\hookrightarrow X^1\hookrightarrow X^2\hookrightarrow\cdots,\]
where for all $n\ge 0$, $X^n$ is obtained from $X^{n-1}$ by attaching $n$-cells. If $A=\varnothing$, we say $X$ is a \textbf{CW complex} in $\cat{C}$. If the total numbers of cells attached is finite, we call $X$ a \textbf{finite CW complex} in $\cat{C}$. If for each $n\ge 0$, there are only finitely many $n$-cells attached, we call $X$ a \textbf{CW complex of finite type} in $\cat{C}$. If no more cells are attached after a certain $n\ge 0$, we say $X$ is a \textbf{finite-dimensional CW complex} in $\cat{C}$.

It was shown by \citet{Bubenik2025} that finite-dimensional
CW complexes, CW complexes of finite type, and finite relative
CW complexes need not be topological in $\cat{Cl}$ or $\cat{PsTop}$. However, this does not happen in the case of finite CW complexes. In particular, from \new{\citet[Corollary 4.26]{Ebel2023}} we immediately have the following:

\begin{proposition}
\label{prop:finite_cw_complexes_are_topological}
    The construction of finite CW complexes in the categories $\cat{Top}$ and $\cat{PsTop}$ agree. In particular,  the construction of finite CW complexes in the categories $\cat{Top}$ and $\cat{Cl}$ agree. 
\end{proposition}

The construction of finite CW complexes in $\cat{Top}$ and $\cat{Cl}$ agreeing is also a special case of applying \new{\citet[Theorems 4.2 and 4.4]{Bubenik2025}}.

\subsection{Graphs as closure spaces} 
Here we recall how the categories of simple directed graphs and simple undirected graphs are isomorphic to full subcategories of closure spaces. 

\begin{definition}
\label{def:digraphs}
    A \textbf{simple directed graph}, or \textbf{digraph} for short, is a pair $(X,E)$ where $X$ is a set and $E$ is a reflexive relation on $X$. A \textbf{simple undirected graph}, or \textbf{graph} for short, is a digraph $(X,E)$ such that $E$ is symmetric. 
    A \textbf{digraph morphism} $f\colon (X,E)\to (Y,F)$ is a set map $f\colon X\to Y$ such that whenever $xEx'$ we have $f(x)Ff(x')$.
\end{definition}

If $(X,E)$ is a digraph and $A\subset X$, the \textbf{induced subgraph} $(A,E)$ is defined by $aE_Aa'$ if and only if $aEa'$ for all $a,a'\in A$.
We denote the category of digraphs and digraph morphisms by $\cat{DiGph}$. Let $\cat{Gph}$ be the full subcategory of $\cat{DiGph}$ consisting of graphs.

\begin{definition}
    A closure space $(X,c)$ is \textbf{Alexandroff} if $c(A)=\bigcup_{x\in A}c(x)$ for all $A\subset X$. A closure space $(X,c)$ is \textbf{symmetric} if whenever $y\in c(x)$, we have $x\in c(y)$. Let $\cat{Cl_a}$, $\cat{Cl_{sA}}$ denote the full subcategories of $\cat{Cl}$ consisting of Alexandroff and symmetric Alexandroff closure spaces, respectively.
\end{definition}

To each digraph $(X,E)$ we can associate an Alexandroff closure space $(X,c_E)$ defined by $c_E(x)=\{y\in X\,|\, xEy\}$ and $c_E(A)=\bigcup_{x\in A}c_E(x)$. Obviously if $E$ was symmetric, then $(X,c_E)$ will be a symmetric Alexandroff closure space. Conversely, given an Alexandroff closure space $(X,c)$ we can associate with it a digraph $(X,E_c)$ defined by $xE_cy$ if and only if $y\in c(x)$ and if $(X,c)$ was a symmetric Alexandroff closure space then $E_c$ will also be symmetric. 

\begin{remark}
    \new{Sometimes} digraphs are defined without any loops, i.e. $xEx$ is not allowed for any $x\in X$ and a digraph morphism is defined as function $f\colon (X,E)\to (Y,F)$ such that whenever $xEx'$ we have $f(x)Ff(x')$ or $f(x)=f(x')$. However, this category of digraphs is clearly isomorphic to the category of digraphs from \cref{def:digraphs}. The reason why we prefer our definition is that otherwise we would have to define the closure operator associated to a digraph as $c_E(x)=\{x\}\cup \{y\in X\,|\, xEy\}$. However, whenever we draw a digraph in this paper we will not include any loops for simplicity.
\end{remark}

These assignments are functorial and furthemore are inverses to one-another. In particular, if $(X,E)$ is a digraph and $A\subset X$ then the induced subgraph $(A,E_A)$ corresponds to the subspace closure on $A$ induced from $(X,c_E)$.
In other words we have the following:

\begin{proposition}
    The categories $\cat{Cl_A}$ and $\cat{DiGph}$ are isomorphic. This isomorphism also restricts to the subcategories $\cat{Cl_{sA}}$ and $\cat{Gph}$. 
\end{proposition}

The proof follows from definitions, see \new{\citet{Bubenik2024}} and \new{\citet{Dikranjan1995}} for example.

\begin{remark}
    \label{remark:finite_pseudotopological_spaces}
    Note that every finite pseudotopological space is in fact a finite closure space and hence equivalent to a digraph. Indeed, let $(X,\Lambda)$ be a finite pseudotopological space. Since $X$ is a finite set, for each $x\in X$, the collection $\{\lambda\,|\lambda\to x\}$ is finite. Because every pseudotopological space is a limit space, we have that for all $x\in X$, $\mathcal{U}(x)=\bigcap\{\lambda\,|\,\lambda\to x\}\to x$ which by \cref{prop:alternative_char_of_closure_spaces} shows that $(X,\Lambda)$ is a closure space.
\end{remark}

\section{Homotopy and homology for pseudotopological spaces}
\label{section:homotopy_and_homology}

In this section, we develop the basics of homotopy and singular homology groups of pseudotopological spaces.  Most of the results, if not the proofs, are completely analogous to the standard results about homotopy and singular homology groups of topological spaces.  Of particular interest in this section is a pseudotopological version of the Quillen fiber lemma, \cref{corollary:weak_hom_eq_on_inter_cover}, which plays a key role in our main result, \cref{theorem:vr_weak_homotopy_equivalence}.  

One of the main results of \new{\citet{rieser2022cofibration}} for our purposes is that $\cat{PsTop}$ is an I-category \new{\citep[Corollary 4.23]{rieser2022cofibration}}.  In particular, this means that $\cat{PsTop}$ has a cylinder object given by the obvious cylinder construction; that is, for any pseudotopological space $X$, the \textbf{cylinder object} is given by $(X\times [0,1], i_0, i_1,p)$ where $[0,1]=I$ is the topological unit interval, $i_0, i_1\colon X \to  X\times [0,1]$ by $i_0(x)=(x,0)$ and $i_1(x)=(x,1)$ are \textbf{inclusions}, and $p\colon X\times I \to X$ by $p(x,t)=x$ is the \textbf{projection}. 

\subsection{\texorpdfstring{$\pi_0$}{pi\_0} in \texorpdfstring{$\cat{PsTop}$}{PsTop}}
Let $X$ be a pseudotopological space. A \textbf{path} from $x$ to $y$ in $X$ is a continuous map $\alpha\colon I\to X$ such that $\alpha(0)=x$ and $\alpha(1)=y$. If $\alpha$ is a path, its \textbf{inverse path} is the map $\overline{\alpha}\colon I\to X$ defined by $t\mapsto 1-t$. If $\alpha$ and $\beta$ are two paths in $X$ such that $\alpha(1)=\beta(0)$, the \textbf{concatenation} of $\alpha$ and $\beta$, $\alpha* \beta$, is path in $X$ from $\alpha(0)$ to $\beta(1)$ defined by 
\[\alpha* \beta(t)=\begin{cases}
    \alpha(2t),& 0\le t\le\frac{1}{2}\\
    \beta(2t-1), & \frac{1}{2}\le t\le 1
\end{cases}.\]
We put a relation on $X$ by saying $x\sim y$ if there is a path from $x$ to $y$. By reversing and concatenating paths, which is made possible by \cref{prop:pasting_lemma,prop:finite_cw_complexes_are_topological}, it follows that $\sim$ is an equivalence relation on $X$. The equivalence classes of $\sim$ are called the \textbf{path components} of $X$. We write $\pi_0(X)$ for the set of path components of $X$. We say $X$ is \textbf{path-connected} or $0$-\textbf{connected} if $\pi_0(X)$ is a singleton. Equivalently, $X$ is path connected if there is a path between any pair of points in $X$.

\subsection{Homotopies of maps in \texorpdfstring{$\cat{PsTop}$}{PsTop}}
\label{section:homotopies}
This existence of a cylinder object yields the definition of homotopy as the obvious one. That is, continuous maps $f,g\colon X\to Y$ in $\cat{PsTop}$ are \textbf{homotopic}, and we write $f\simeq g$ if there exists a continuous $H\colon X\times I\to Y$ such that $H\circ i_0=f$ and $H\circ i_1=g$ (equiv. $H(x,0)=f(x)$ and $H(x,1)=g(x)$). The map $H$ is called a \textbf{homotopy} between $f$ and $g$, denoted $H\colon f\simeq g$. Given $H\colon f\simeq g$, the \textbf{inverse homotopy} $-H\colon X\times I\to Y$ is the map $(x,t)\mapsto (x,1-t)$. We have $-H\colon g \simeq f$. If $H\colon f\simeq g$ and $L\colon g\simeq h$ are given, the \textbf{glued homotopy} $H+L\colon X\times I\to Y$ is given by 
\[H+L(x,t)=\begin{cases}
    H(x,2t), & 0\le t\le \frac{1}{2}\\
    L(x,2t-1), & \frac{1}{2}\le t\le 1
\end{cases},\]
and we have $H+L\colon f\simeq h$.
This is indeed a continuous map by \cref{prop:pasting_lemma,prop:finite_cw_complexes_are_topological}. This was also argued more directly in \new{\citet[Section 3.2.1]{rieser2022cofibration}}. It follows that $\simeq$ is an equivalence relation on $\textnormal{Hom}_{\cat{PsTop}}(X,Y)$. For a continuous map $f\colon X\to Y$ its equivalence class is denoted by $[f]$ and is called the \textbf{homotopy class} of $f$. The set of homotopy classes of maps $f\colon X\to Y$ is denoted by $[X,Y]$. For a subspace $A\subset X$, we say a homotopy $H\colon X\times I\to Y$ is \textbf{relative} to $A$ if for all $t\in I$, we have $H(x,t)=f(x)=g(x)$ for all $x\in A$. That is, the homotopy $H$ is constant on $A$. In this case we write $H\colon f\simeq g\textnormal{ (rel A)}$. We will use $H_t$ to denote the map $H_t\colon X\to Y$ defined by $H_t(x)=H(x,t)$.

We say two pseudotopological spaces $X$ and $Y$ are \textbf{homotopy equivalent} or have the same \textbf{homotopy type} if there exist continuous maps $f\colon X\to Y$ and $g\colon Y\to X$ such that $fg\simeq \id_{Y}$ and $gf\simeq \id_{X}$. If $X$ and $Y$ are homotopy equivalent via maps $f\colon X\to Y$ and $g\colon Y\to X$ we call $f$ and $g$ \textbf{homotopy equivalences} and we say $g$ is a \textbf{homotopy inverse} of $f$. A map $f\colon X\to Y$ is called a \textbf{null homotopic} if it is homotopic to a constant map and in this case the homotopy is called a \textbf{null homotopy}. For a subspace $A$ of $X$ a map $r\colon X\to X$ is a \textbf{retraction} of $X$ onto $A$ if $r(X)=A$ and $r|_A=\id_{A}$, and we say $A$ is a \textbf{retract} of $X$. A \textbf{deformation retraction} is a homotopy between the identity map on $X$ and a retraction $r$. In this case, we say $X$ \textbf{deformation retracts onto $A$}. If a deformation retraction is a homotopy relative to $A$, we say it is a \textbf{strong deformation retraction}. A space $X$ is \textbf{contractible} if $X$ deformation retracts to a one point space and a null homotopy of $\id_{X}$ is called a \textbf{contraction}. 

\begin{lemma}
\label{lemma:only_nbhd_whole_space_then_contractible}
    Let $Y$ be a closure space containing a point $a$ such that the only neighborhood in $Y$ containing $a$ is $Y$ itself. Then $Y$ is contractible.
\end{lemma}

\begin{proof}
    Define $H\colon Y\times I\to X$ by $H(y,t)=\begin{cases}
    y, & 0\le t<1,\\
    a, &t=1
    \end{cases}$. We need to show that $H$ is continuous. Let $(y,t)\in Y\times I$ and let $U$ be a neighborhood of $H(y,t)$ in $Y$. There are two cases to consider.

    \textbf{Case 1:} Suppose that $t=1$ or $y=a$. Then $H(y,t)=a$ and $U=Y$ by assumption. Thus $Y\times I$ is a neighborhood of $(y,t)$ such that $H(Y\times I)\subset Y=U$.

    \textbf{Case 2:} Suppose that $t<1$. Then $H(y,t)=y$ and $H^{-1}(U)=U\times [0,1)$ is a neighborhood of $(y,t)$ in $Y\times I$ such that $H(U\times [0,1])\subset U$.

    Thus, $H$ is continuous by \cref{theorem:continuity_in_terms_of_nbhds}.
\end{proof}

\subsection{Mapping cylinder in \texorpdfstring{$\cat{PsTop}$}{PsTop}}

Here we define the mapping cylinder in $\cat{PsTop}$. We will be using the fact that $\cat{PsTop}$ is Cartesian closed.

\begin{definition}
\label{def:mapping_cylinder}
Let $f\colon X\to Y$ be continuous between pseudotopological spaces. The \textbf{mapping cylinder} of $f$, denoted $M_f$, is given by the pushout of the diagram
$$
\xymatrix{
X\ar[r]^f\ar[d]^{i_0} & Y \ar@{-->}[d]^j\\
X\times I \ar@{-->}[r]^g& M_f
}
$$
where $i_0(x)=(x,0)$.  The mapping cylinder is explicitly given by 
$$M_f=((X\times I)\sqcup Y)/((x,0)\sim f(x))
$$
$j(y)=[y]$, and $g(x,t)=[x,t].$
\end{definition}

\begin{proposition}
\label{prop: mapping cylinder}
Let $f\colon X\to Y$ in $\cat{PsTop}$.  Then $f$ factors as $f=r\circ i$ where $i\colon X\to M_f$, is the inclusion given by composition
\[X\xrightarrow{i_1}X\times I\xrightarrow{g}M_f,\]
$i_1(x)=(x,1)$, and $r\colon M_f\to Y$, $r([x,t])=f(x)$ and $r([y])=y$, is a homotopy equivalence.
\end{proposition}

\begin{proof}
Note that $i\colon X\to M_f$ is given by $i(x)=[x,1]$. Clearly $f=r\circ i$. We first check that $r$ is indeed continuous. Let $k\colon X\times I\to Y$ be the map $(x,t)\mapsto f(x)$. Note that $k$ factors as $k=f\circ \pi_X$, where $\pi_X\colon X\times I\to X$ is the canonical projection, and thus $k$ is continuous. Note that $r$ is the unique continuous map induced by the pushout $M_f$ and the maps $h\colon X\times I\to Y$ and $\mathrm{id}_Y\colon Y\to Y$.

$$
\xymatrix{
X\ar[r]^f\ar[d] & Y \ar[d]^j \ar@/^/[rdd]^{\mathrm{id}_Y}\\
X\times I \ar[r] \ar@/_/[drr]^k & M_f \ar@{-->}[rd]^r\\
&&Y
}
$$

From the commutative diagram above, we have that $r\circ j=\id_Y$. Thus, we only need to show $j\circ r \simeq \id_{M_f}$. 

Since $\cat{PsTop}$ is Cartesian closed, the unit interval $I$ is an exponentiable object, i.e., the functor $-\times I\colon \cat{PsTop}\to \cat{PsTop}$ is a left adjoint (\cref{example:function_spaces}(3)). Therefore 
$$
\xymatrix{
X\times I\ar[rr]^{f\times \mathrm{id}_I}\ar[dd]^{i_0\times \mathrm{id}_I} & &Y\times I \ar@{-->}[dd]^{j\times \mathrm{id}_I}\\
\\
X\times I\times I \ar@{-->}[rr]^{g\times \mathrm{id}_I}&& M_f\times I
}
$$
is also a pushout.  Define $m\colon X\times I \times I\to M_f$ by $m(x,t,s)=[x,t(1-s)]$ and $\ell\colon Y\times I \to M_f$ by $\ell(y,s)=[y]$.  Write write $m=g\circ m' \colon X\times I\times I\to X\times I\to M_f$ where $m'$ is the map $(t,s)\to t(1-s)$. Then $m'$ is a map between topological spaces $I\times I\to I$ and is known to be continuous. Similarly, we can write $\ell$ as a composition $\ell=j\circ\pi_Y$, where $\pi_Y\colon Y\times I\to Y$ is the canonical projection, yielding that $\ell$ is also continuous.  Furthermore, $m\circ (i_0\times \id_I)=\ell\circ (f\times \id_I)$ so that there is a unique continuous map $H\colon M_f\times I \to M_f$ given by the pushout:

$$
\xymatrix{
X\times I\ar[rr]^{f\times \mathrm{id}_I}\ar[dd]^{i_0\times \mathrm{id}_I} & &Y\times I \ar[dd]_{j\times \mathrm{id}_I} \ar@/^/[dddr]^{\ell}&&\\
\\
X\times I\times I \ar[rr]^{g\times \mathrm{id}_I} \ar@/_/[rrrd]^m && M_f\times I\ar@{-->}[rd]^H&&\\
&&& M_f\\
}
$$

Observe that
\begin{eqnarray*}
H([x,t],0)&=&[x,t]=\id_{M_f}[x,t]\\
H([y],0)&=&[y]=\id_{M_f}[y]
\end{eqnarray*}

and

\begin{eqnarray*}
H([x,t],1)&=&[x,0]=[f(x)]=j(f(x))=j\circ r([x,t])\\
H([y],1)&=&[y]=j(y)=j\circ r([y]).
\end{eqnarray*}

Therefore $H$ is a homotopy and thus
 $j\circ r\simeq \id_{M_f}.$
\end{proof}

\begin{remark}
\label{remark:no_mapping_cylinder_in_cl}    
Since $I$ is not an exponentiable object in $\cat{Cl}$ (\cref{remark:exponential_objects}), we cannot guarantee the existence of mapping cylinders in $\cat{Cl}$. More precisely, the pushout $M_f$ still exists as a closure space as $\cat{Cl}$ is cocomplete. However in arguing for the homotopy $H\colon M_f\times I\to M_f$ between $\id_{M_f}$ and $r\colon M_f\to Y$ being continuous, in $\cat{PsTop}$ we used the fact that product with $I$ preserves pushouts. We cannot do this in $\cat{Cl}$. Mapping cylinders will be of great use in our proofs and thus the need to work in $\cat{PsTop}$ instead of $\cat{Cl}$.
\end{remark}

\subsection{Homotopy groups of pseudotopological spaces}

We now define homotopy groups of pseudotopological spaces and derive the long exact sequence of homotopy groups.  

We call a pair $(X,x_0)$ consisting of a pseudotopological space $X$ and a \textbf{base point} $x_0\in X$ a \textbf{pointed space}. A \textbf{pointed map} $f\colon (X,x_0)\to (Y,y_0)$ is a continuous map $f\colon X\to Y$ such that $f(x_0)=y_0$. A homotopy $H\colon f\simeq g$ is \textbf{pointed} if $H_t$ is pointed for all $t\in I$. Denote the set of pointed homotopy classes by $[(X,x_0),(Y,y_0)]$ or by $[X,Y]^0$ (with fixed base points assumed). We can analogously define concepts such as
\textbf{pointed homotopy equivalence}, \textbf{pointed contractible} and \textbf{pointed null homotopy} for instance.

A \textbf{pair} $(X,A)$ \textbf{of pseudotopological space} consists of a pseudotopological space $X$ with a subspace $A$ of $X$. A \textbf{continuous map of pairs} $f\colon (X,A)\to (Y,B)$ is a continuous map $f\colon X\to Y$ such that $f(A)\subset B$. A \textbf{homotopy of pairs} is then a homotopy $H\colon X\times I\to Y$ such that $H_t$ is a map of pairs for all $t\in I$. We write $[(X,A),(Y,B)]$ for the homotopy classes of maps of pairs. Note that if $X$ is a closure space (resp. topological space), then $A$ is a closure space (resp. topological space) as the categories $\cat{Cl}$ and $\cat{Top}$ are reflective subcategories of $\cat{PsTop}$ (\cref{remark:colimits}).

\begin{lemma}
\label{lemma:homotopy_quotients}
    Let $q\colon X\to Y$ be a quotient map of pseudotopological spaces. Suppose $\{H_t\colon Y\to Z\}_{t\in I}$ is a family of functions (not necessarily continuous) such that the collection  $H_t\circ q$, $t\in I$, is a homotopy. Then the family $H_t$ is also a homotopy.
\end{lemma}

\begin{proof}
    The product $q\times \id_{X}\colon X\times I\to Y\times I$ is a quotient map. This is because a quotient map is an example of a colimit and $\cat{PsTop}$ is Cartesian closed  and thus is  $I$ is an  exponentiable object and therefore it preserves colimits. The composition $H\circ (q\times \id_{X})$ is continuous by assumption and thus by colimit property of quotients, the map $H$ must be continuous as well.
\end{proof}

A continuous map of pairs $f\colon (X,A)\to (Y,y_0)$ into a pointed space induces a pointed map $\overline{f}\colon X/A\to Y$. Furthermore, by \cref{lemma:homotopy_quotients} the assignment $f\mapsto \overline{f}$ induces a bijection 
\[[(X,A),(Y,y_0)]\cong [X/A,Y]^0.\]
With this in mind, when we have a pair $(X,A)$ we consider the quotient space $X/A$ as a pointed space (with $A$ identified with a point) with base point $\{A\}$. If $A=\varnothing$, then $X/A=X^+$ is the space $X$ with a separate base point.

Let $I^n$ be the $n$-fold product of $I$ with itself. Denote by $\partial I^n$ the set $\partial I^n=\{(t_1,\dots, t_n)\in I^n \colon t_i\in \{0,1\} \text{ for at least one } i\}$, with the subspace structure induced from $I^n$, for $n\geq 1$. We set $I^0=\{0\}$, and $\partial I^0=\varnothing$. In $I^n/\partial I^n$ we use $\partial I^n$ as the base point. For $n=0$ we have $I^0/\partial I^0=\{0\}\sqcup \{*\}$, that is we get an additional disjoint base point $*$. 

The $n$-th homotopy group, $n\ge 1$, of a pointed pseudotopological space $(X,x_0)$ is the set of relative homotopy classes of maps
\[\pi_n(X,x_0)=[(I^n,\partial I^n),(X,x_0)],\]
with the following group structure. Suppose that $[f],[g]\in \pi_n(X,x_0)$. Then a representative of $[f]+[g]$ is $f+_ig$ where 
\begin{equation}
\label{eq:homotopy_group_operation}
    (f+_ig)(t_1,\dots, t_n)=\begin{cases}
    f(t_1,\dots,t_{i-1},2t_i,\dots,t_n),& 0\le t_i\le \frac{1}{2}\\
    g(t_1,\dots,t_{i-1},2t_{i}-1,\dots,t_n)& \frac{1}{2}\le t_i\le 1
\end{cases}.
\end{equation}
This construction extends the usual definition of homotopy groups for topological spaces.  Furthermore, the proof of \new{\citet[Propositon 6.1.1]{tomDieck2008}} applies verbatim and it follows that for $n\ge 2$, $\pi_n(X,x_0)$ are abelian groups and that the definition of $[f]+[g]$ does not depend on the choice of $i$, $1\le i\le n$.
Finally, if $X$ is a path connected pseudotopological space, then $\pi_n(X,x_0)$ does not depend on the base point $x_0$ and we write $\pi_n(X)$ instead.
For $n=0$, $\pi_0(X,x_0)$ is the set of path components of $X$, $\pi_0(X)$, with base point at $[x_0]$.

\subsection{Relative homotopy groups}
We now define relative homotopy groups (sets) for a pointed pair $(X,A)$, i.e., a pair of spaces $(X,A)$ with a base point $x_0\in A\subset X$. For $n\ge 1$, let $J^n=\partial I^n\times I\cup I^n\times \{0\}\subset \partial I^{n+1} \subset I^{n+1}$ as a subspace of $I^{n+1}$. Since $I^{n+1}$ is topological, so is $J^n$. Let $J^0=\{0\}\subset I$. Then define
\[\pi_{n+1}(X,A,x_0)=[(I^{n+1},\partial I^{n+1}, J^n),(X,A,x_0)].\]
In other words, $\pi_{n+1}(X,A,x_0)$ is set of homotopy classes of maps of triples. A map of triples $f\colon (I^{n+1},\partial I^{n+1},J^n)\to (X,A,x_0)$ means that $f\colon I^{n+1}\to X$ is a map such that $f(\partial I^{n+1})\subset A$ and $f(J^n)\subset \{x_0\}$. A homotopy of maps of triples is a homotopy $H$ such that each $H_t$ is also a map of triples for all $t\in I$. For $n\ge 1$, $\pi_{n+1}$ has a group structure, the same one as defined in \eqref{eq:homotopy_group_operation}. For $n=0$, there is no group structure. 

\begin{lemma}
\label{lemma:cube_horns}
The space $J^n$ is a deformation retraction on $\topint{n+1}$. Given a set of faces $F$ of $\topint{n-1}$ and their union $J_F^{n-1}$, $J_F^{n-1}\subset \topint{n}$ is the inclusion of a deformation retraction. 
\end{lemma}

\begin{proof}
    Both deformation retractions are achieved via straight line homotopies. For example, there are unique straight lines between $a=(\frac{1}{2},\dots ,\frac{1}{2},1)$ and every point in $J^n$. Similarly for $J_F^{n-1}$.
\end{proof}

If $f\colon(X,A,x_0)\to (Y,B,y_0)$ is a map of pointed pairs, i.e., the map preserves the base point, then composition with $f$ induces a map $f_*\colon\pi_n(X,A,x_0)\to \pi_n (Y,B,y_0)$ which is a group homomorphism for $n\ge 2$. In the same fashion, a map of pointed spaces $f\colon(X,x_0)\to (Y,y_0)$ induces a group homomorphism $f_*\colon\pi_n(X,x_0)\to (Y,y_0)$ for $n\ge 1$. Additionaly, $g_*\circ f_*=(g\circ f)_*$ and $\id_*=\id$. These observations show that $\pi_n$ is a functor. Furthermore, $f\simeq g$ implies that $f_*=g_*$ and $\pi_n(X,A,x_0)$ is abelian for $n\ge 3$. 

Given $f \colon (I^{n+1},\partial I^{n+1},J^n)\to (X,A,x_0)$, restricting to $I^n$, identified with $I^n\times\{1\}$ in $I^{n+1}$, we have a map $\partial f\colon (I^n,\partial I^n)\to (A,x_0)$. By considering homotopy classes this induces a \textbf{boundary operator} $\partial\colon\pi_{n+1}(X,A,x_0)\to \pi_n(A,x_0)$. The boundary operator $\partial$ is a group homomorphisms for $n\ge 1$ and for $n=0$ we have $\partial([f])=[f(1)]$. Furthermore, there is a morphism $j_*\colon\pi_n(X,x_0)\to \pi_n(X,A,x_0)$ where we identify $\pi_n(X,x_0)$ with $\pi_n(X,\{x_0\},x_0)$ and define $j_*$ to be the morphism induced by inclusion $j\colon(X,\{x_0\},x_0)\hookrightarrow (X,A,x_0)$.

\begin{theorem}
\label{theorem:homotopy_exact_sequence}
Given a pointed pair of pseudotopological spaces $(X,A,x_0)$, the sequence 
\[\cdots\to \pi_n(A,x_0)\xrightarrow{i_*}\pi_n(X,x_0)\xrightarrow{j_*}\pi_n(X,A,x_0)\xrightarrow{\partial}\cdots \to \pi_1(X,A,x_0)\xrightarrow{\partial}\pi_0(A,x_0)\xrightarrow{i_*}\pi_0(X,x_0),\]
is exact, where $i_*,j_*$ are the maps induced by inclusions $i\colon (A,x_0)\to (X,x_0)$ and $j\colon (X,\{x_0\},x_0)\to (X,A,x_0)$, respectively and $\partial$ is the boundary operator described above.
\end{theorem}

\begin{proof}
    We prove exactness at $\pi_n(A,x_0)$, $\pi_n(X,x_0)$ and $\pi_n(X,A,x_0)$, for all $n$.

    \textbf{Exactness at $\pi_n(A,x_0)$.} We show that $\textnormal{im}(\partial)\subset\ker(i_*)$. Let $f$ be a representative of a homotopy class $[f]\in \pi_{n+1}(X,A,x_0)$, $f\colon (\topint{n+1},\partial\topint{n+1},J^n)\to (X,A,x_0)$. By definition, $\partial f=f|_{\topint{n}\times \{1\}}\colon (\topint{n},\partial \topint{n})\to (A,x_0)$. Then $i_*(\partial f)\colon (\topint{n},\partial \topint{n})\to (X,x_0)$ is homotopic to the constant map, relative to the boundary, as the original map $f$ is such a homotopy.

    We now show that $\ker(i_*)\subset \textnormal{im}(\partial)$. Let $f$ be a representative of a homotopy class $[f]\in \pi_n(A,x_0)$ such that $i_{*}f=f\colon (\topint{n},\partial\topint{n})\to (X,x_0)$ is homotopic to the constant map at $x_0$, $c_{x_0}$. Such a homotopy is a map $H\colon \topint{n}\times I$ satisfying $H_1=f$, $H_0=c_{x_0}$ and $H|_{\partial\topint{n}\times I}=c_{x_0}$. This means that $H$ is a map of triples $H\colon (\topint{n+1},\partial\topint{n+1}, J^n)\to (X,A,x_0)$ and $\partial H=f$, by construction.
 
    \textbf{Exactness at $\pi_n(X,A,x_0)$.} We show that $\textnormal{im}(j_*)\subset \ker(\partial)$. Let $f\colon (\topint{n},\partial\topint{n})\to (X,x_0)$ be a representative of a homotopy class $[f]\in \pi_n(X,x_0)$. The map $j_*f$ is a map of triples $f\colon (\topint{n},\partial\topint{n},J^{n-1})\to (X,A,x_0)$. In particular $j_*f|_{\topint{n-1}\times \{1\}}=c_{x_0}$, the constant map at $x_0$. In other words, $\partial j_*f=c_{x_0}$.

    We now show that $\ker(\partial)\subset \textnormal{im}(j_*)$. Let $f\colon (\topint{n},\partial\topint{n},J^{n-1})\to (X,A,x_0)$ be a representative of a homotopy class $[f]\in \pi_n(X,A,x_0)$ which lies in $\ker(\partial)$. In other words, $\partial f\colon (\topint{n-1}\times\{1\},\partial\topint{n-1}\times\{1\})\to (A,x_0)$ is homotopic to the constant map $c_{x_0}$, at $x_0$, relative to the boundary. Let $H$ be a homotopy relative to the boundary, between $\partial f$ and $c_{x_0}$. Define $H'\colon J^n=\topint{n}\times \{0\}\cup\partial\topint{n}\times I\to X$ by setting $H'=f$ on $\topint{n}\times\{0\}$, the homotopy $H$ on $\topint{n-1}\times \{1\}\times I$ and is equal to the constant map $c_{x_0}$ on the rest of $\partial\topint{n}\times I$. Therefore $H'$ is continuous by \cref{prop:pasting_lemma}. By \cref{lemma:cube_horns} there exists a map $F\colon \topint{n+1}\to X$ which restricts to $H'$ on $J^n$, that is a homotopy of maps of triples $(\topint{n},\partial\topint{n},J^{n-1})\to (X,A,x_0)$ between $f$ and $F_1\colon (\topint{n},\partial\topint{n})\to (X,x_0)$. Hence $[f]=j_*([F_1]$.
 
    \textbf{Exactness at $\pi_n(X,x_0)$.} We show that $\textnormal{im}(i_*)\subset \ker(j_*)$. Let $f\colon(\topint{n},\partial\topint{n})\to (A,x_0)$ be a representative of a homotopy class $[f]\in \pi_n(A,x_0)$. Then $i_*(f)=f\colon (\topint{n},\partial\topint{n})\to (X,x_0)$ is the map $f$, with an expanded codomain. Furthermore, $j_*i_*f=f$ as a map of triples $f\colon (\topint{n},\partial\topint{n},J^{n-1})\to (X,A,x_0)$, but from the original definition of $f$ we have $f(\topint{n})\subset A$. Let $H\colon \topint{n}\times I\to \topint{n}$ be the strong deformation retraction from \cref{lemma:cube_horns} of $\topint{n}$ to $J^{n-1}$. Then $f\circ H_0=f$ and $f\circ H_1=c_{x_0}$, the constant map at $x_0$. Thus $f\circ H$ is a homotopy between $f$ and $c_{x_0}$ in $\pi_n(X,A,x_0)$.

    We now show that $\ker(j_*)\subset \textnormal{im}(i_*) $. Let $f\colon (\topint{n},\partial\topint{n})\to (X,x_0)$ be a representative of a homotopy class $[f]\in \pi_n(X,x_0)$ such that $j_*f=f\colon (\topint{n},\partial\topint{n},J^{n-1})\to (X,A,x_0)$ is null homotopic in $\pi_n(X,A,x_0)$. Let $H\colon \topint{n}\times I$ be a homotopy between $f\colon (\topint{n},\partial\topint{n},x_0)\to (X,A,x_0)$ and the constant map at $x_0$, $c_{x_0}$. This implies that $H(\topint{n}\times\{1\})\subset \{x_0\}\subset A$ and $H(\partial\topint{n}\times I)\subset A$. Furthermore, $I^n\times I$ strongly deformation retracts to $\topint{n}\times\{1\}\cup \partial\topint{n}\times I$, which is a ``reflected" copy of $J^n$, by \cref{lemma:cube_horns}. Composing $H$ with this deformation retraction we get that $f$ is homotopic relative to $\partial\topint{n}$ to a map $g$, whose image is contained in $A$.
\end{proof}

\subsection{Weak homotopy equivalences}
\label{section:weak_homotopy_equivalence}
A map $f\colon X\to Y$ of pseudotopological spaces is called an $n$-\textbf{equivalence}, if for all $x\in X$, $f_*\colon \pi_k(X,x)\to \pi_k(Y,f(x))$ is an isomorphism for all $0\le k<n$ and a surjective homomorphism for $k=n$. A map of pairs $f\colon (X,A)\to (Y,B)$ is called an $n$-\textbf{equivalence} if $(f_*)^{-1}(\textnormal{Im}(\pi_0(B)\to \pi_0(Y))=\textnormal{Im}(\pi_0(A)\to \pi_0(X))$ and for all $a\in A$, $\pi_k(X,A,a)\to (Y,B,f(a))$ is an isomorphism for all $0\le k<n$ and a surjective homomorphism for $k=n$. Note that the condition of path components is automatically satisfied if both $X$ and $Y$ are path connected. In both the absolute and relative cases, $f$ is called a \textbf{weak homotopy equivalence} if $f$ is an $n$-equivalence for all $n$.  We say that the pair $(X,A)$ of pseudotopological spaces is $n$-connected if $\pi_i(X,A)=0$ for all $i\le n$. The following are inspired by the work of \citet{May1990}.

\begin{definition}
\label{def:cofibration}
    A map $i\colon A\to X$ of pseudotopological spaces has the \textbf{homotopy extension property} (HEP) for a pseudotopological space $Y$ if for each homotopy $h\colon A\times I \to Y$ and each map $f\colon X\to Y$ with $fi(a)=h(a,0)$ there exists a homotopy $H\colon X\times I\to Y$ such that $H(x,0)=f(x)$ and $H(i(a),t)=h(a,t)$. The map $i\colon A\to X$ is a \textbf{cofibration} if it has the HEP for all pseudotopological spaces $Y$. The HEP can be summarized with the following diagram:
\begin{center}
\begin{tikzcd}
&X\arrow[dr,"i_0^X"]\arrow[drr,bend left=30,"f"]\\
A\arrow [ur,"i"]\arrow[dr,"i_0^A"] && X\times I\arrow[r,dashed,"H"]& Y\\
& A\times I\arrow[ur,"i\times id"]\arrow[urr,bend right=30,"h"]
\end{tikzcd}
\end{center}
\end{definition}

\begin{lemma}
    \label{lemma:cofibrations_of_compact_hausdorff_spaces}
    If a pair of topological Hausdorff spaces $(X,A)$ satisfies the HEP in $\cat{Top}$, then $(X,A)$ satisfies the HEP in $\cat{PsTop}$.
\end{lemma}

\begin{proof}
    Since $(X,A)$ has the HEP in $\cat{Top}$, there exists a retract $r\colon X\times I\to (X\times\{0\})\cup (A\times I)$. Let $Y$ be a pseudotopological space and let $f\colon A\times I\to Y$ and $g\colon X\times\{0\}\to Y$ be such that $f(a,0)=g(a,0)$ for all $a\in A$. Define $H\colon X\times I\to Y$ by
    \[H(x,t)=\begin{cases}
        f\circ r(x,t),& \forall (x,t)\in X\times I \textnormal{ such that } r(x,t)\in A\times I\\
        g\circ r(x,t),& \forall (x,t)\textnormal{ such that } r(x,t)\in X\times \{0\} 
    \end{cases}.\]
    $H$ is well defined by assumptions on $f$ and $g$. Since $A$ and $X$ are Hausdorff and the inclusion $A\hookrightarrow X$ is a cofibration in $\cat{Top}$, $A$ is closed in $X$. Therefore, $\{A\times I,X\times \{0\}\}$ is a closed cover of $(X\times \{0\})\cup (A\times I)$. Therefore the map $f\cup g\colon (X\times \{0\})\cup (A\times I)\to Y$ is continuous by \cref{prop:pasting_lemma}. Hence as $H=(f\cup g)\circ r$, we have that $H$ is continuous.
\end{proof}

\begin{corollary}
\label{corollary:CW_pairs_HEP}
    Let $X$ be a finite $CW$ complex and let $A$ be a subcomplex. Then the pair $(X,A)$ has the HEP, in $\cat{PsTop}$. 
\end{corollary}

Let $X$ be a pseudotopological space with subspaces $A$ and $B$. We call the triple $(X;A,B)$ a \textbf{triad}. 
If $\{A,B\}$ is an interior cover for $X$, we call $(X;A,B)$ an \textbf{excisive triad}. A \textbf{map of a triads} $f\colon (X;A,B)\to (Y;C,D)$ is a map $f\colon X\to Y$ such that $f(A)\subset C$ and $f(B)\subset D$. The statement and proof of \cref{theorem:n_equivalences} below is borrowed almost verbatim from \new{\citet[Proposition 4K.1]{Hatcher2002}}, however there are a few instances where the arguments might not carry over when going from $\cat{Top}$ to $\cat{PsTop}$. We write out the proof for completeness sake in the Appendix. This way the reader can verify the arguments carry over, albeit with small modifications. 

\begin{theorem}
\label{theorem:n_equivalences}
    Let $f\colon (X;X_1,X_2)\to (Y;Y_1,Y_2)$ be a map of excisive triads in $\cat{PsTop}$ such that $f\colon (X_i,X_1\cap X_2)\to (Y_i,Y_1\cap Y_2)$ is an $n$-equivalence for $i=1,2$. Then $f\colon (X,X_i)\to (Y,Y_i)$ is an $n$-equivalence for $i=1,2$. 
\end{theorem}

\begin{corollary}
\label{corollary:weak_equivalences}
    Let $f\colon(X;X_1,X_2)\to (Y;Y_1,Y_2)$ be a map of excisive triads such that $f:X_1\cap X_2\to Y_1\cap Y_2$, $f\colon X_1\to Y_1$, and $f\colon X_2\to Y_2$ are weak equivalences. Then $f$ is a weak equivalence.
\end{corollary}

\begin{corollary}
\label{corollary:weak_hom_eq_on_inter_cover}
    Let $f\colon X\to Y$ be a map of pseudotopological spaces and let $\mathcal{U}$ be a finite interior cover for $Y$. Let $\mathcal{V}$ be the the closure of $\mathcal{U}$ under finite intersections. That is 
    \[\mathcal{V}=\{V\subset X\,|\,\exists U_1,\dots ,U_n\in \mathcal{U}, V=\bigcap_nU_n\}.\]
    If $f^{-1}(V)\to V$ is a weak equivalence for all $V\in \mathcal{V}$, then $f\colon X\to Y$ is a weak equivalence. 
\end{corollary}

\begin{proof}
    Suppose that $\mathcal{U}=\{U_1,U_2\}$ is an interior cover of $Y$ consisting of two elements. Note that by \cref{theorem:continuity_in_terms_of_nbhds}, $\{f^{-1}(U_1),f^{-1}(U_2)\}$ is an interior cover of $X$. Then the statement follows from \cref{corollary:weak_equivalences}. The general case of finite interior covers can be reduced to the cover by two elements in the following way. Given a finite interior cover $\mathcal{U}=\{U_1,\dots, U_n\}$, we create a new interior cover by two elements by defining $U'_1=U_1$ and $U'_2=\bigcup_{i=2}^nU_i$. 
\end{proof}

\begin{remark}
    In the category of topological spaces, $\cat{Top}$, the statement of \cref{corollary:weak_hom_eq_on_inter_cover} is true for arbitrary open covers \new{\citep[Theorem 6.7.11]{tomDieck2008}}. The argument goes as follows. Since $I^n$ is compact, any map $f:I^n\to X$ or $f:I^n\to Y$ will have compact image and thus we can reduce the case to finite open covers and then apply \cref{theorem:n_equivalences}. However, since we are working with interior covers, in $\cat{Cl}$ or $\cat{PsTop}$, compactness means that we only get a finite subcover of the image  of $f$ (\cref{section:Compactness}), which might not be an interior cover and thus we cannot apply \cref{theorem:n_equivalences} to the subcover.
\end{remark}

\subsection{Singular (co)homology of pseudotopological spaces}
The $n$-\textbf{dimensional standard simplex} is the topological space
\[\simp{n}=\{(t_0,\dots, t_n)\in \mathbb{R}^{n+1}\,|\,\sum_{i=0}^nt_i=1,t_i\ge 0\},\]
with the topology inherited from $\mathbb{R}^{n+1}$. For a set of points $v_0,\dots ,v_m\in \mathbb{R}^k$, we denote by $[v_0,\dots ,v_m]$ the convex hull of $v_0,\dots, v_m$ in $\mathbb{R}^k$. Thus, we have $\simp{n}=[e_0,\dots, e_n]$ where the $e_i$'s  are the standard basis vectors in $\mathbb{R}^{n+1}$.
Let $X$ be a pseudotopological space. A map $\sigma:\simp{n}\to X$ of pseudotopological spaces is called a \textbf{singular} $n$-\textbf{simplex}. We denote by $C_n(X)$ the free abelian group generated by all singular $n$-simplices, that is $C_n(X)=\mathbb{Z}\langle\{\sigma:\simp{n}\to X\,|\, \sigma \textnormal{ is continuous}\}\rangle$. An element $c\in C_n(X)$ is called a \textbf{singular} $n$-\textbf{chain}. The \textbf{singular boundary operator} $\partial_n$ is a group homomorphism $\partial_n\colon C_n(X)\to C_{n-1}(X)$ defined on singular simplices by
\[\partial_n\sigma=\sum_{i=0}^n(-1)^i\sigma|_{[e_0,\dots,\hat{e}_i,\dots,e_n]}\]
where the notation $\hat{e}_i$ means that $e_i$ is excluded. Then, $\partial_n$ is extended linearly to all of $C_n(X)$. If $c\in \ker \partial_n$ we call the chain $c$ a \textbf{cycle}. If $c\in \textnormal{Im}\partial_n$, we call $c$ a \textbf{boundary}. For $n=0$, $\partial_0$ is set to be the zero morphism. A straightforward computation, just like in $\cat{Top}$, shows that $\partial_{n-1}\partial_n=0$ and thus we have a chain complex of abelian groups $(C_{\bullet}(X),\partial_{\bullet})$. We denote the homology groups of this chain complex by $H_n(X)$ and we call them the \textbf{singular homology groups} of $X$. In particular, by definition the abelian group $H_n(X)=\ker\partial_n/\textnormal{Im}\partial_{n+1}$ are the $n$-cycles modulo boundaries.

Let $f\colon X\to Y$ be a map of pseudotopological spaces. If $\sigma:|\Delta^n|\to X$ is a singular $n$-simplex on $X$, then $f\circ\sigma$ is a singular simplex on $Y$. Thus we have an induced chain map $f_{\#}\colon C_n(X)\to C_n(Y)$. One can show this map respects boundary operators and therefore induces a map $f_*\colon H_n(X)\to H_n(Y)$ on homology. 

Let $X$ be a pseudotopological space and let $A$ be a subspace. The chains on $A$ are also chains on $X$, that is there is an inclusion of subgroups $C_n(A)\hookrightarrow C_n(X)$. Let $C_n(X,A)$, be the quotient group $C_n(X)/C_n(A)$. The boundary map $\partial_n\colon C_n(X)\to C_{n-1}(X)$ takes $C_n(A)$ to $C_{n-1}(A)$, hence it induces a quotient boundary map $\partial_n\colon C_n(X,A)\to C_{n-1}(X,A)$. Therefore we have a chain complex of abelian groups $(C_{\bullet}(X,A),\partial_{\bullet})$.
The \textbf{relative singular homology groups} $H_n(X,A)$, are the homology groups of the this chain complex. Standard arguments from homological algebra give the following:

\begin{theorem}
\label{theorem:long_exact_homology}
    Given a pair of pseudotopological spaces $(X,A)$, there exists a long exact sequence of homology groups:
    \[\cdots\to H_n(A)\xrightarrow{i_*} H_n(X)\xrightarrow{j_*} H_{n}(X,A)\xrightarrow{\partial}H_{n-1}(A)\xrightarrow{i_*}H_{n-1}(X)\to\cdots \to H_0(X,A)\to 0\]
where $i\colon C_n(A)\to C_n(X)$ is the inclusion and $j\colon C_n(X)\to C_n(X,A)$ is the quotient map. 
\end{theorem}

Let $G$ be an abelian group. We define $C_n(X;G)$ to be the singular chains on $X$ with coefficients in $G$, that is $C_n(X;G)=C_n(X)\otimes G$. We still get a chain complex $(C_{\bullet}(X;G),\partial_{\bullet})$ whose homology groups we denote by $H_n(X;G)$. We call these the singular homology groups of $X$ with coefficients in $G$.

For a chain complex of abelian groups $(C_{\bullet},\partial_{\bullet})$, let $(C_G^{\bullet},\delta^{\bullet})$ be its associated cochain complex, with coefficients in $G$. Elements of $C_G^n$ are called $n$-\textbf{cochains} value in $G$. More specifically, $C_G^n=\textnormal{Hom}_{\cat{Ab}}(C_n,G)$. The \textbf{coboundary operator} $\delta^{n-1}\colon C_G^{n-1}\to C_G^{n}$ is the dual of the boundary operator; for a cochain $\alpha\in C_{G}^{n-1}$ we have $\delta^{n-1}(\alpha)=\alpha\circ \delta_n$. As $\partial\partial=0$ it then follows that $\delta\delta=0$ and thus $(C_G^{\bullet},\delta^{\bullet})$ is indeed a cochain of groups with coefficients in $G$. The cohomology groups of this cochain complex are called \textbf{singular cohomology groups with coefficients in} $G$ of $X$. We denote these cohomology groups by $H^n(X;G)$.

Applying this construction to singular chain complexes of pseudotopological spaces, we obtain singular cochains of pseudotopological spaces. That is, we define $C^n(X;G)=C^n_G(X)$. For a pair of spaces $(X,A)$,  the short exact sequence 
\[0\to C_n(A)\to C_n(X)\to C_n(X)/C_n(A)\to 0\]
dualizes to the following short exact sequence
\[0\to C_n(X,A;G)\to C_n(X;G)\to C_n(A;G)\to 0\]
where by definition $C_n(X,A;G)=\textnormal{Hom}_{\cat{Ab}}(C_n(X)/C_n(A),G)=\textnormal{Hom}_{\cat{Ab}}(C_n(X,A),G)$. As we have a short exact sequence of cochain complexes, standard arguments in homological algebra give us the following.

\begin{theorem}
\label{theorem:long_exact_cohomology}
    Given a pair of pseudotopological spaces $(X,A)$, there exists a long exact sequence of cohomology groups:
    \[\cdots\to H^n(X,A;G)\to H^n(X;G)\to H^{n}(A;G)\to H^{n+1}(X,A;G)\to H^{n-1}(X;G)\to\cdots.\] 
\end{theorem}

Given a map of pseudotopological spaces $f\colon X\to Y$, composing with $f$ induces a map $f^*\colon H^n(Y;G)\to H^n(X;G)$. The following is an expected result. 

\begin{theorem}
    If $f,g\colon X\to Y$ are homotopic maps of pseudotopological spaces, then $f_*=g_*\colon H_n(X;G)\to H_n(Y;G)$ and $f^*\colon H^n(Y;G)\to H^n(X;G)$ for any $n\ge 0$.
\end{theorem}

\begin{proof}
    The proof is a verbatim replication of a standard one used for topological spaces, for example see \new{\citet[Theorem 4.4.9]{Spanier1966Algebraic}}.
\end{proof}

\begin{theorem}
\label{theorem:weak_homotopy_iso_on_homology}
    A weak homotopy equivalence $f\colon X\to Y$ of pseudotopological spaces induces isomorphisms
    $f_* \colon H_n(X;G)\to H_n(Y;G)$ and $f^*\colon H^n(Y;G)\to H^n(X;G)$ for all $n$ and all coefficient groups $G$.
\end{theorem}

\begin{proof}
    Replace $Y$ by the mapping cylinder $M_f$ and note that from the long exact homotopy, homology and cohomology exact sequences of the pair $(M_f,X)$ (\cref{theorem:homotopy_exact_sequence,theorem:long_exact_homology,theorem:long_exact_cohomology}), it suffices to show that if $(A, B)$ is an $n$-connected pair of path connected pseudotopological spaces, then $H_i(A, B; G) = 0$ and $H^i(A, B; G) = 0$ for all $i \le n$ and all $G$. 
    
    Let $\alpha = \sum_j n_j \sigma_j$ be a relative cycle representative of an element in $H_k(A, B; G)$, for some singular $k$ simplices $\sigma_j \colon |\Delta^k|\to A$. Build a finite $\Delta$ complex $K$ from a disjoint union of $k$ simplices, one for each $\sigma_j$, by identifying all $(k-1)$ dimensional faces of these $k$ simplices for which the corresponding restrictions of the $\sigma_j$'s are equal. This $\Delta$ complex has its usual topological structure by \cref{prop:finite_cw_complexes_are_topological}. 
    Thus the $\sigma_j$'s induce a map $\sigma \colon K\to A$. Since $\alpha$ is a relative cycle, $\partial\alpha$ is a chain in $A$. 
    
    Let $L \subset K$ be the subcomplex consisting of $(k-1)$ simplices corresponding to the singular $(k-1)$ simplices in $\partial\alpha$, that is $\sigma(L) \subset A$. The chain $\alpha$ is the image under the chain map $\sigma_{\#}$ of a chain  $\tilde{\alpha}$ in $K$, with $\partial\tilde{\alpha}$ a chain in $L$. In relative homology we then have $\sigma_*[\tilde{\alpha}]=[\alpha]$. If we assume $\pi_i(A, B) = 0$ for $i \le k$, then $\sigma\colon  (K, L)\to (A, B)$ is homotopic rel $L$ to a map with image in $B$. This fact was shown when arguing for exactness at $\pi_n(X,A,x_0)$ in the proof of \cref{theorem:homotopy_exact_sequence}. Hence $\sigma_*[\tilde{\alpha}]$ is in the image of the map $H_k(A, A; G)\to H_k(A, B; G)$, and since $H_k(A, A; G) = 0$ we conclude that $[\alpha] = \sigma_{*}[\tilde{\alpha}] = 0$. This proves the result for homology, and the result for cohomology then follows by the universal coefficient theorem (the universal coefficient theorem is an entirely algebraic result involving chain complexes and so it doesn't depend on the category of spaces we are working with).
\end{proof}

\section{Main results}
\label{section:main_results}

In this section we show our main result. Recall $\cat{Gph}$ and $\cat{DiGph}$, the categories of graphs and digraphs with graph morphisms, respectively. 

\subsection{Directed Vietoris-Rips complexes}

Let $V=$ be a finite set.  An \textbf{(abstract) simplicial complex $K$} on $V$ is a collection of nonempty subsets of $V$ such that
\begin{enumerate}
\item If $\sigma \in K$ and $\tau \subset \sigma$, then $\tau \in K$.
\item $\{v\} \in K$ for every $v\in V$.
\end{enumerate}
The set $V$ is called the \textbf{vertex set} of $K$. If $\sigma\in K$ with cardinality $n+1$, for some $n\ge 0$,  we say that $\sigma$ is an $n$\textbf{-simplex} of $K$. If $K'$ is any other simplicial complex, a \textbf{simplicial map} $f \colon K \to K'$ is a function from the vertex set of $K$ to the vertex set of $K'$ such that if $\sigma$ is a simplex of $K$, then $f(\sigma)$ is a simplex of $K'$. If $\tau\subset \sigma$ we say that $\tau$ is a \textbf{face} of $\sigma$. 

We say that $L$ is a \textbf{subcomplex} of $K$ if $L\subset K$ and $L$ is a simplicial complex. In addition, if every simplex of $K$ whose vertices are in $L$ is also a simplex of $L$, we say that $L$ is a \textbf{full subcomplex} of $K$. Let $\cat{Simp}$ denote the category of abstract simplicial complexes with morphisms given by simplicial maps. 

Let $K$ be a simplicial complex. We think of $K$ as a category whose objects are the simplices in $K$ and there is a unique morphism $\sigma\to \tau$ if and only if $\sigma\subset \tau$. Pick a total order on the vertex set of $K$, $V$, and define a functor $F\colon K\to \cat{PsTop}$ in the following way. For any $n$-simplex $\sigma\in K$, let $F(\sigma) = \simp{n}$ be the standard n-simplex. The order on the vertex set then specifies a unique bijection between the elements of $\sigma$ and vertices of $\simp{n}$, ordered in the usual way $e_0 < e_1 < \dots < e_n$. If $\tau \subset \sigma$ is a face of $\sigma$, then this bijection specifies a unique $m$-dimensional face of $\simp{n}$. Define $F(\tau\subset \sigma):F(\tau) \to F(\sigma)$ to be the unique affine linear embedding of $\simp{m}$ as that distinguished face of $\simp{n}$, such that the map on vertices is order-preserving. The \textbf{geometric realization} of $K$, $|K|$ is defined as 
\[|K|=\colimit_{\sigma\in K}F(\sigma).\]
In particular $|K|=\coprod_{\sigma\in K}F(\sigma)/\sim$, where if $\tau\subset \sigma$, then for $x\in F(\tau)$ we have $x\sim F(\tau\subset \sigma)(x)\in F(\sigma)$.

Note that if $K$ is finite, its geometric realization $|K|$ is an example of a finite CW complex and is thus a topological space by \cref{prop:finite_cw_complexes_are_topological}. 

\begin{definition}
    For $(X,E)\in \cat{DiGph}$, let $\dVR{X,E}$, the \textbf{directed Vietoris-Rips complex}, be the set (simplicial complex) 
    \[\dVR{X,E}=\{\sigma \subset X\,|\, \exists \textnormal{ an ordering of points in } \sigma,\sigma=\{x_0,\dots, x_n\} \textnormal{ with } i<j\Longrightarrow x_iEx_j\}\] This is also sometimes called the \textbf{directed clique complex} of $(X,E)$ \new{\citep{Masulli2016}}. If $(X,E)$ had no bidirected edges, $\dVR{X,E}$ would correspond to totally ordered subsets of $(X,E)$. Thus $\dVR{X,E}$ is a generalization of the order complex assigned to a partially ordered set. The order complex was shown by  \citet{McCord1966} to be weakly homotopy equivalent to its underlying poset.
    
    For $(X,E)\in \cat{Gph}$ let $\VR{X,E}$ be the set of clique subgraphs of $(X,E)$, that is, \[\VR{X,E}=\{\sigma\subset X\,|\, \forall x,x'\in \sigma, xEx'\}.\] 
    We call $\VR{X,E}$ the \textbf{Vietoris-Rips complex} of $(X,E)$. This is also sometimes called the \textbf{clique complex} of $(X,E)$. 
\end{definition}

Note that $\dVR{X,E}$ and $\VR{X,E}$ coincide when $(X,E)$ is a graph. See \cref{fig:vr_examples} for some examples of the $\dVR{X,E}$ construction.

\begin{figure}[H]
\centering
    \begin{tikzpicture}[scale=0.75]       
        \node[shape=circle,fill=red!60,draw=black,scale=0.375] (A) at (1,-4) {};          
        \node[shape=circle,fill=red!60,draw=black,scale=0.375] (B) at (4,-4) {};
        \node[shape=circle,fill=red!60,draw=black,scale=0.375] (C) at (1,-1) {};
        \node[shape=circle,fill=red!60,draw=black,scale=0.375] (D) at (4,-1) {};         
        \path [->,line width=1](A) edge node[left] {} (B);
        \path [->,line width=1](A) edge node[left] {} (C);
        \path [->,line width=1](A) edge node[left] {} (D);
        \path [->,line width=1](B) edge node[left] {} (D); 
        \path [->,line width=1](C) edge node[left] {} (D);   
    \end{tikzpicture}
    \hspace{2em} 
    \begin{tikzpicture}[scale=0.75]       
        \node[shape=circle,fill=red!60,draw=black,scale=0.375] (A) at (1,-4) {};          
        \node[shape=circle,fill=red!60,draw=black,scale=0.375] (B) at (4,-4) {};
        \node[shape=circle,fill=red!60,draw=black,scale=0.375] (C) at (1,-1) {};
        \node[shape=circle,fill=red!60,draw=black,scale=0.375] (D) at (4,-1) {};         
        \path [->,line width=1](A) edge node[left] {} (B);
        \path [->,line width=1](A) edge node[left] {} (C);
        \path [->,line width=1](D) edge node[left] {} (A);
        \path [->,line width=1](D) edge node[left] {} (B); 
        \path [->,line width=1](C) edge node[left] {} (D);   
    \end{tikzpicture}
    \hspace{2em}
    \begin{tikzpicture}[scale=0.75]       
        \node[shape=circle,fill=red!60,draw=black,scale=0.375] (A) at (1,-4) {};          
        \node[shape=circle,fill=red!60,draw=black,scale=0.375] (B) at (4,-4) {};
        \node[shape=circle,fill=red!60,draw=black,scale=0.375] (C) at (1,-1) {};
        \node[shape=circle,fill=red!60,draw=black,scale=0.375] (D) at (4,-1) {};         
        \path [->,line width=1](A) edge node[left] {} (B);
        \path [->,line width=1](A) edge node[left] {} (C);
        \path [->,line width=1](D) edge node[left] {} (A);
        \path [->,line width=1](B) edge node[left] {} (D); 
        \path [->,line width=1](C) edge node[left] {} (D);   
    \end{tikzpicture}
    \hspace{2em}
        \begin{tikzpicture}[scale=0.75]       
        \node[shape=circle,fill=red!60,draw=black,scale=0.375] (A) at (1,-4) {};          
        \node[shape=circle,fill=red!60,draw=black,scale=0.375] (B) at (4,-4) {};
        \node[shape=circle,fill=red!60,draw=black,scale=0.375] (C) at (1,-1) {};
        \node[shape=circle,fill=red!60,draw=black,scale=0.375] (D) at (4,-1) {};         
        \path [->,line width=1](A) edge node[left] {} (B);
        \path [->,line width=1](A) edge node[left] {} (C);
        \path [->,line width=1](B) edge node[left] {} (D); 
        \path [->,line width=1](C) edge node[left] {} (D);   
    \end{tikzpicture}
    
    \begin{tikzpicture}[scale=0.75]                       
        \draw[fill=green!35, line width=1] (1,-4) rectangle (4,-1);
        \node[shape=circle,fill=red!60,draw=black,scale=0.375] (A') at (1,-4) {};          
        \node[shape=circle,fill=red!60,draw=black,scale=0.375] (B') at (4,-4) {};
        \node[shape=circle,fill=red!60,draw=black,scale=0.375] (C') at (1,-1) {};
        \node[shape=circle,fill=red!60,draw=black,scale=0.375] (D') at (4,-1) {};         
        \path [line width=1](A') edge node[left] {} (B');
        \path [line width=1](C') edge node[left] {} (D');
        \path [line width=1](A') edge node[left] {} (C');
        \path [line width=1](B') edge node[left] {} (D');
        \path [line width=1](A') edge node[left] {} (D'); 
    \end{tikzpicture}
    \hspace{2em} 
     \begin{tikzpicture}[scale=0.75]               

        \filldraw[fill=green!35, draw=black, line width=1] (1,-4) -- (4,-4) -- (4,-1) -- cycle;
        \node[shape=circle,fill=red!60,draw=black,scale=0.375] (A') at (1,-4) {};          
        \node[shape=circle,fill=red!60,draw=black,scale=0.375] (B') at (4,-4) {};
        \node[shape=circle,fill=red!60,draw=black,scale=0.375] (C') at (1,-1) {};
        \node[shape=circle,fill=red!60,draw=black,scale=0.375] (D') at (4,-1) {};         
        \path [line width=1](A') edge node[left] {} (B');
        \path [line width=1](C') edge node[left] {} (D');
        \path [line width=1](A') edge node[left] {} (C');
        \path [line width=1](B') edge node[left] {} (D');
        \path [line width=1](A') edge node[left] {} (D'); 
    \end{tikzpicture}
    \hspace{2em}
    \begin{tikzpicture}[scale=0.75]                       
        \node[shape=circle,fill=red!60,draw=black,scale=0.375] (A') at (1,-4) {};          
        \node[shape=circle,fill=red!60,draw=black,scale=0.375] (B') at (4,-4) {};
        \node[shape=circle,fill=red!60,draw=black,scale=0.375] (C') at (1,-1) {};
        \node[shape=circle,fill=red!60,draw=black,scale=0.375] (D') at (4,-1) {};         
        \path [line width=1](A') edge node[left] {} (B');
        \path [line width=1](C') edge node[left] {} (D');
        \path [line width=1](A') edge node[left] {} (C');
        \path [line width=1](B') edge node[left] {} (D');
        \path [line width=1](A') edge node[left] {} (D'); 
    \end{tikzpicture}
    \hspace{2.3em} 
    \begin{tikzpicture}[scale=0.75]                         \node[shape=circle,fill=red!60,draw=black,scale=0.375] (A') at (1,-4) {};          
        \node[shape=circle,fill=red!60,draw=black,scale=0.375] (B') at (4,-4) {};
        \node[shape=circle,fill=red!60,draw=black,scale=0.375] (C') at (1,-1) {};
        \node[shape=circle,fill=red!60,draw=black,scale=0.375] (D') at (4,-1) {};         
        \path [line width=1](A') edge node[left] {} (B');
        \path [line width=1](C') edge node[left] {} (D');
        \path [line width=1](A') edge node[left] {} (C');
        \path [line width=1](B') edge node[left] {} (D');
    \end{tikzpicture}
\caption{Four different digraphs (top) and the geometric realizations of their respective directed Vietoris-Rips complexes (bottom).}
\label{fig:vr_examples}
\end{figure}

\begin{lemma}
\label{lemma:complex_of_nbhd_is_a_cone}
    Let $(X,E)$ be an digraph (resp. a graph). For each $x\in X$, let $U_x=\{y\in X\,|\, yEx\}$ and let $E_{U_x}$ be the restriction of $E$ to $U_x$.  Then $|\dVR{U_x,E_{U_x}}|$ (resp. $|\VR{U_x, E_{U_x}}|$) is a cone with apex corresponding to $x$ and is therefore contractible. 
\end{lemma}

\begin{proof}
    We prove the directed case, and the undirected case will thus follow. Let $x\in X$, $U_x=\{y\in X\,|\, yEx\}$, and $V_x=U_x\setminus \{x\}$. We prove that $\dVR{U_x,E_{U_x}}$ is a cone with apex $x$ over the simplicial complex $\dVR{V_x,E_{V_x}}$. Note that every simplex $n$-simplex of $\dVR{V_x,E_{V_x}}$ has an ordering of vertices of the form $y_0,\dots, y_n$, where $y_iEy_j$, for $i<j$ in $(X,E)$. Furthermore, since $y_i\in V_x$ for all $i$, we also have that $y_iEx$. Therefore $\{y_0,\dots ,y_n,x\}$ is a simplex in $\dVR{U_x,E_{U_x}}$. Furthermore every simplex of $\dVR{U_x,E_{U_x}}\setminus \dVR{V_x,E_{V_x}}$ is of the form $\{y_0,\dots,y_n,x\}$ where $y_0,\dots, y_n$ is an ordering of vertices of a simplex of $V_x$. Therefore $|\dVR{U_x,E_{U_x}}|$ is  cone with apex at $x$ and therefore contractible.
\end{proof}

\begin{proposition}
    $\dVR{-}\colon \cat{DiGph}\to \cat{Simp}$ and $\VR{-}\colon \cat{Gph}\to \cat{Simp}$ are functors.
\end{proposition}

\begin{proof}
    We prove that $\dVR{-}$ and $\VR{-}$ send morphisms to themselves, that is if $f\colon (X,E)\to (Y,F)$ is a morphism of graphs, the induced morphisms by the functors will be the same set theoretic function $f$. We prove the directed case and the undirected case will follow immediately.

    Let $f\colon (X,E)\to (Y,F)$ be a graph homomorphism of digraphs. Let $\sigma\in \dVR{X,E}$. Then there exists an ordering of points in $\sigma$, $\sigma=\{x_0,\dots ,x_n\}$ such that $x_iEx_j$ for all $i<j$. Since $f$ is a graph morphism, we have $f(x_i)Ff(x_j)$ for all $i<j$ and thus $f(\sigma)$ is a simplex in $\dVR{Y,F}$.
\end{proof}

\begin{lemma}\label{lemma:subgraph_full_subcomplex}
    If $A$ is an induced subgraph of a digraph $X$, then $\dVR{A}$ is a full subcomplex of $\dVR{X}$. If $A$ is an induced subgraph of a graph $X$, then $\VR{A}$ is a full subcomplex of $\VR{X}$.
\end{lemma}

\begin{proof}
    We prove the directed case. The undirected case will then follow. Let $\sigma\in \dVR{X}$ be such that $\sigma\subset A$, i.e., the vertices of $\sigma$ are in $A$. Let $\sigma=\{x_0,\dots ,x_n\}$ be an ordering of elements in $\sigma$ that gave us that $\sigma\in \dVR{X}$. Then $x_iEx_j$ in $X$, for $i<j$. As $A$ is an induced subgraph of $X$, we also have $x_iE_Ax_j$, for each $i<j$. Therefore, $\sigma\in \dVR{A}$. 
\end{proof}

Let $(X,E)$ be a digraph. For each $x\in X$, let $U_x=\{y\,|\, yEx\}.$ Since $E$ is reflexive, $x\in U_x$ so that $x\not \in X\setminus U_x$. Furthermore, every $y\in X$ such that $yEx$ satisfies $y\in U_x$, hence $y\not \in X\setminus U_x$.  It follows that $x\not \in c_E(X\setminus U_x)$, hence $x\in X\setminus c_E(X\setminus U_x)$.  Thus $U_x$ is a neighborhood of $x$, in fact it is the smallest neighborhood of $x$. It follows that $\{U_x\}_{x\in X}$ is an interior cover of $(X,c_E)$.

\begin{lemma}
    If $(X,E)$ is a digraph or a graph and $x\in X$, then $U_x$ (as an induced subgraph) is contractible. 
\end{lemma}

\begin{proof}
    The result follows from \cref{lemma:only_nbhd_whole_space_then_contractible} by setting $U_x=Y$ and $x=a$.
\end{proof}

\subsection{Definition of \texorpdfstring{$f_X$}{f\_X}}
Here we define a map $f_X\colon |\dVR{X,E}|\to (X,E)$ between the geometric realization of the directed Vietoris-Rips complex and the its underlying digraph, that we will later show is a weak homotopy equivalence in $\cat{PsTop}$. Naturally, the map $f_X$ will depend on $E$ and not just $X$ but that will be clear from the definition, however there are choices involved with defining $f_X$ and it is not unique in any way. In order to keep the notation somewhat simple we have opted for $f_X$ and not $f_{(X,E)}$ as the symbol. The relation $E$ will always be clear from context. The same definition of $f_X$ will work for graphs as well. Thus, let $(X,E)$ be a given digraph.

Let $\sigma=\{v_0,\dots , v_n\}$ be an $n$-simplex in $\dVR{X,E}$. Here we are simply enumerating the vertices of $\sigma$ and this enumeration has nothing to do with 
an ordering that was used to determine that $\sigma\in \dVR{X,E}$.  The geometric realization of $\sigma$, $|\sigma|$ is the standard $n$-simplex $|\Delta^n|$ (which is guaranteed by \cref{prop:finite_cw_complexes_are_topological} to indeed be homeomorphic to $|\Delta^n|$). Thus the $v_i$'s can be viewed as a finite collection of points in Euclidean space in general position, yielding a Voronoi diagram subdivision of $\mathbb{R}^{n+1}$. Thus to each $v_i$ there is an assigned $V(v_i)$, its Voronoi cell, which is geometrically a convex polytope of points in $\mathbb{R}^n$ that are closest to $v_i$, or lie on an intersection of hyperplane bisectors among several of $v_i$'s. Let $\sigma_{i}=V(v_i)\cap |\sigma|$, that is $\sigma_i$ is the intersection of $V(v_i)$ with  the $n$-simplex spanned by the $v_i$'s. Since $|\sigma|$ is also convex, each $\sigma_i$ is convex. 

\begin{figure}[H]
\centering
\begin{tikzpicture}[line cap=round,>=latex,line join=round,x=1cm,y=1cm,scale=0.35]
\fill[line width=0pt,color=ffqqqq,fill=ffqqqq,fill opacity=0.5] (-7,-1) -- (-5.25,2.225) -- (-3.5,1.275387596899225) -- (-3.5,-1) -- cycle;
\fill[line width=2pt,color=qqqqff,fill=qqqqff,fill opacity=0.5] (-3.5,-1) -- (0,-1) -- (-1.75,2.225) -- (-3.5,1.275387596899225) -- cycle;
\fill[line width=2.6pt,color=qqqqff,fill=qqqqff,fill opacity=0.5] (-5.25,2.225) -- (-3.5,1.275387596899225) -- (-1.75,2.225) -- (-3.5,5.45) -- cycle;
\draw [line width=2.6pt,color=black] (-1.75,2.225)-- (-3.5,5.45);
\draw [line width=2.6pt,color=black] (-3.5,5.45)-- (-5.25,2.225);
\draw [line width=2.6pt,color=black] (-3.5,-1)-- (0,-1);
\draw [line width=2.6pt,color=black] (0,-1)-- (-1.75,2.225);
\draw [line width=2.6pt,color=ffqqqq] (-7,-1)-- (-5.25,2.225);
\draw [line width=2.6pt,color=ffqqqq] (-3.5,-1)-- (-7,-1);
\draw [line width=2.6pt,color=ffqqqq] (-5.25,2.225)-- (-3.5,1.275387596899225);
\draw [line width=2.6pt,color=ffqqqq] (-3.5,-1)-- (-3.5,1.275387596899225);
\begin{scriptsize}
\draw [fill=ffqqqq] (-7,-1) circle (6.5pt);
\draw [fill=black] (-3.5,5.45) circle (6.5pt);
\draw [fill=black] (0,-1) circle (6.5pt);
\draw [] (-7,-1.5) node[below]{$v_0$};
\draw [] (0,-1.5) node[below]{$v_1$};
\draw [] (-3.5,6.7) node[below]{$v_2$};
\draw [] (-3.5,-2.5) node[below]{$a$};
\end{scriptsize}
\end{tikzpicture}
\quad
\begin{tikzpicture}[line cap=round,>=latex,line join=round,x=1cm,y=1cm,scale=0.35]
\fill[line width=0pt,color=qqqqff,fill=qqqqff,fill opacity=0.5] (-7,-1) -- (-5.25,2.225) -- (-3.5,1.275387596899225) -- (-3.5,-1) -- cycle;
\fill[line width=2pt,color=ffqqqq,fill=ffqqqq,fill opacity=0.5] (-3.5,-1) -- (0,-1) -- (-1.75,2.225) -- (-3.5,1.275387596899225) -- cycle;
\fill[line width=2pt,color=qqqqff,fill=qqqqff,fill opacity=0.5] (-5.25,2.225) -- (-3.5,1.275387596899225) -- (-1.75,2.225) -- (-3.5,5.45) -- cycle;
\draw [line width=2.6pt,color=black] (-7,-1)-- (-5.25,2.225);
\draw [line width=2.6pt,color=black] (-3.5,-1)-- (-7,-1);
\draw [line width=2.6pt,color=ffqqqq] (-3.5,-1)-- (0,-1);
\draw [line width=2.6pt,color=ffqqqq] (0,-1)-- (-1.75,2.225);
\draw [line width=2.6pt,color=ffqqqq] (-3.5,1.275387596899225)-- (-3.5,-1);
\draw [line width=2.6pt,color=black] (-1.75,2.225)-- (-3.5,5.45);
\draw [line width=2.6pt,color=black] (-3.5,5.45)-- (-5.25,2.225);
\draw [line width=2.6pt,color=ffqqqq] (-3.5,1.275387596899225)-- (-1.75,2.225);
\begin{scriptsize}
\draw [fill=black] (-7,-1) circle (6.5pt);
\draw [fill=black] (-3.5,5.45) circle (6.5pt);
\draw [fill=ffqqqq] (0,-1) circle (6.5pt);
\draw [] (-7,-1.5) node[below]{$v_0$};
\draw [] (0,-1.5) node[below]{$v_1$};
\draw [] (-3.5,6.7) node[below]{$v_2$};
\draw [] (-3.5,-2.5) node[below]{$b$};
\end{scriptsize}
\end{tikzpicture}
\quad
\begin{tikzpicture}[line cap=round,>=latex,line join=round,x=1cm,y=1cm,scale=0.35]
\fill[line width=0pt,color=qqqqff,fill=qqqqff,fill opacity=0.5] (-7,-1) -- (-5.25,2.225) -- (-3.5,1.275387596899225) -- (-3.5,-1) -- cycle;
\fill[line width=2pt,color=qqqqff,fill=qqqqff,fill opacity=0.5] (-3.5,-1) -- (0,-1) -- (-1.75,2.225) -- (-3.5,1.275387596899225) -- cycle;
\fill[line width=2pt,color=ffqqqq,fill=ffqqqq,fill opacity=0.5] (-5.25,2.225) -- (-3.5,1.275387596899225) -- (-1.75,2.225) -- (-3.5,5.45) -- cycle;
\draw [line width=2.6pt,color=black] (-7,-1)-- (-5.25,2.225);
\draw [line width=2.6pt,color=black] (-3.5,-1)-- (-7,-1);
\draw [line width=2.6pt,color=black] (-3.5,-1)-- (0,-1);
\draw [line width=2.6pt,color=black] (0,-1)-- (-1.75,2.225);
\draw [line width=2.6pt,color=ffqqqq] (-5.25,2.225)-- (-3.5,1.275387596899225);
\draw [line width=2.6pt,color=ffqqqq] (-3.5,1.275387596899225)-- (-1.75,2.225);
\draw [line width=2.6pt,color=ffqqqq] (-1.75,2.225)-- (-3.5,5.45);
\draw [line width=2.6pt,color=ffqqqq] (-3.5,5.45)-- (-5.25,2.225);
\begin{scriptsize}
\draw [fill=black] (-7,-1) circle (6.5pt);
\draw [fill=ffqqqq] (-3.5,5.45) circle (6.5pt);
\draw [fill=black] (0,-1) circle (6.5pt);
\draw [] (-7,-1.5) node[below]{$v_0$};
\draw [] (0,-1.5) node[below]{$v_1$};
\draw [] (-3.5,6.7) node[below]{$v_2$};
\draw [] (-3.5,-2.5) node[below]{$c$};
\end{scriptsize}
\end{tikzpicture}

\begin{tikzpicture}[line cap=round,>=latex,line join=round,x=1cm,y=1cm,scale=0.35]
\fill[line width=0pt,color=qqqqff,fill=qqqqff,fill opacity=0.5] (-7,-1) -- (-5.25,2.225) -- (-3.5,1.275387596899225) -- (-3.5,-1) -- cycle;
\fill[line width=2pt,color=qqqqff,fill=qqqqff,fill opacity=0.5] (-3.5,-1) -- (0,-1) -- (-1.75,2.225) -- (-3.5,1.275387596899225) -- cycle;
\fill[line width=2.6pt,color=qqqqff,fill=qqqqff,fill opacity=0.5] (-5.25,2.225) -- (-3.5,1.275387596899225) -- (-1.75,2.225) -- (-3.5,5.45) -- cycle;
\draw [line width=2.6pt,color=black] (-1.75,2.225)-- (-3.5,5.45);
\draw [line width=2.6pt,color=black] (-3.5,5.45)-- (-5.25,2.225);
\draw [line width=2.6pt,color=black] (-3.5,-1)-- (0,-1);
\draw [line width=2.6pt,color=black] (0,-1)-- (-1.75,2.225);
\draw [line width=2.6pt,color=black] (-7,-1)-- (-5.25,2.225);
\draw [line width=2.6pt,color=black] (-3.5,-1)-- (-7,-1);
\draw [line width=2.6pt,color=ffqqqq] (-3.5,-1)-- (-3.5,1.275387596899225);
\begin{scriptsize}
\draw [fill=black] (-7,-1) circle (6.5pt);
\draw [fill=black] (-3.5,5.45) circle (6.5pt);
\draw [fill=black] (0,-1) circle (6.5pt);
\draw [] (-7,-1.5) node[below]{$v_0$};
\draw [] (0,-1.5) node[below]{$v_1$};
\draw [] (-3.5,6.7) node[below]{$v_2$};
\draw [] (-3.5,-2.5) node[below]{$d$};
\end{scriptsize}
\end{tikzpicture}
\quad
\begin{tikzpicture}[line cap=round,>=latex,line join=round,x=1cm,y=1cm,scale=0.35]
\fill[line width=0pt,color=qqqqff,fill=qqqqff,fill opacity=0.5] (-7,-1) -- (-5.25,2.225) -- (-3.5,1.275387596899225) -- (-3.5,-1) -- cycle;
\fill[line width=2pt,color=qqqqff,fill=qqqqff,fill opacity=0.5] (-3.5,-1) -- (0,-1) -- (-1.75,2.225) -- (-3.5,1.275387596899225) -- cycle;
\fill[line width=2.6pt,color=qqqqff,fill=qqqqff,fill opacity=0.5] (-5.25,2.225) -- (-3.5,1.275387596899225) -- (-1.75,2.225) -- (-3.5,5.45) -- cycle;
\draw [line width=2.6pt,color=black] (-1.75,2.225)-- (-3.5,5.45);
\draw [line width=2.6pt,color=black] (-3.5,5.45)-- (-5.25,2.225);
\draw [line width=2.6pt,color=black] (-3.5,-1)-- (0,-1);
\draw [line width=2.6pt,color=black] (0,-1)-- (-1.75,2.225);
\draw [line width=2.6pt,color=black] (-7,-1)-- (-5.25,2.225);
\draw [line width=2.6pt,color=black] (-3.5,-1)-- (-7,-1);
\draw [line width=2.6pt,color=ffqqqq] (-5.25,2.225)-- (-3.5,1.275387596899225);
\begin{scriptsize}
\draw [fill=black] (-7,-1) circle (6.5pt);
\draw [fill=black] (-3.5,5.45) circle (6.5pt);
\draw [fill=black] (0,-1) circle (6.5pt);
\draw [] (-7,-1.5) node[below]{$v_0$};
\draw [] (0,-1.5) node[below]{$v_1$};
\draw [] (-3.5,6.7) node[below]{$v_2$};
\draw [] (-3.5,-2.5) node[below]{$e$};
\end{scriptsize}
\end{tikzpicture}
\quad
\begin{tikzpicture}[line cap=round,>=latex,line join=round,x=1cm,y=1cm,scale=0.35]
\fill[line width=0pt,color=qqqqff,fill=qqqqff,fill opacity=0.5] (-7,-1) -- (-5.25,2.225) -- (-3.5,1.275387596899225) -- (-3.5,-1) -- cycle;
\fill[line width=2pt,color=qqqqff,fill=qqqqff,fill opacity=0.5] (-3.5,-1) -- (0,-1) -- (-1.75,2.225) -- (-3.5,1.275387596899225) -- cycle;
\fill[line width=2.6pt,color=qqqqff,fill=qqqqff,fill opacity=0.5] (-5.25,2.225) -- (-3.5,1.275387596899225) -- (-1.75,2.225) -- (-3.5,5.45) -- cycle;
\draw [line width=2.6pt,color=black] (-1.75,2.225)-- (-3.5,5.45);
\draw [line width=2.6pt,color=black] (-3.5,5.45)-- (-5.25,2.225);
\draw [line width=2.6pt,color=black] (-3.5,-1)-- (0,-1);
\draw [line width=2.6pt,color=black] (0,-1)-- (-1.75,2.225);
\draw [line width=2.6pt,color=black] (-7,-1)-- (-5.25,2.225);
\draw [line width=2.6pt,color=black] (-3.5,-1)-- (-7,-1);
\draw [line width=2.6pt,color=ffqqqq] (-1.75,2.225)-- (-3.5,1.275387596899225);
\begin{scriptsize}
\draw [fill=black] (-7,-1) circle (6.5pt);
\draw [fill=black] (-3.5,5.45) circle (6.5pt);
\draw [fill=black] (0,-1) circle (6.5pt);
\draw [] (-7,-1.5) node[below]{$v_0$};
\draw [] (0,-1.5) node[below]{$v_1$};
\draw [] (-3.5,6.7) node[below]{$v_2$};
\draw [] (-3.5,-2.5) node[below]{$f$};
\end{scriptsize}
\end{tikzpicture}
\quad 
\begin{tikzpicture}[line cap=round,>=latex,line join=round,x=1cm,y=1cm,scale=0.35]
\fill[line width=0pt,color=qqqqff,fill=qqqqff,fill opacity=0.5] (-7,-1) -- (-5.25,2.225) -- (-3.5,1.275387596899225) -- (-3.5,-1) -- cycle;
\fill[line width=2pt,color=qqqqff,fill=qqqqff,fill opacity=0.5] (-3.5,-1) -- (0,-1) -- (-1.75,2.225) -- (-3.5,1.275387596899225) -- cycle;
\fill[line width=2.6pt,color=qqqqff,fill=qqqqff,fill opacity=0.5] (-5.25,2.225) -- (-3.5,1.275387596899225) -- (-1.75,2.225) -- (-3.5,5.45) -- cycle;
\draw [line width=2.6pt,color=black] (-1.75,2.225)-- (-3.5,5.45);
\draw [line width=2.6pt,color=black] (-3.5,5.45)-- (-5.25,2.225);
\draw [line width=2.6pt,color=black] (-3.5,-1)-- (0,-1);
\draw [line width=2.6pt,color=black] (0,-1)-- (-1.75,2.225);
\draw [line width=2.6pt,color=black] (-7,-1)-- (-5.25,2.225);
\draw [line width=2.6pt,color=black] (-3.5,-1)-- (-7,-1);
\draw [fill=ffqqqq] (-3.5,1.275387596899225) circle (6.5pt);
\begin{scriptsize}
\draw [fill=black] (-7,-1) circle (6.5pt);
\draw [fill=black] (-3.5,5.45) circle (6.5pt);
\draw [fill=black] (0,-1) circle (6.5pt);
\draw [] (-7,-1.5) node[below]{$v_0$};
\draw [] (0,-1.5) node[below]{$v_1$};
\draw [] (-3.5,6.7) node[below]{$v_2$};
\draw [] (-3.5,-2.5) node[below]{$h$};
\end{scriptsize}
\end{tikzpicture}
\caption{Examples of $\sigma_{\tau}$ (in red) for the two simplex $\sigma=\{v_0,v_1,v_2\}$ for: a) $\tau=\{v_0\}$, b) $\tau=\{v_1\}$, c) $\tau=\{v_2\}$, d) $\tau=\{v_0,v_1\}$, e) $\tau=\{v_0,v_2\}$, f) $\tau=\{v_1,v_2\}$ and h) $\tau=\{v_0,v_1,v_2\}$. }
\label{fig:voronoi_cells_example}
\end{figure}

Let $\textnormal{bar}_{\sigma}$, called the barycenter of $|\sigma|$, be the point in $|\sigma|$ that is equidistant to all the $v_i$'s. If $\sigma$ is a $0$-simplex, say $\sigma=\{v\}$, the barycenter is the point $v$. For $\tau\subset \sigma$, let $\sigma_{\tau}=\bigcap_{i,v_i\in \tau}\sigma_i=\bigcap_{i,v_i\in \tau}\partial\sigma_i$; that is, $\sigma_{\tau}$ are the points in $|\sigma|$ that are equidistant from all the vertices in $\tau$, see \cref{fig:voronoi_cells_example}. In particular, $\sigma_{\sigma}=\{\textnormal{bar}_{\sigma}\}$.

We define $f_X$ by induction on the $n$-skeleta of $|\dVR{X,E}|$. Recall that the $n$-\textbf{skeleton} of $|\dVR{X,E}|$ is the geometric realization of the abstract simplicial subcomplex of $\dVR{X,E}$ consisting of only those $m$-simplices in $\dVR{X,E}$ with $m\le n$. 

Suppose that $f_X$ is defined on the $(n-1)$-skeleton of $|\dVR{X,E}|$, $n\ge 1$. Let $\sigma=\{v_0,\dots, v_n\}$ be an $n$-simplex, where now we have assumed a specific ordering of vertices in $\sigma$ that was used to determine that $\sigma\in \dVR{X,E}$, i.e., we have $v_iEv_j$ for $i<j$ in $(X,E)$. Since $f_X$ is already defined on all $m$-simplices, $m\le n$, we only need to extend $f_X$ on the interior of $|\sigma|$, in $|\dVR{X,E}|$. We define $f_X(x)$, $x\in\textnormal{interior}(|\sigma|)$ by considering two cases:
\begin{enumerate}[left=0pt]
   \item[(1)] If $x$ has a closest point in $|\sigma|$, say $v_i$, that is $x\in \textnormal{interior}(\sigma_i)$, we define $f_X(x)=v_i$.
    \item[(2)] Let $\tau$ be the largest subset of $\sigma$ such that $x\in \sigma_{\tau}\cap\textnormal{interior}(|\sigma|)$. Set $f_X(x)=v_k$ where $k=\sup\{i\,|\, v_i\in \tau\}$. In particular, if $\tau=\sigma$ then $x=\textnormal{bar}_{\sigma}$ and $f_X(x)=v_n$.
\end{enumerate}

Geometrically, we are mapping the points closest to $v_i$ in $|\sigma|$ to $v_i$. When considering points on the boundaries of Voronoi cells, we map the points to the point in $\sigma$ with the largest index that is involved with the boundary. See \cref{fig:1} for a colored illustration, where we have chosen a total ordering of vertices of a digraph and a graph with $3$ elements that is consistent across all the skeleta of the resulting directed Vietoris-Rips complex. 

\begin{figure}[H]
\centering
\begin{tikzpicture}[line cap=round,>=latex,line join=round,x=1cm,y=1cm,scale=0.5]
\clip(-8,-3) rectangle (22.0,7);
\fill[line width=0pt,color=ffqqqq,fill=ffqqqq,fill opacity=0.5] (-7,-1) -- (-5.25,2.225) -- (-3.5,1.275387596899225) -- (-3.5,-1) -- cycle;
\fill[line width=2pt,color=qqqqff,fill=qqqqff,fill opacity=0.5] (-3.5,-1) -- (0,-1) -- (-1.75,2.225) -- (-3.5,1.275387596899225) -- cycle;
\fill[line width=2pt,color=qqccqq,fill=qqccqq,fill opacity=0.5] (-5.25,2.225) -- (-3.5,1.275387596899225) -- (-1.75,2.225) -- (-3.5,5.45) -- cycle;
\draw [line width=2pt] (-5.25,2.225)-- (-3.5,1.275387596899225);
\draw [line width=2pt] (-3.5,1.275387596899225)-- (-3.5,-1);
\draw [line width=2pt] (-3.5,1.275387596899225)-- (-1.75,2.225);
\draw [line width=5.2pt,color=ffqqqq] (-7,-1)-- (-5.25,2.225);
\draw [line width=0.4pt,color=ffqqqq] (-3.5,1.275387596899225)-- (-3.5,-1);
\draw [line width=5.2pt,color=ffqqqq] (-3.5,-1)-- (-7,-1);
\draw [line width=5.2pt,color=qqqqff] (-3.5,-1)-- (0,-1);
\draw [line width=5.2pt,color=qqqqff] (0,-1)-- (-1.75,2.225);
\draw [line width=2pt,color=qqqqff] (-1.75,2.225)-- (-3.5,1.275387596899225);
\draw [line width=5.2pt,color=qqqqff] (-3.5,1.275387596899225)-- (-3.5,-1);
\draw [line width=5.2pt,color=qqccqq] (-5.25,2.225)-- (-3.5,1.275387596899225);
\draw [line width=5.2pt,color=qqccqq] (-3.5,1.275387596899225)-- (-1.75,2.225);
\draw [line width=5.2pt,color=qqccqq] (-1.75,2.225)-- (-3.5,5.45);
\draw [line width=5.2pt,color=qqccqq] (-3.5,5.45)-- (-5.25,2.225);
\draw [->,line width=2pt] (3.5,-1) -- (10.5,-1);
\draw [->,line width=2pt] (3.5,-1) -- (7,5.45);
\draw [->,line width=2pt] (10.5,-1) -- (7,5.45);
\draw [line width=2pt] (14,-1) -- (21,-1);
\draw [line width=2pt] (21,-1) -- (17.5,5.45);
\draw [line width=2pt] (17.5,5.45) -- (14,-1);
\begin{scriptsize}
\draw [fill=ffqqqq] (-7,-1) circle (6.5pt);
\draw [fill=qqccqq] (-3.5,5.45) circle (6.5pt);
\draw [fill=qqqqff] (0,-1) circle (6.5pt);
\draw [fill=ffqqqq] (3.5,-1) circle (6.5pt);
\draw [fill=qqqqff] (10.5,-1) circle (6.5pt);
\draw [fill=qqccqq] (7,5.45) circle (6.5pt);
\draw [fill=ffqqqq] (14,-1) circle (6.5pt);
\draw [fill=qqqqff] (21,-1) circle (6.5pt);
\draw [fill=qqccqq] (17.5,5.45) circle (6.5pt);
\draw [] (-7,-1.5) node[below]{$x_0$};
\draw [] (3.5,-1.5) node[below]{$x_0$};
\draw [] (14,-1.5) node[below]{$x_0$};
\draw [] (0,-1.5) node[below]{$x_1$};
\draw [] (10.5,-1.5) node[below]{$x_1$};
\draw [] (21,-1.5) node[below]{$x_1$};
\draw [] (-3.5,6.7) node[below]{$x_2$};
\draw [] (7,6.7) node[below]{$x_2$};
\draw [] (17.5,6.7) node[below]{$x_2$};
\end{scriptsize}
\end{tikzpicture}
\caption{Example of a map $f_X$ on a $|\Delta^2|=|\dVR{X,E}|=|\VR{X,F}|$ (left) produced from the digraph $(X,E)$ (middle) or the graph $(X,F)$ (right).}
\label{fig:1}
\end{figure}

However, because faces of simplices in $\dVR{X,E}$ could have been considered with different orderings as we move up the $n$-skeleta, the map $f_X$ could send points in lower dimensional faces of $\sigma_{\tau}$ that are not in the interior of $|\sigma|$ in $|\dVR{X,E}|$ to different choices of vertices in $(X,E)$. See \cref{example:continuity_of_f_X}.

\begin{example}
\label{example:continuity_of_f_X}
Let $(X,E)$ be the clique graph on $3$ vertices, $X=\{x,y,z\}$. Then $|\VR{X,E}|$ is a $|\Delta^2|$. Suppose that for the triangle $[x,y,z]$ we have picked the ordering $x=v_0$, $y=v_1$, $z=v_2$, that for the edge $[x,y]$ we picked the ordering $x=w_1$, $y=w_0$, that for the edge $[x,z]$ we picked the ordering $x=u_1$, $z=u_0$ and that for the edge $[y,z]$ we picked the ordering $y=q_1$, $z=q_0$. Then the definition of $f_X$ will be different from the one given in \cref{fig:1} for the same clique graph. See \cref{fig:2} for a colored illustration.
\end{example}

\begin{figure}[H]
\centering
\begin{tikzpicture}[line cap=round,line join=round,>=stealth,x=1cm,y=1cm,scale=0.5]
\clip(-8,-3) rectangle (12.0,7);
\fill[line width=0pt,color=ffqqqq,fill=ffqqqq,fill opacity=0.5] (-7,-1) -- (-5.25,2.225) -- (-3.5,1.275387596899225) -- (-3.5,-1) -- cycle;
\fill[line width=2pt,color=qqqqff,fill=qqqqff,fill opacity=0.5] (-3.5,-1) -- (0,-1) -- (-1.75,2.225) -- (-3.5,1.275387596899225) -- cycle;
\fill[line width=2pt,color=qqccqq,fill=qqccqq,fill opacity=0.5] (-5.25,2.225) -- (-3.5,1.275387596899225) -- (-1.75,2.225) -- (-3.5,5.45) -- cycle;
\draw [line width=2pt] (-5.25,2.225)-- (-3.5,1.275387596899225);
\draw [line width=2pt] (-3.5,1.275387596899225)-- (-3.5,-1);
\draw [line width=2pt] (-3.5,1.275387596899225)-- (-1.75,2.225);
\draw [line width=5.2pt,color=ffqqqq] (-7,-1)-- (-5.25,2.225);
\draw [line width=0.4pt,color=ffqqqq] (-3.5,1.275387596899225)-- (-3.5,-1);
\draw [line width=5.2pt,color=ffqqqq] (-3.5,-1)-- (-7,-1);
\draw [line width=5.2pt,color=qqqqff] (-3.5,-1)-- (0,-1);
\draw [line width=5.2pt,color=qqqqff] (0,-1)-- (-1.75,2.225);
\draw [line width=2pt,color=qqqqff] (-1.75,2.225)-- (-3.5,1.275387596899225);
\draw [line width=5.2pt,color=qqqqff] (-3.5,1.275387596899225)-- (-3.5,-1);
\draw [line width=5.2pt,color=qqccqq] (-5.25,2.225)-- (-3.5,1.275387596899225);
\draw [line width=5.2pt,color=qqccqq] (-3.5,1.275387596899225)-- (-1.75,2.225);
\draw [line width=5.2pt,color=qqccqq] (-1.75,2.225)-- (-3.5,5.45);
\draw [line width=5.2pt,color=qqccqq] (-3.5,5.45)-- (-5.25,2.225);
\draw [line width=2pt] (3.5,-1) -- (10.5,-1);
\draw [line width=2pt] (3.5,-1) -- (7,5.45);
\draw [line width=2pt] (10.5,-1) -- (7,5.45);
\draw [fill=ffqqqq] (-3.5,-1) circle (6.5pt);
\draw [fill=ffqqqq] (-5.25,2.225) circle (6.5pt);
\draw [fill=ffqqqq] (-5.25,2.225) circle (6.5pt);
\draw [fill=qqqqff] (-1.75,2.225) circle (6.5pt);
\begin{scriptsize}
\draw [fill=ffqqqq] (-7,-1) circle (6.5pt);
\draw [fill=qqccqq] (-3.5,5.45) circle (6.5pt);
\draw [fill=qqqqff] (0,-1) circle (6.5pt);
\draw [fill=ffqqqq] (3.5,-1) circle (6.5pt);
\draw [fill=qqqqff] (10.5,-1) circle (6.5pt);
\draw [fill=qqccqq] (7,5.45) circle (6.5pt);
\draw [] (-7,-1.5) node[below]{$x$};
\draw [] (3.5,-1.5) node[below]{$x$};
\draw [] (0,-1.5) node[below]{$y$};
\draw [] (10.5,-1.5) node[below]{$y$};
\draw [] (-3.5,6.7) node[below]{$z$};
\draw [] (7,6.7) node[below]{$z$};
\end{scriptsize}
\end{tikzpicture}
\caption{Example of a map $f_X$ on a $|\Delta^2|=|\VR{X,E}|$ (left) produced from the graph $(X,E)$ (right).}
\label{fig:2}
\end{figure}

We claim that $f_X$ is continuous, regardless of the different choices in defining $f_X$ due to potentially different orderings of vertices considered in different skeleta. Furthermore, as $|\dVR{X,E}|$ is a topological space and we view $(X,E)$ as a closure space, $f_X$ is a morphism in $\cat{Cl}$.

\begin{proposition}
    The function $f_X\colon |\dVR{X,E}|\to (X,E)$ is a continuous map in $\cat{Cl}$.
\end{proposition}
    
\begin{proof}
    We prove by induction on $n$ that $f_X$ is continuous on the $n$-skeleton of $|\dVR{X,E}|$.
    Suppose $n=0$. The $0$-skeleton of $|\dVR{X,E}|$ are the vertices of $X$ with the discrete topology and any set map out of that space is continuous in $\cat{Cl}$. In particular, $f_X$ is continuous on the $0$-skeleton.

    Now suppose that the map $f_X$ is continuous on the $(n-1)$-skeleton, $n\ge 1$. It is sufficient to show that the restrictions of $f_X$ to all $|\sigma|$, for $\sigma$ an $n$-simplex in $\dVR{X,E}$ are continuous, since then $f_X$ will be continuous on the $n$-skeleton by the pasting lemma (\cref{prop:pasting_lemma}).
    Let $\sigma=\{v_0,\dots, v_n\}\in |\dVR{X,E}|$, where $v_iEv_j$, for $i<j$. 
    We show that $f_X$ is continuous at all $x\in |\sigma|$, relying on \cref{theorem:continuity_in_terms_of_nbhds}. For $v\in (X,E)$, its minimal neighborhood, when considering $(X,E)$ as a closure space, is the set $U_v=\{v'\in X\,|\, v'Ev\}$. There are several cases to consider. In Case 1, $x$ has a unique vertex $v_i$ in $\sigma$ that is closest to it and in that case $x$ belongs to $\sigma_i$ and is not on the boundary of any $\sigma_j$'s. The other possibility is that $x$ could be equidistant from multiple vertices in $|\sigma|$. This possibility gives two additional cases. In Case 2, $x$ is in the interior of $|\sigma|$ in $|\dVR{X,E}|$. In Case 3, $x$ belongs to a lower dimensional face of $|\sigma|$. 
    \begin{enumerate}[left=0pt]
        \item Suppose we are in Case 1. Then $x$ is in the interior of $\sigma_i$ in $|\sigma|$ and $f_X(x)=v_i$. Since the interior of $\sigma_i$ in $\sigma$ is open in $|\sigma|$, we can find a small enough open neighborhood of $x$, $W$, such that $W\subset \textnormal{interior}(\sigma_i)$. Thus $f(W)=\{v_i\}\subset U_{v_i}$ and hence $f_X$ is continuous at $x$.
        \item Suppose we are in Case 2. Let $\tau$ be the largest subset of $\sigma$ such that $x\in \sigma_{\tau}$ and $x$ is also in the interior of $|\sigma|$ in $\dVR{X,E}$. Let $v\in \tau\subset \sigma$ be the vertex in the ordering we have chosen in $\sigma$ with the highest index. Thus $f_X(x)=v$. By definition, $\sigma_{\tau}=\bigcap_{i,v_i\in \tau} \partial\sigma_i$. Hence, every neighborhood of $x$ in $\sigma$ will have an intersection with all $\sigma_i$ (and their interiors), for $i$ such that $v_i\in \tau$. Let $W$ be a small enough open neighborhood of $x$, such that $W$ only intersects the $\sigma_i$'s for $v_i\in \tau$, and no other Voronoi cells. Note that if $\tau=\sigma$, then $x=\textnormal{bar}_{\sigma}$ and $W$ will intersect all the Voronoi cells in $\sigma$. Thus, $f_X(W)=\{v_i\,|\,v_i\in \tau\}$. Furthermore, for $i$ such that $v_i\in \tau$, we have $v_iEv$ by the choice of $v$. Thus $\{v_i\in \tau\}\subset U_v$ and therefore $f_X(W)\subset U_v$. It thus follows that $f_X$ is continuous at $x$.
        \item Suppose we are in Case 3. Let $\tau$ be the largest subset of $\sigma$ such that $x\in \sigma_{\tau}$ and $x$ belongs to a lower dimensional face of $|\sigma|$.  Let $k=\sup \{i\,|\,v_i\in \tau\}$ 
        be the highest index of vertices from $\sigma$ that are in $\tau$. Then $v_iEv_k$ for all $0\le i\le n$.
        Assume that $\sigma'\subset \sigma$ is the smallest face of $\sigma$ such that $x\in |\sigma'|$. Note that $\sigma'$ cannot be a single point, as otherwise $x$ would not be equidistant from multiple points in $\sigma$. 
        Since $x\in \sigma_{\tau}$, we have that $\tau\subset \sigma'$. The maximality of $\tau$ in $\sigma$ implies that $\tau$ is also the largest subset of $\sigma'$ such that $x\in \sigma'_{\tau}$. When constructing $\dVR{X,E}$ up to the $(n-1)$-skeleton, the ordering used to determine $\sigma'\in \dVR{X,E}$ and the definition of $f_X$ on $|\sigma'|$ could have been different from the one we have in $\sigma$. Suppose $|\sigma'|=m+1<n$ and let $\sigma'=\{w_0,\dots, w_m\}$ be an ordering of vertices in $\sigma'$ that was used to determine that $\sigma'\in \dVR{X,E}$ and define $f_X$ on $|\sigma'|$. In other words, we have $w_iEw_j$ for all $i<j$. Let $l=\sup\{i\,|\, w_i\in \tau\}$ be the highest index of vertices from $\sigma'$ that are in $\tau$. Then by definition, $f_X(x)=w_l$.
        Since $\sigma_{\tau}=\bigcap_{i,v_i\in \tau}\sigma_i$, any neighborhood of $x$ intersects $\sigma_i$ (and its interior), for all $i$ such that $v_i\in \tau$. Let $W$ be a small enough neighborhood of $x$ in $|\sigma|$ such that $W$ only intersects the $\sigma_i$'s for $v_i\in \tau$ and no other Voronoi cells. From the ordering and labeling of vertices in $\sigma$ and $\sigma'$ and the definition of $\tau$, we have $v_iEw_l$ for all $v_i\in \tau$. Therefore, $f_X(W)=\{v_i\,|\,v_i\in \tau\}\subset U_{w_l}=\{v'\in X\,|\, v'Ew_l\}$. Therefore $f_X$ is continuous at $x$. \qedhere
    \end{enumerate}
\end{proof}

When $(X,E)$ is a graph, we can impose a total order on the vertices of $X$, and enumerate them as $\{x_0,\dots x_N\}$. The following is an immediate consequence of the definition of $f_X$ and the fact we can make the choices of $f_X$ on different skeleta we were making inductively consistent with the chosen total order as every edge in $(X,E)$ is bidirected.

\begin{corollary}
    Let $\sigma=\{v_0,\dots, v_n\}$ be an $n$-simplex in $\VR{X,E}$, where the order is gotten by restricting the total order on $X$.  Then $f_X$ restricted to $|\sigma|$ has can be defined as
    \[f_X(x)=\begin{cases}
        v_i, & x\textnormal{ is in the interior of } \sigma_i\\
        v_j, & j=\sup\{k\,|\, x\in \partial\sigma_k\}
    \end{cases},\]
    and this definition makes $f_X$ a continuous map of closure spaces. 
\end{corollary}

\subsection{\texorpdfstring{$f_X$}{f\_X} is a weak homotopy equivalence}

Here we show that the map $f_X\colon |\dVR{X,E}|\to (X,E)$ is a weak homotopy equivalence in $\cat{PsTop}$. For $A\subset X$, let $(A,E)$ be the induced subgraph of $(X,E)$ on $A$ (We used the notation $(A,E_A)$ for induced subgraphs before, but from now on for simplicity we write $(A,E)$).

\begin{lemma}
\label{lemma:open_start_def_ret}
    Let $(X,E)$ be a digraph and let $A\subset X$. Then $f^{-1}_X(A)$ deformation retracts to  $|\dVR{A,E}|$ in $|\dVR{X,E}|$.
\end{lemma}

\begin{proof} 
By \cref{lemma:subgraph_full_subcomplex}, $\dVR{A,E}$ is a full subcomplex of $\dVR{X,E}$. Let $C=X\setminus A$ be the complement of $A$, and consider the induced subgraph $(C,E)$. Then $\dVR{C,E}$ is the set of simplices in $\dVR{X,E}$ whose geometric realization is disjoint from $|\dVR{A,E}|$. By \cref{lemma:subgraph_full_subcomplex}, $\dVR{C,E}$ is also a full subcomplex of $\dVR{X,E}$. Therefore for $\sigma\in \dVR{X,E}$, if the vertices of $\sigma$ are in $A$, then $\sigma\in \dVR{A,E}$. If the vertices of $\sigma$ are in $C=X\setminus A$, then $\sigma\in \dVR{C,E}$. Otherwise if $\sigma$ has vertices both in $A$ and $C$, then the geometric realization of $\sigma$, is a join $|\sigma|=|\alpha|*|\beta|$ of geometric realizations of a simplex $\alpha\in \dVR{A,E}$ and a simplex $\beta\in \dVR{C,E}$.

We define a retract $r\colon f^{-1}_X(A)\to |\dVR{A,E}|$. Let $a\in f^{-1}_X(A)$. Then $a$ is in the geometric realization of some simplex $\sigma\in \dVR{X,E}$ with the property that $|\sigma|\cap |\dVR{A,E}|\neq \varnothing$.  If $|\sigma|\subset|\dVR{A,E}|$, define $r(a)=a$. Otherwise, the vertices in $\sigma$ lie both in $A$ and $C$. Write $\sigma=\alpha\sqcup \beta$ where the vertices of $\alpha$ are in $\dVR{A,E}$ and the vertices of $\beta$ are in $\dVR{C,E}$. As noted above, $|\sigma|$ is a join of $|\alpha|$ and $|\beta|$.  Since we are in the case where $a\not \in |\alpha|$, $a$ lies on a  unique line segment joining a point of $|\alpha|$ to a point of $|\beta|$. Denote the end point of this line segment that lies in $|\dVR{A,E}|$ by $s_{\sigma}(a)$. By construction of $f_X$, note that $f_X^{-1}(A)\cap |\sigma| =\left(\bigcup_{i\in \alpha}\textnormal{interior}(\sigma_i)\right) \cup N$, where $N=\{x\in |\sigma|\,|\, \exists i\in \alpha, x\in \partial\sigma_i, f_X(x)\in A\}$
and $\sigma_i$ are the Voronoi cells of $|\sigma|$ associated to vertices in $\alpha$. The set $N$ are the points of the boundaries of the Voronoi cells corresponding to vertices from $\alpha$ that are in $f^{-1}_X(A)$. Because $a\in f_X^{-1}(A)\cap |\sigma|$, any point on the line segment between $a$ and $s_{\sigma}(a)$ also lies in $f^{-1}_X(A)\cap |\sigma|$ as we are moving closer and closer to $|\alpha|$, and $\sigma_i$ being the Voronoi cells guarantee that. Define $r(a)=s_{\sigma}(a)$ and observe that $r$ is the desired retract. To show that $r$ is continuous, we apply \new{\citet[Lemma 70.1]{Munkres1984}}.  Write $C=X\setminus A$ with $|C|=|\dVR{C,E}|$ and observe that $r$ is given by the restriction to $f^{-1}_X(A)$ of the retraction $f\colon |\dVR{X,E}|\setminus |\dVR{C,E}|\to |\dVR{A,E}|$ as constructed in the proof of \new{\citet[Lemma 70.1]{Munkres1984}}.  Since $f_X^{-1}(A)\subseteq |\dVR{X,E}|\setminus |\dVR{C,E}|$, the restriction $r$ is continuous. Then the straight line homotopy is a homotopy between identity on $f^{-1}(A)$ and $r$ and thus the result follows.
\end{proof}

\begin{theorem}
\label{theorem:vr_weak_homotopy_equivalence}
    For each finite digraph $(X,E)$ there exist a weak homotopy equivalence $f_X\colon |\dVR{X,E}|\to (X,E)$ in $\mathbf{PsTop}$ ($f_X$ is a morphism in $\cat{Cl}$, but our homotopy theory is defined in $\cat{PsTop}$).
\end{theorem}

\begin{proof}
    We proceed by induction on the cardinality of $X$, $n=|X|$. If $n=1$ the claim is immediate. Suppose the claim is true for any digraph with cardinality $m<n$. Let $(X,E)$ be a digraph with $|X|=n$. Consider the interior cover of $(X,c_E)$, $\mathcal{U}=\{U_x\}_{x\in X}$, where $U_x=\{y\in X\,|\, yEx\}$. If $\exists x\in X$ such that the cardinality of $U_x$ is $n$, $|U_x|=n$, then $U_x=X$. By \cref{lemma:only_nbhd_whole_space_then_contractible} it follows that $(X,c_E)$ is contractible. Furthermore, as $|\dVR{U_x,E}|$ is always a cone by \cref{lemma:complex_of_nbhd_is_a_cone}, and hence contractible, the claim follows. Now suppose that $\forall x\in X$, $|U_x|<n$. Let $\mathcal{U'}$ be the closure of the collection $\mathcal{U}$ under finite intersections, $\mathcal{U'}=\{A\subset X\,|\, \exists U_1,\dots, U_k\in \mathcal{U}, A=\bigcap_{i=1}^k U_i\}$. For each $A\in \mathcal{U}'$, $f_X^{-1}(A)$ deformation retracts to $|\dVR{A,E}|$ in $|\dVR{X,E}|$ by \cref{lemma:open_start_def_ret}. Furthermore, by the induction hypothesis we have that $f_A\colon |\dVR{A,E}|\to (A,E)$ is a weak homotopy equivalence. Thus, by \cref{corollary:weak_hom_eq_on_inter_cover} the result follows.
\end{proof}

\subsection*{Acknowledgements}
\new{The authors would like to thank Antonio Rieser for helpful discussions and suggestions, and we thank the referees for their many helpful comments which improved our paper. This project was initiated during the workshop ``Discrete and combinatorial homotopy theory" at the American Institute for Mathematics, and we thank AIM for the generous support.}

\subsection*{Declarations}

\noindent \textbf{Funding} \new{This work was supported by the NSF under Grant No. DMS-1929284. The second author was supported by an AMS-Simons Research Enhancement Grant for PUI Faculty.}

\noindent \textbf{Conflict of interest} On behalf of all authors, the corresponding author states that there is no conflict of interest.

\appendix
\section{Proofs about interior operators}\label{secA1}

We have the following characterization of continuity of maps of closure spaces in terms of interiors of sets. 

\begin{lemma}
\label{lemma:continuity_characterization_in_terms_of_interiors}
    A function $f\colon (X,c)\to (Y,d)$ between closure spaces is continuous if and only if $f^{-1}(i_d(A))\subset i_c(f^{-1}(A))$ for all $A\subset Y$. 
\end{lemma}

\begin{proof}
    Suppose that $f$ is continuous, that is $f(c(A))\subset d(f(A))$ for all $A\subset X$. Since closure operations on sets are order-preserving relations under $\subset$, we have that for all $A\subset Y$, $c(f^{-1}(A))\subset f^{-1}(d(A))$ is an equivalent characterization of continuity. Let $A\subset Y$. Then
    \begin{align*}
        f^{-1}(i_d(A))&=f^{-1}(Y\setminus d(Y\setminus A))
        =f^{-1}(Y)\setminus f^{-1}(d(Y\setminus A))=\\
        &=X\setminus f^{-1}(d(Y\setminus A))\subset X\setminus c(f^{-1}(Y\setminus A))=\\
        &=X\setminus c(f^{-1}(Y)\setminus f^{-1}(A))
        =X\setminus c(X\setminus f^{-1}(A))= i_c(f^{-1}(A)).
    \end{align*}
    Now suppose that for all $A\subset Y$, $f^{-1}(i_d(A))\subset i_c(f^{-1}(A))$. Let $A\subset Y$. Then 
    \begin{gather*}
        f^{-1}(i_d(Y\setminus A))=f^{-1}(Y\setminus d(A))=f^{-1}(Y)\setminus f^{-1}(d(A))=X\setminus f^{-1}(d(A))\subset i_c(f^{-1}(Y\setminus A)).
    \end{gather*}
    On the other hand
    \begin{gather*}
        i_c(f^{-1}(Y\setminus A))=X\setminus c(X\setminus f^{-1}(Y\setminus A))=X\setminus c(X\setminus (f^{-1}(Y)\setminus f^{-1}(A))=\\
        =X\setminus c(X\setminus (X\setminus f^{-1}(A)))=X\setminus c(f^{-1}(A)).
    \end{gather*}
    From these it follows that $c(f^{-1}(A))\subset f^{-1}(d(A))$ and hence $f$ is continuous.
\end{proof}

Let $(X,\Lambda_X)$ and $(Y,\Lambda_Y)$ be pseudotopological spaces. Let $c_X$ and $c_Y$ be the adherence operators for $\Lambda_X$ and $\Lambda_Y$ respectively. A continuous map $f\colon (X,\Lambda_X)\to (Y,\Lambda_Y)$ is a \textbf{closed map} if for all $A\subset X$ that are closed, i.e., $c_X(A)=A$, we have $c_Y(f(A))=f(A)$, that is $f(A)$ is closed in $(Y,\Lambda_Y)$. Note that this immediately also defines a closed map of closure and topological spaces.

Assuming that the map $i$ in \eqref{cd:pushout} (the definition of a pushout of $\cat{C}$-spaces, \cref{def:pushout}) is an inclusion, we can take $Z=(X\setminus A)\sqcup Y$ in the definition of the pushout of $\cat{C}$-spaces, for $\cat{C}\in \{\cat{Top},\cat{Cl},\cat{PsTop}\}$. We then have the following characterization of the adherence operator on the pushout.

\begin{lemma}
    \label{lemma:pushout_closures_along_closed_inclusions}
    Assume that $i$ in the pushout \eqref{cd:pushout} is a closed inclusion. Then the following are true:
    \begin{enumerate}[left=0pt]
        \item Let $B\subset Y\subset Z$. Then $e(B)=d(B)$.
        \item Let $B\subset X\setminus A\subset Z$. Then $e(B)=g(c(B))$.
        \item Let $B\subset Z$. Then $e(B)=d(B\cap Y)\cup g(c(B\cap (X\setminus A)))$.
        \item $j$ is a closed inclusion.
        \item $i_e(B)=i_{X\setminus A}(B\cap (X\setminus A))\sqcup (i_d(B\cap Y)\setminus f(A\cap c((X\setminus A)\setminus B)))$, where $i_{X\setminus A}$ is the interior operator associated to the subspace closure operator $c_{X\setminus A}$ on $X\setminus A$ induced by the closure operator $c$ for the closure space $(X,c)$. Finally, $i_e$ and $i_d$ are the interior operators for $(Z,e)$ and $(Y,d)$, respectively.
    \end{enumerate}
\end{lemma}

\begin{proof}
    The proofs of $(1)-(4)$ follow almost immediately from definitions and are also available in \new{\citet[Lemma 2.4]{Bubenik2025}} in the case of closure spaces, but they apply verbatim for the adherence operators in question. All that is left is to prove $(5)$.

    By definition we have $Z\setminus i_e(B)=e(Z\setminus B)$. From $(3)$ we have 

\begin{equation*} \label{eq:2}
e(Z\setminus B)=d((Z\setminus B)\cap Y)\cup g(c((Z\setminus B)\cap (X\setminus A))). 
\end{equation*}
Observe that 
    
\begin{equation*} 
(Z\setminus B)\cap Y=((Y\sqcup (X\setminus A))\setminus B)\cap Y=Y\setminus B=Y\setminus (B\cap Y).
\end{equation*}
Similarly 

\begin{equation*} \label{eq:4}
(Z\setminus B)\cap (X\setminus A)=(X\setminus A)\setminus B=(X\setminus A)\setminus (B\cap (X\setminus A)).
\end{equation*}
Furthermore
 \begin{eqnarray*}
g(c((X\setminus A)\setminus B)))&=&g(c((X\setminus A)\setminus B)\cap (X\setminus A))\sqcup g(c((X\setminus A)\setminus B)\cap A)\\
           &=&g(c_{X\setminus A}((X\setminus A)\setminus B))\sqcup g(c((X\setminus A)\setminus B)\cap A)\\
           &=&c_{X\setminus A}((X\setminus A)\setminus B)\sqcup f((c(X\setminus A)\setminus B)\cap A).
 \end{eqnarray*}
Combining all this we get:
 \begin{eqnarray*}
        i_e(B)&=& Z\setminus e(Z\setminus B)\\ 
        &=& Z\setminus [(d((Z\setminus B) \cap Y))\cup g(c((Z\setminus B) \cap (X\setminus A))]\\
        &=& Z\setminus (d(Y\setminus B)\cup g(c (X\setminus A)\setminus B))\\ 
        &=&(Z\setminus d(Y\setminus B))\cap (Z\setminus g(c((X\setminus A)\setminus B)))  \\
        &=&((Y\setminus d(Y\setminus B))\sqcup (X\setminus A))\cap  ((X\setminus A)\setminus c_{X\setminus A}((X\setminus A)\setminus B)\sqcup Y\setminus f(A\cap c(X\setminus A)\setminus B))\\
        &=&(i_d(B\cap Y)\sqcup (X\setminus A))\cap (i_{X\setminus A}(B\cap (X\setminus A))\sqcup Y\setminus f(A\cap c((X\setminus A)\setminus B))) \\
        &=& i_{X\setminus A}(B\cap (X\setminus A))\sqcup (i_d(B\cap Y)\setminus f(A\cap c((X\setminus A)\setminus B))). 
      \end{eqnarray*} 
\end{proof}

Recall the definition of excisive triads and maps of excisive triads (\cref{section:weak_homotopy_equivalence}). 

\begin{example}
\label{example:excisive_triad_mapping_cylinder}
    Let $f\colon(X;A,B)\to (Y;C,D)$ be a map of excisive triads. One might try to replace $Y$ by the mapping cylinder of $f$, and hope to have a map (inclusion) of excisive triads $\Tilde{f}\colon (X;A,B)\to (M_f;M_{f|_{A}},M_{f|_{B}})$ where $M_{f|_{A}}$ and $M_{f|_{B}}$ are the mapping cylinders of restrictions of $f$ to $A$ and $B$ respectively. Clearly, $\Tilde{f}$ is a map of triads $\Tilde{f}\colon (X;A,B)\to (M_f;M_{f|_{A}},M_{f|_{B}})$, however $(M_f;M_{f|_{A}},M_{f|_{B}})$ may not be excisive. As an example, we have the following construction borrowed from \new{\citet[Chapter 4.K]{Hatcher2002}}. Suppose $X=\{A,B\}$ i.e., $X$ has two points $A$ and $B$, with the discrete topology and if $Y=\{C\}=\{D\}$ is a singleton space and $f\colon X\to Y$ is the constant map. Then the interior of $M_{f|_{A}}$ in $M_f$ will not contain the point $\{C\}$ in $M_f$. The same is true for the interior of $M_{f|_{B}}$ and thus $(M_f;M_{f|_{A}},M_{f|_{B}})$ is not excisive. However, one can still show $\Tilde{f}$ is a map of excisive triads, but we need to be more nuanced in selecting the subsets of $M_f$.
\end{example}

The following result shows us how to select interior covers of mapping cylinders that are induced from interior covers of the underlying spaces in order to avoid the issue from \cref{example:excisive_triad_mapping_cylinder}.

\begin{lemma}
\label{lemma:excisive_triad_mapping_cylinder}
    Let $f\colon (X;X_1,X_2)\to (Y;Y_1,Y_2)$ be a map of excisive triads. Let $M_f$ be the mapping cylinder of $f$. Let $M_i=M_{f|_{X_i}}\cup f^{-1}(Y_i)\times (0,\frac{1}{2})$, for $i=1,2$, where $M_{f|_{X_i}}=((X_i\times I)\sqcup Y_i)/((x,0)\sim f(x))$ is the mapping cylinder of $f|_{X_i}\colon X_i\to Y_i$. Then $(M_f;M_1,M_2)$ is an excisive triad and the induced map $\Tilde{f}\colon X\to M_f$ is a map of excisive triads $\Tilde{f}\colon (X;X_1,X_2)\to (M_f;M_1,M_2)$. 
\end{lemma}

\begin{proof}[Proof of \cref{lemma:excisive_triad_mapping_cylinder}]
    Clearly by construction we have $\Tilde{f}(X_i)\subset M_i$ for $i=1,2$. What is left to argue is that $M_1$ and $M_2$ form an interior cover of $M_f$. Let $i_f$ and $c_f$ be the interior and adherence operators of $M_f$, respectively. We rely on \cref{lemma:pushout_closures_along_closed_inclusions} to compute interiors of $M_1$ and $M_2$ in $M_f$. First we explicitly match the notation in the pushouts in \cref{def:pushout} and \cref{def:mapping_cylinder}. That is, we replace $A$ by $X$, $X\setminus A$ by $X\times (0,1]$, $Z$ by $M_f$, and we keep $Y$ as is. Note also that the inclusion $i_0\colon X\hookrightarrow X\times [0,1]$ is closed. Therefore by \cref{lemma:pushout_closures_along_closed_inclusions} we have:
    \[i_f(M_1)=i_{X\times (0,1]}(M_1\cap (X\times (0,1])))\sqcup (i_d(M_1\cap Y)\setminus f((X\times \{0\})\cap c\times \tau ((X\times (0,1])\setminus M_1))),\]
    where $i_{X\times (0,1]}$ is the interior operator for the subspace $X\times (0,1]$ of $X\times I$, and $c\times \tau$ is the adherence for $X\times I$. Note that $M_1\cap Y=Y_1$. Furthermore, 
    \begin{eqnarray*}
        X\times (0,1]\setminus M_1&=&X\times (0,1]\setminus (X_1\times [0,1]\cup f^{-1}(Y_1)\times (0,\frac{1}{2}))\\
        &=&((X\setminus X_1)\times (0,1])\cap ((X\setminus f^{-1}(Y_1)\times (0,\frac{1}{2}))\sqcup (X\times [\frac{1}{2},1]))\\
        &=&(X\setminus f^{-1}(Y_1)\times (0,\frac{1}{2}))\sqcup ((X\setminus X_1)\times [\frac{1}{2},1]),
    \end{eqnarray*}
    since $X\setminus f^{-1}(Y_1)\subset X\setminus X_1$ (as $X_1\subset f^{-1}(Y_1)$). Therefore 
    \begin{eqnarray*}
        c\times \tau((X\times (0,1])\setminus M_1))&=&c(X\setminus f^{-1}(Y_1))\times [0,\frac{1}{2}]\cup c(X\setminus X_1)\times[\frac{1}{2},1]\\
        &=&X\setminus i_c(f^{-1}(Y_1)))\times [0,\frac{1}{2}]\cup c(X\setminus X_1)\times [\frac{1}{2},1].
    \end{eqnarray*}
    Hence 
    \begin{gather*}
        f((X\times \{0\})\cap c\times \tau((X\times (0,1])\setminus M_1))=f(X\setminus i_c(f^{-1}(Y_1)))\subset f(X\setminus f^{-1}(i_d(Y_1)))
    \end{gather*}
    since $f^{-1}(i_d(Y_1))\subset i_c(f^{-1}(Y_1))$ by \cref{lemma:continuity_characterization_in_terms_of_interiors}. Furthermore $i_d(Y_1)\setminus f(X\setminus f^{-1}(i_d(Y_1)))=i_d(Y_1)\subset i_d(Y_1)\setminus f(X\setminus i_c(f^{-1}(Y_1)))$. All of this together gives
    \[i_d(Y_1)\subset i_d(M_1\cap Y)\setminus f((X\times \{0\})\cap c\times \tau ((X\times (0,1])\setminus M_1)).\]

    On the other hand, we have that 
    \begin{gather*}
        c_{X\times (0,1]}((X\times (0,1])\setminus (M_1\cap (X\times (0,1])))=c_{X\times (0,1]}((X\times (0,1])\setminus M_1)\\
        =c_{X\times (0,1]}(((X\setminus X_1)\times [\frac{1}{2},1])\cup ((X\setminus f^{-1}(Y_1)\times (0,\frac{1}{2}))=\\
        =(c(X\setminus X_1)\times [\frac{1}{2},1])\cup (c(X\setminus f^{-1}(Y_1))\times (0,\frac{1}{2}]).
    \end{gather*}
    Therefore 
    \begin{eqnarray*}
        i_{X\times (0,1]}(M_1\cap (X\times (0,1]))&=&(X\times (0,1])\setminus c_{X\times (0,1]}((X\times (0,1])\setminus (M_1\cap (X\times (0,1])))\\
        &=&(X\times (0,1])\setminus ((c(X\setminus X_1)\times [\frac{1}{2},1])
        \\
        &&\cup (c(X\setminus f^{-1}(Y_1))\times (0,\frac{1}{2}]))\\
        &=&((X\setminus c(X\setminus X_1)\times [\frac{1}{2},1])\sqcup (X\times (0,\frac{1}{2})))\\
        &&\cap ((X\setminus c(X\setminus f^{-1}(Y_1))\times (0,\frac{1}{2}])\sqcup(X\times (\frac{1}{2},1])))\\
       &=&((i_c(X_1)\times [\frac{1}{2},1])\sqcup (X\times (0,\frac{1}{2})))\\
       &&\cap ((i_c(f^{-1}(Y_1))\times (0,\frac{1}{2}]))\sqcup (X\times (\frac{1}{2},1)))\\
        &=&(i_c(X_1)\times [\frac{1}{2},1])\sqcup (i_c(f^{-1}(Y_1))\times (0,\frac{1}{2})),
    \end{eqnarray*}
    since $X_1\subset f^{-1}(Y_1)$, implying that $i_c(X_1)\subset i_c(f^{-1}(Y_1))$. For the same reason and the above equality, we have 
    \[i_c(X_1)\times (0,1]\subset i_{X\times (0,1]}(M_1\cap (X\times (0,1])).\]
    In everything so far, we could have replaced $X_1$ and $X_2$ and $Y_1$ by $Y_2$ and gotten analogous statements. Finally, we thus have $(i_c(X_i)\times(0,1])\sqcup i_d(Y_i)\subset i_f(M_i)$, for $i=1,2$. Since $X_1,X_2$ and $Y_1,Y_2$ were interior covers for $X$ and $Y$, respectively, the result follows. 
\end{proof}

\section{Proof of \cref{theorem:n_equivalences}}

Here we prove \cref{theorem:n_equivalences}. Before we proceed, we need to also prove a couple of preliminary results. For convenience sake we will work with the following alternate but equivalent definitions of homotopy groups in what follows. 

Let $D^n$ be the topological $n$-disk and let $S^{n-1}$ be its boundary, $\partial D^n$, i.e. the topological $(n-1)$-sphere. Note that maps $(I^n,\partial I^n)\to (X,x_0)$ in $\cat{PsTop}$ are the same as pointed maps $(I^n/\partial I^n,\partial I^n/\partial I^n)\to (X,x_0)$. Thus by \cref{lemma:homotopy_quotients}, $[(I^n,\partial I^n),(X,x_0)]\cong [I^n/\partial I^n,X]^0$. Note that in $\cat{PsTop}$, $I^n/\partial I^n$ is a topological space by \cref{prop:finite_cw_complexes_are_topological} and thus homeomorphic to $S^n$. Thus we can define 
\[\pi_n(X,x_0)=[(S^n,s_0), (X,x_0)],\] 
where $s_0=\partial I^n/\partial I^n$. The group operation on $\pi_n(X,x_0)$ is then defined as follows. For $[f],[g]\in \pi_n(X,x_0)$, let $[f]+[g]$ be represented by the homotopy class of the composition $S^n\xrightarrow{q}S^n\vee S^n\xrightarrow{f\vee g} X$, where $q$ is a quotient map collapsing the equator in $S^n$ to a point, and $S^n\vee S^n$ is a wedge of spheres, which is topological by \cref{prop:finite_cw_complexes_are_topological}. Similarly, we also define \[\pi_n(X,A,x_0)=[(D^n,S^{n-1},s_0),(X,A,x_0)],\] 
noting that collapsing $J^{n-1}$ to a point converts $(I^n,\partial I^n,J^{n-1})$ to $(D^n,S^{n-1},s_0)$ and this will not affect homotopy classes of maps out of either space, by \cref{lemma:homotopy_quotients}.

\begin{lemma}
\label{lemma:compression_lemma}
    A map $f:(D^n,S^{n-1},s_0)\to (X,A,x_0)$ represents zero in $\pi_n(X,A,x_0)$ if and only if it is homotopic relative to $S^{n-1}$ to a map $D^n\to A$.
\end{lemma}

\begin{proof}
    Suppose $f\simeq g$ (rel $S^{n-1}$), with $g:D^n\to A$. Then $[f]=[g]$ in $\pi_{n}(X,A,x_0)$ and $[g]=0$ via the homotopy obtained by composing $g$ with a deformation retraction of $D^n$ onto $s_0$.

    Conversely, if $[f]=0$ in $\pi_n(X,A,x_0)$ via a homotopy $H:D^n\times I\to X$, then restricting $H$ to a family of $n$-disks in $D^n\times I$ starting with $D^n\times \{0\}$ and ending with the disk $D^n\times \{1\}\cup S^{n-1}\times I$, $\{D^n\times \{t\}\cup S^{n-1}\times [0,t]\}_{t\in I}$, all the disks in the family having the same boundary $S^{n-1}\times \{0\}$, then we get a homotopy from $f$ to a map into $A$, stationary on $S^{n-1}$.
\end{proof}

The following proofs of \cref{lemma:hatcher_lemma,theorem:n_equivalences} are taken almost verbatim from the proof of \new{\citet[Proposition 4K.1]{Hatcher2002}}, with small modifications invoking \cref{lemma:finite_cover_whose_image_is_contained_in_covering_system,prop:pasting_lemma} to make sure the arguments carry over from $\cat{Top}$ to $\cat{PsTop}$.

\begin{lemma}
\label{lemma:hatcher_lemma}
    Let $j:(X,A)\hookrightarrow (Y,B)$ be an inclusion of pairs in $\cat{PsTop}$. The following are equivalent for all $n\ge 1$:
    \begin{enumerate}[left=0pt]
        \item  For all choices of base points the map $j_*:\pi_m(X, A)\to \pi_m(Y,B)$, induced by inclusion, is surjective for $m = n$ and is injective for $m = n - 1$.
        \item Let $\partial D^n$ be written as the union of hemispheres $\partial_+D^n$ and $\partial_-D^n$ intersecting in $S^{n-2}$ at the ``equator". Then every map
        \[(D^n\times \{0\} \cup \partial_+D^n\times I, \partial_-D^n\times \{0\} \cup S^{n-2}\times I) \to (Y,B),\]
        taking $(\partial_+ D^n\times \{1\}, S^{n-2}\times\{1\})$ to $(X, A)$, extends to map $(D^n\times I, \partial_-D^n\times I)\to (Y,B)$ taking $(D^n\times \{1\}, \partial_-D^n\times \{1\})$ to $(X, A)$.
        \item Condition (2) with the added hypothesis that the restriction of the given map to $\partial_+D^n\times I$ is independent of the $I$ coordinate.
    \end{enumerate}
\end{lemma}

\begin{proof}
    Note that (2) and (3) are equivalent since the stronger hypothesis in (3) can
    always be achieved by composing with a homotopy of $D^n\times I$ that shrinks $\partial_+D^n\times I$ to $\partial_+D^n\times \{1\}$. We now show that (3) implies (1) and that (1) implies (2).

    \textbf{(3) implies (1).} Let $f:(\partial_+D^n\times \{1\},S^{n-2}\times\{1\})\to (X,A)$ be a representative of $[f]\in\pi_{n-1}(X,A)$. Suppose that the relative homotopy class of $f$ is in the kernel of $j_*$. We claim there is an extension of $f$ over $D^n\times \{0\}\cup \partial_+D^n\times I$, with the constant homotopy on $\partial_+D^n\times I$ and mapping $(D^n\times\{0\},\partial_-D^n\times \{0\})$ to $(Y,B)$. The extension is obtained in the following way. First we have copies of $f$ along the $I$ coordinate on $\partial_+D^n\times I$. Then, since $[f]=0$ in $\pi_{n-1}(Y,B)$, there is a homotopy $H:\partial_+D^n\times I\to Y$ relative to $S^{n-2}$, to a map $H_1$ whose image is contained in $B$ by \cref{lemma:compression_lemma}. There is an obvious homeomorphism between $\partial_+ D^n\times I$ and $D^n$ where $\partial_+ D^n\times \{1\}$ is identified with $\partial_- D^n$. Therefore $H$ represents the desired extension of $f$ on $D^n\times \{0\}$. The continuity of this extension is guaranteed by \cref{prop:pasting_lemma}. Condition (3) then gives an extension over $D^n\times I$. The restriction of this map to $D^n\times \{1\}$ takes $\partial_-D^n\times\{1\}$ to $A$. Note that every point in $\partial_+D^n$ has a reflected copy in $\partial_- D^n$ and a straight line through $D^n$ connecting them. As $f$ was defined on $\partial_+D^n\times \{1\}$, a straight line homotopy between points on $\partial D^n\times \{1\}$ through $D^n\times \{1\}$ gives a homotopy of $f$ rel $S^{n-2}\times \{1\}$ to a map $\partial_-D^n\times \{1\}\to A$ which represents $0$ in $\pi_{n-1}(X,A)$, by \cref{lemma:compression_lemma}. This shows that $[f]=0$ in $\pi_{n-1}(X,A)$ and hence $j_*$ is injective. To check surjectivity of $j_*$, represent an element of $\pi_n(Y,B)$ by a map $f:D^n\times \{0\}\to Y$ taking $\partial_-D^n\times \{0\}$ to $B$ and $\partial_+D^n\times \{0\}$ to a chosen base point. Extend $f$ over $\partial_+D^n\times I$ via the constant homotopy, then extend over $D^n\times I$ by applying (3). This gives a homotopy of $f$ to a map representing an element of the image of $j_*:\pi_n(X,A)\to \pi_n(Y,B)$. 

    \textbf{(1) implies (2).}  Suppose $f$ as in the hypothesis in (2) is given. Let $f_1$ be the restriction of $f$ to $\partial_+ D^n\times \{1\}$. Then $f_1:(\partial_+D^n\times \{1\},S^{n-2}\times \{1\})\to (X,A)$ is an element of $\pi_{n-1}(X,A)\subset \pi_{n-1}(Y,B)$ as $j_*$ is injective for $m=n-1$ by (1). The restriction of $f$ to $\partial_+D^n\times I$ gives a homotopy between $f_1$ and $f_0$, where $f_0$ is the restriction of $f$ to $\partial_+D^n\times \{0\}$. Note that $f_0:(\partial_+D^n\times\{0\},S^{n-2}\times \{0\})\to (Y,B)$. The restriction of $f$ to $D^n\times\{0\}$ gives a homotopy rel $S^{n-2}\times \{0\}$ between $f_0$ and $g_0$, where $g_0$ is the restriction of $f$ to $\partial_-D^n\times \{0\}$. Note that $g_0$ has image contained in $B$, and thus as $g_0:(\partial_-D^n\times\{0\},S^{n-2}\times \{0\})\to (B,B)$, we have that $[g_0]=0$ in $\pi_{n-1}(Y,B)$ by \cref{lemma:compression_lemma}. Thus, it follows that $[g_0]=[f_0]=[f_1]=0$ in $\pi_{n-1}(Y,B)$ and by injectivity of $j_*$, we have that $[f_1]=0$ in $\pi_{n-1}(X,A)$. Therefore, $f_1$ is homotopic rel $S^{n-2}\times \{1\}$ to a map $g_1$ whose image is contained in $A$. This homotopy can be realized on the disk $D^n\times \{1\}$ such that $g_1$ is defined on $\partial_-D^n\times\{1\}$ with image in $A$. Therefore we can extend $f$ over $D^n\times \{1\}$ such that the extension takes $(D^n\times \{1\}, \partial_-D^n\times \{1\})$ to $(X,A)$. Let  $E^n \subset \partial_-D^n\times I$ be a small disk intersecting $\partial_-D^n\times \{1\}$ in a hemisphere $\partial_+E^n$ of its boundary, and that does not intersect $\partial_-D^n\times \{0\}$. Let $\partial_-E^n\subset \partial_-D^n\times[0,1)$ be the boundary of $E^n$ disjoint from $\partial_-D^n\times \{1\}$. We can assume the extended $f$ has a constant value $x_0 \in A$ on $\partial_+E^n$. Indeed we can precompose the extension of $f$ over $D^n\times\{1\}$ with a map $D^n\to D^n$ that keeps everything fixed, except that it contracts a circular sector to a straight line joining the center of the disk to a point that gets mapped to $x_0$ under $f$. We can match the arc that determines the circular sector with the boundary $\partial_+E^n$. There is a homeomorphism between $D^n$ and $D^n\times \{0,1\}\cup \partial_+D^n\times I$ that identifies $\partial D^n$ with $\partial_-D^n\times \{1\}\cup S^{n-2}\times I\cup \partial_-D^n\times \{0\}$. Therefore, we can consider the extended $f$ as representing an element of $\pi_n(Y , B, x_0)$. Since $j_*$ is surjective for $m=n$, there exists a map $f':(D^n, S^{n-1},s_0)\to (X,A,x_0)$ such that $j_*[f']=[f]$ in $\pi_n(Y,B,x_0)$. Note that we can identify $D^n$ with $E^n$ and thus we can view $f'$ as a map $f':(E^n,\partial_-E^n,\partial_+E^n)\to (X,A,x_0)$. Observe that $D^n\times I$ is a join of $D^n\times \{0,1\}\cup \partial_+D^n\times I$ and $E^n$. The homotopy between $f'$ and $f$, relative the appropriate boundaries, can thus be realized on the straight lines making the join. This homotopy is therefore an extension of $f$ over $D^n\times I$ taking $(E^n, \partial_-E^n)$ to $(X, A)$ and the rest of $\partial_-D^n\times I$ to $B$. Compose this extended $f$ with a deformation of $D^n\times I$ pushing $E^n$ into $D^n\times \{1\}$. This gives the desired result.
\end{proof}

Finally, we prove \cref{theorem:n_equivalences}, which we reprinted below for convenience.

\begin{theorem}[\cref{theorem:n_equivalences}]
    Let $f\colon (X;X_1,X_2)\to (Y;Y_1,Y_2)$ be a map of excisive triads in $\cat{PsTop}$ such that $f\colon (X_i,X_1\cap X_2)\to (Y_i,Y_1\cap Y_2)$ is an $n$-equivalence for $i=1,2$. Then $f\colon (X,X_i)\to (Y,Y_i)$ is an $n$-equivalence for $i=1,2$. 
\end{theorem}

\begin{proof}
    Replacing $Y$ by the mapping cylinder of $f$ with its induced decomposition as an excisive triad, as in \cref{lemma:excisive_triad_mapping_cylinder}, we may assume without loss of generality that $f$ is an inclusion. Recall the conditions (1), (2) and (3) from \cref{lemma:hatcher_lemma}. By \cref{lemma:hatcher_lemma} it suffices to show that condition (2) for the inclusions $(X_1, X_1 \cap X_2)\hookrightarrow (Y_1, Y_1 \cap Y_2)$ and $(X_2, X_1 \cap X_2)\hookrightarrow (Y_2, Y_1 \cap Y_2)$ implies (3) for the inclusion $(X, X_1)\hookrightarrow (Y,Y_1)$. Suppose $f:D^n\times \{0\}\cup\partial_+D^n\times I\to Y$ is a map satisfying the hypothesis in condition (3). For what follows it will be convenient to identify $D^n$ with $I^n$. The identification is such that $\partial_- D^n$ corresponds to the face $I^{n-1}\times \{1\}$, which we denote by $\partial_-I^n$, and $\partial_+D^n$ corresponds to the remaining faces of $I^n$, which we denote by $\partial_+ I^n$. Under these correspondences, we have a map $f$ on $I^n\times \{0\}$ taking $\partial_+ I^n\times \{0\}$ to $X$ and $\partial_-I^n\times \{0\}$ to $Y_1$ , and on $\partial_+I^n\times I$ we have the constant homotopy.
    Since $(Y; Y_1, Y_2)$ is an excisive triad, we can subdivide each of the $I$ factors of $I^n\times \{0\}$ into subintervals so that $f$ takes each of the resulting $n$ dimensional subcubes of $I^n\times \{0\}$ into either $Y_1$ or $Y_2$. This is possible to do by \cref{lemma:finite_cover_whose_image_is_contained_in_covering_system} and the Lebesgue number lemma. The extension of $f$ we construct will have the following property:

    (*) If $K$ is a one of the subcubes of $I^n\times \{0\}$, or a lower dimensional face of such a  
    
    \quad cube, then the extension of $f$ takes $(K\times I, K\times \{1\})$ to $(Y_1, X_1)$ or $(Y_2, X_2)$ whenever 
    
    \quad $f$ takes $K$ to $Y_1$ or $Y_2$, respectively.

    Initially $f$ was defined defined on $\partial_+I^n\times I$ with image in $X$, independent of the $I$ coordinate, and we may assume the condition (*) holds here since we may assume that $X_1 = X \cap Y_1$ and $X_2 = X \cap Y_2$, these conditions holding for the mapping cylinder construction described in \cref{lemma:excisive_triad_mapping_cylinder}. Suppose we now want to extend $f$ over $K\times I$ for $K$ one of the subcubes. We can assume that $f$ has already been extended to $\partial_+ K\times I$ so that (*) is satisfied, by induction on $n$ and on the sequence of subintervals of the last coordinate of $I^n\times \{0\}$. To extend $f$ over $K\times I$, let us first deal with the cases that the given $f$ takes $(K, \partial_- K)$ to $(Y_1,Y_1\cap Y_2)$ or $(Y_2,Y_1\cap Y_2)$. Then by (2) for the inclusion $(X_1, X_1\cap X_2)\hookrightarrow (Y_1, Y_1 \cap Y_2)$ or $(X_2,X_1\cap X_2)\hookrightarrow (Y_2, Y_1 \cap Y_2)$ we may extend $f$ over $K\times I$ so that (*) is still satisfied. If neither of these two cases applies, then the given $f$ takes $(K, \partial_-K)$ just to $(Y_1, Y_1)$ or $(Y_2, Y_2)$, and we can apply (2) to construct the desired extension of $f$ over $K\times I$. 
\end{proof}

\bibliographystyle{plainnat}
\bibliography{References}

\end{document}